\documentclass[a4paper, 12pt]{amsart}
\usepackage{amsmath}
\usepackage{amssymb, latexsym, amsthm}
\usepackage[all]{xy}

\newtheorem{Thm}{Theorem}[section]
\newtheorem{Prop}[Thm]{Proposition}
\newtheorem{Cor}[Thm]{Corollary}
\newtheorem{Lem}[Thm]{Lemma}

\theoremstyle{definition}
\newtheorem*{definition}{Definition}
\newtheorem*{remark}{Remark}

\numberwithin{equation}{section}

\begin{document}

\newcommand{\Coim}{\mathrm{Coim}}

\newcommand{\Q}[0]{\mathbb{Q}}
\newcommand{\F}[0]{\mathbb{F}}
\newcommand{\Z}[0]{\mathbb{Z}}
\newcommand{\N}[0]{\mathbb{N}}
\renewcommand{\O}[0]{\mathcal{O}}
\newcommand{\p}[0]{\mathfrak{p}}
\newcommand{\m}[0]{\mathrm{m}}
\newcommand{\Tr}{\mathrm{Tr}}
\newcommand{\Hom}[0]{\mathrm{Hom}}
\newcommand{\Gal}[0]{\mathrm{Gal}}
\newcommand{\Res}[0]{\mathrm{Res}}
\newcommand{\id}{\mathrm{id}}
\newcommand{\cl}{\mathrm{cl}}
\newcommand{\Li}{\mathrm{Li}}
\newcommand{\wt}{\widetilde}
\newcommand{\mult}{\mathrm{mult}}
\newcommand{\adm}{\mathrm{adm}}
\newcommand{\tr}{\mathrm{tr}}
\newcommand{\pr}{\mathrm{pr}}
\newcommand{\Ker}{\mathrm{Ker}}
\newcommand{\ab}{\mathrm{ab}}
\newcommand{\sep}{\mathrm{sep}}
\newcommand{\f}{\mathrm{f}}
\newcommand{\Md}{\mathrm{Md}}
\newcommand{\triv}{\mathrm{triv}}
\newcommand{\alg}{\mathrm{alg}}
\newcommand{\ur}{\mathrm{ur}}
\newcommand{\Coker}{\mathrm{Coker}}
\newcommand{\Aut}{\mathrm{Aut}}
\newcommand{\Ext}{\mathrm{Ext}}
\newcommand{\Iso}{\mathrm{Iso}}
\newcommand{\M}{\mathcal{M}}
\newcommand{\GL}{\mathrm{GL}}
\newcommand{\Fil}{\mathrm{Fil}}
\newcommand{\Fr}{\mathrm{Fr}}
\newcommand{\an}{\mathrm{an}}
\renewcommand{\c}{\mathcal }
\newcommand{\uL}{\underline{\mathcal L}}
\newcommand{\uM}{\underline{\mathrm{M}\mathrm{\Gamma }}}
\newcommand{\W}{\mathcal W}
\renewcommand{\L}{\mathcal L}
\newcommand{\R}{\mathcal R}
\newcommand{\crys}{\mathrm{crys}}
\newcommand{\st}{\mathrm{st}}
\newcommand{\CM}{\mathrm{CM\Gamma }}
\newcommand{\CV}{\mathcal{C}\mathcal{V}}
\newcommand{\To}{\longrightarrow}
\newcommand{\MF}{\underline{\mathrm{MF}}}

\renewcommand{\c}{\mathcal}
\newcommand{\uC}{\underline{\mathrm{CM}\mathrm{\Gamma }}}
\renewcommand{\L}{\mathcal L}

\title[Varieties with Bad Reduction at 3 Only]
{Projective Varieties with Bad Semi-stable Reduction at 3 Only}
\author{Victor Abrashkin}
\address{Department of Mathematical Sciences, Durham University, Science Laboratories, 
South Rd, Durham DH1 3LE, United Kingdom\ \&\ Steklov Mathematical 
Institute, Gubkina str. 8, 119991, Moscow, Russia}
\email{victor.abrashkin@durham.ac.uk}
\date{July 20, 2010}

\begin{abstract} Suppose $F=W(k)[1/p]$ where $W(k)$ is 
the ring of Witt vectors with coefficients in algebraically 
closed field $k$ of characteristic $p\ne 2$. 
We construct integral theory of $p$-adic semi-stable representations of 
the absolute Galois group of $F$ with Hodge-Tate weights from $[0,p)$. 
This modification of Breuil's theory results in the following application 
in the spirit of Shafarevich's Conjecture. 
If $Y$ is a projective algebraic variety over $\Q $ with good reduction modulo 
all primes $l\ne 3$ and semi-stable reduction modulo 3 
then for the Hodge numbers of 
$Y_{\mathbb C}=Y\otimes _{\Q}\mathbb C$, it holds 
$h^2(Y_{\mathbb C})=h^{1,1}(Y_{\mathbb C})$.
\end{abstract}
\maketitle

\section*{Introduction}

Everywhere in the paper $p$ is a fixed prime number, $p\ne 2$, $k$ 
is algebraically closed field of charactersitic $p$, 
$F$ is the fraction field of the ring of Witt vectors $W(k)$, 
$\bar F$ is a fixed algebraic closure of $F$ and 
$\Gamma _F=\Gal (\bar F/F)$ is the absolute Galois group of $F$. 

Suppose $Y$ is a projective algebraic variety over $\Q $. 
Denote by $Y_{\mathbb C}$ the corresponding complex variety 
$Y\otimes _{\Q }\mathbb C$. For integers $n,m\geqslant 0$, set 
$h^n(Y_{\mathbb C})=\dim _{\mathbb C}H^n(Y_{\mathbb C},\mathbb C)$ and 
$h^{n,m}(Y_{\mathbb C})=\dim _{\mathbb C}H^n(\Omega _{Y_{\mathbb C}}^{m})$.

The main result of this paper can be stated as follows.

\begin{Thm} \label{T0.1} If $Y$ has semi-stable 
reduction modulo $3$ and good reduction 
 modulo all primes $l\ne 3$ then $h^2(Y_{\mathbb C})=h^{1,1}(Y_{\mathbb C})$. 
\end{Thm}

Remind that a generalization of the Shafarevich Conjecture about the non-existence of non-trivial 
abelian varieties over $\Q$ with everywhere good reduction was proved by Fontaine 
\cite{refFo2} 
and the author \cite {refAb2}, and states that 
\begin{equation}\label{E0.1}
h^1(Y_{\mathbb C})=h^3(Y_{\mathbb C})=0,\ \ h^2(Y_{\mathbb C})
=h^{1,1}(Y_{\mathbb C})
\end{equation}
if $Y$ has everywhere good reduction. (The Shafarevich Conjecture 
appears then as the equality 
$h^1(Y_{\mathbb C})=0$.) 
This result became possible due to the following two important achievements 
of Fontaine's theory of $p$-adic crystalline representations:
\medskip 

--- the Fontaine-Messing theorem relating etale and de 
Rham cohomology of smooth proper schemes over $W(k)$ 
in dimensions $[0,p)$, \cite{refFM} (it was later proved by 
Faltings in full generality, \cite{refFa});

--- the Fontaine-Laffaille integral theory of crystalline representations 
of $\Gamma _F$ with Hodge-Tate weights from $[0,p-2]$, \cite{refFL}.

Note that the Fontaine-Laffaille theory works essentially for 
Hodge-Tate weights from  
$[0,p)$ but does not give all Galois invariant lattices in the corresponding 
crystalline representations. Nevertheless, this theory admits  
improvement developed 
by the author in \cite{refAb1}. As a result, there was obtained a suitable 
integral theory for the case of Hodge-Tate  weights from $[0,p)$, which allowed us 
to prove some extras to statements \eqref{E0.1},  in particular, that modulo 
the 
Generalized Riemann Hypothesis it holds 
$h^4(Y_{\mathbb C})=h^{2,2}(Y_{\mathbb C})$.  

Since that time there was a huge progress in the study  
of semi-stable $p$-adic representations. Tsuji \cite{refTsu} proved 
a semi-stable case of the relation between etale and crystalline cohomology and 
Breuil \cite{refBr1, refBr2} developed an analogue of the Fontaine-Laffaille theory 
in the context of semi-stable representations 
(even for ramified basic fields). The papers 
\cite{refB-K} and \cite{refSch} studied the problem 
of the existence of abelian varieties over $\Q$ with only one prime of 
bad semi-stable reduction. Note that the progress in this direction 
is quite restrictive because our knowledge of algebraic number fields with 
prescribed ramification at a given prime number 
$p$ (and unramified outside $p$) is very 
far from to be complete. Theorem \ref{T0.1} represents 
an exceptional situation where the 
standard tools: the Odlyzko estimates of the minimal discriminants of 
algebraic number fields and the modern computing facilities (SAGE) are sufficient to 
resolve upcoming problems. In addition, the proof of 
this theorem requires a modification of Breuil's 
theory to work with semi-stable representations of $\Gamma _F$ with 
Hodge-Tate weights from $[0,p)$.

The structure of this paper can be described as follows. 

In Section \ref{S1} we introduce the category $\uL ^*$ of filtered 
$(\varphi ,N)$-modules over $\c W_1:=k[[u]]$. This is a special 
pre-abelian category, that is an additive category 
with kernels, cokernels 
and sufficiently nice behaving short exact sequences. Note that 
such categories play quite appreciable role in all our constructions. 
In Section \ref{S2}  we construct the functor $\c V^*$ from $\uL ^*$ to 
the category  of $\F _p[\Gamma _F]$-modules $\uM _F$ by introducing a ``truncated'' 
version of Fontaine's ring of semi-stable periods $\hat A_{st}$. 
The functor $\c V^*$ is not fully faithful but by taking 
into account the maximal etale subobjects of 
filtered modules from $\uL ^*$ we define a modification 
$\CV^*$ of $\c V^*$. This functor gives already a fully faithful functor 
from $\uL^*$ to the category of cofiltered  $\Gamma _F$-modules $\uC _F$. 
In Section \ref{S3} we give an interpretation of Breuil's theory in terms of 
$\c W:=W(k)[[u]]$-modules 
(Breuil worked with modules over the divided 
powers envelope of $\c W$) 
by introducing the  
category  of filtered $(\varphi ,N)$-modules $\uL ^{ft}$
over $\c W$. The advantage of this construction is that the 
objects of this category appear as strict subquotients of 
$p$-divisible groups in suitable pre-abelian category. This allows 
us to use 
devissage despite that all involved categories are not abelian. 
We also introduce the subcategories $\uL ^{u, ft}$ and, resp., $\uL ^{m, ft}$ of 
unipotent and, resp., multiplicative objects in $\uL ^{ft}$ and prove that 
any $\L\in\uL ^{ft}$ is a canonical extension 
\begin{equation} \label{E0.2}
 0\longrightarrow \L ^u
\longrightarrow\L\longrightarrow\L ^m\longrightarrow 0
\end{equation}
of a multiplicative object $\L ^m$ by a unipotent object $\L ^u$. 
In Section \ref{S4} we study Breuil's functor 
$\c V^{ft}:\uL ^{ft}\longrightarrow\uM _F$ in the situation of 
Hodge-Tate weights from $[0,p)$. We show that on the subcategory 
$\uL ^{u, ft}$  this functor is still 
fully faithful by proving that on the subcategory of 
killed by $p$ unipotent objects the functors $\c V^{ft}$ and $\c V^*$ 
coincide. 
Then we show that for any killed by $p$ object $\L$ of $\uL ^{ft}$, 
the functor $\c V^{ft}$ transforms the standard short exact 
sequence \eqref{E0.2} into a short exact sequence in $\uM _F$, which admits a 
functorial splitting. This splitting is used then to construct 
a modified version $\widetilde{\CV}^{ft} :
\uL ^{ft}\longrightarrow\uC _F$ of $\c V ^{ft}$, which is already 
fully faithful. This gives us an efficient control on all 
Galois invariant lattices of semi-stable representations with weights from 
$[0,p)$. Especially, we have an explicit description of all killed by $p$ 
subquotients of such lattices and the corresponding ramification estimates.
Finally, in Section \ref{S5} we give a proof of 
Theorem \ref{T0.1} following the strategy from \cite{refAb2}. 

Essentially, we obtain 
the following result: if $V$ is a 3-adic representation of 
$\Gamma _{\Q}=\Gal (\bar\Q /\Q)$ which is unramified outside 
3 and is semi-stable at 3 
then there is a $\Gamma _{\Q}$-equivariant filtration by $\Q_3$-subspaces 
$V=V_0\supset V_1\supset V_2\supset V_3=0$ such that for $0\leqslant i\leqslant 2$, 
the $\Gamma _{\Q}$-module $V_{i}/V_{i+1}$ is isomorphic to the product of finitely 
many copies of the Tate twist $\Q _3(i)$. If $V=H^2_{et}(Y_{F},\Q_3)$ then 
looking at the eigenvalues of the Frobenius morphisms of reductions modulo $l\ne 3$,  
we obtain that $V=V_1$ and $V_2=0$, and this implies that 
$h^2(Y_{\mathbb C})=h^{1,1}(Y_{\mathbb C})$. 

Note that our construction of the modification of Breuil's 
functor gives automatically 
the modification of the Fontaine-Laffaille functor, 
which essentially coincides with  
the modification constructed in \cite{refAb1}. It is worth mentioning that 
switching from Breuil's $S$-modules to $\c W$-modules means moving in the direction 
of Kisin's approach \cite{refKis} and recent approach to integral 
theory of $p$-adic representations by Liu \cite{Li1,Li2}. It would be also 
very interesting to study the opportunity to modify  
Breuil's functor over ramified base \cite{refCa1, refCa2} to the case of 
Hodge-Tate weights from $[0,p)$. 
Finally, mention quite surprising matching of the ramification estimates 
for semi-stable representations and the Leopoldt conjecture for the field 
$\Q (\root 3\of 3,\zeta _{9})$, cf. Section \ref{S5}. 

{\bf Acknowledgements.} The author is very grateful to  
Shin Hattori for numerous and helpful discussions.

\newpage

\section{The categories $\widetilde{\uL} ^*$, $\widetilde{\uL}_0 ^*$, 
$\uL ^*$, $\uL _0^*$ }
\label{S1}

Remind that $k$ is algebraically closed field of characteristic $p>2$. 
Let $\c W=W(k)[[u]]$, where $W(k)$ is the ring of Witt vectors with coefficients 
in $k$ and $u$ is an indeterminate. Denote by $\sigma $ 
the automorphism of $W(k)$ induced by the $p$-th power map on $k$ and 
agree to use the same symbol for its continuous extension 
to $\c W$ such that $\sigma (u)=u^p$. Denote by 
$N:\c W\longrightarrow\c W$ the continuous $W(k)$-linear derivation such
that $N(u)=-u$.

We shall often use below the following statement. 

\begin{Lem} \label{L1.1} Suppose $L$ is a finitely generated  
$\W $-module and $\c A$ is a 
$\sigma $-linear operator on $L$. Then the operator $\id _L-\c A$ is epimorphic. 
If, in addition, $\c A$ is nilpotent then $\id _L-\c A$ is bijective. 
\end{Lem} 

\begin{proof} Part b) is obvious. 
In order to prove a) notice first that we can 
replace $L$ by $L/uL$ and, therefore, assume that 
$L$ is a finitely generated  $W(k)$-module. 
Clearly, it will be enough to consider the case $pL=0$. 
Then there is a decomposition of $k$-vector spaces $L=L_1\oplus L_2$, where $\c A$ 
is invertible on $L_1$ and nilpotent on $L_2$. 
It remains to note that $L_1=L_{0}\otimes _{\F _p}k$, where 
$L_0$ is a finite dimensional $\F _p$-vector space such that $\c A|_{L_0}=\id $. 
The existence of $L_0$ is a standard fact of $\sigma $-linear algebra: 
if $s=\dim _kL_1$ and $A\in M_s(k)$ is a matrix of $\c A|_{L_1}$ 
in some $k$-basis of $L_1$ then  
$L_0=\{(x_1,\dots ,x_s)\in k^s\ |\ 
(x_1^p,\dots ,x_s^p)A=(x_1,\dots ,x_s)\}$; the $\F _p$-linear space 
$L_0$  has dimension $s$ because 
the corresponding equations determine an etale algebra of rank $p^s$ 
over algebraically closed field $k$.
 
\end{proof}
\medskip 

\subsection{Definitions and general properties} 
\label{S1.1}

Let $\W _1=\W/p\W$ with induced $\sigma $, $\varphi $ and $N$.

\begin{definition} The objects of the category 
$\widetilde{\uL} _0^*$   are the 
triples \linebreak 
$\L=(L, F(L), \varphi )$, where 

$\bullet $\  $L$ and $F(L)$ are $\W _1$-modules such that 
$L\supset F(L)$;

$\bullet $ \  $\varphi :F(L)\longrightarrow L$ is a $\sigma $-linear
morphism of $\W _1$-modules; (Note that 
$\varphi (F(L))$ is a $\sigma (\c W_1)$-submodule in $L$.)

If $\L _1=(L_1,F(L_1),\varphi )$ is also an object of 
$\widetilde{\uL} _0^*$ then the morphisms 
$f\in\Hom _{\widetilde{\uL} _0^*}(\L _1,\L )$ are 
given by $\W _1$-linear maps $f:L_1\longrightarrow L$ such that 
$f(F(L_1))\subset F(L)$ and $f\varphi =\varphi f$. 
\end{definition} 

\begin{definition} The objects of the category 
$\widetilde{\uL} ^*$   are the 
quadruples $\L=(L, F(L), \varphi , N)$, where 

$\bullet $\  $(L,F(L),\varphi )$ is an object of the category $\widetilde{\uL} ^*_0$;

$\bullet $\ $N:L\longrightarrow L/u^{2p}L$  is a $\c W_1$-differentiation, i.e. 
for all $w\in\W _1$ and $l\in L$, 
$N(wl)=N(w)(l\operatorname{mod}u^{2p}L)+wN(l)$;
 
$\bullet $\ if $\L _1=(L_1,F(L_1),\varphi ,N)$ is another object of 
$\widetilde{\uL} ^*$ then  the morphisms  
$\Hom _{\widetilde{\uL} ^*}(\L _1,\L )$ are 
given by $f:(L_1, F(L_1),\varphi )\To (L,F(L),\varphi )$ 
from  $\widetilde{\uL }^*_0$ such that 
$fN=Nf$. (We use the same notation $f$ for 
the reduction of $f$ modulo $u^{2p}L$.)
\end{definition} 

The categories $\widetilde{\uL} ^*$ and $\widetilde{\uL }^*_0$ are   
additive.

\begin{definition} The category $\uL _0^*$ is a 
full subcategory of $\widetilde{\uL }_0^*$  
consisting of the objects $\L =(L,F(L),\varphi )$ such that 

$\bullet $\ $L$ is a free $\W _1$-module of finite rank;

$\bullet $\ $F(L)\supset u^{p-1}L$;

$\bullet $\ the natural embedding $\varphi (F(L))\subset L$ induces the identification 
$\varphi (F(L))\otimes _{\sigma (\c W_1)}\c W_1=L$.  
\end{definition}

Note that $\varphi $ induces a map $F(L)/u^{2p}L\To L/u^{2p}L$: use that 
$u^{2p}L\subset u^{p+1}F(L)\subset u^2F(L)$ and $\varphi (u^2F(L))\subset u^{2p}L$. 
We shall 
denote this map by the same symbol $\varphi $. 

\begin{definition} The category $\uL ^*$ is a 
full subcategory of $\widetilde{\uL }^*$  
consisting of the objects $\L =(L,F(L),\varphi ,N)$ such that 

$\bullet $\ $(L,F(L),\varphi )\in\uL ^*_0$;

$\bullet $\ for all $l\in F(L)$, $uN(l)\in F(L)\operatorname{mod}u^{2p}L$ and 
$N(\varphi (l))=\varphi (uN(l))$.   
\end{definition}

The categories $\uL ^*_0$ and $\uL ^*$ are additive. 

In the case of objects $(L,F(L),\varphi ,N)$ of 
$\uL ^*$ the 
morphism $N$ can be uniquely recovered from  
the $\c W_1$-differentiation $N_1=N\operatorname{mod}u^pL$ 
due to the following property. 

\begin{Prop} \label{P1.2} Suppose $(L,F(L),\varphi )\in\uL ^*_0$ and 
$N_1:L\mapsto L/u^pL$ is a $\c W_1$-differentiation such that for any 
$m\in F(L)$, $uN_1(m)\in F(L)\operatorname{mod}u^pL$ and 
$N_1(\varphi (l))=\varphi (uN_1(l))$. Then there is a unique 
$\c W_1$-differentiation $N:L\To L/u^{2p}L$ such 
that $N\operatorname{mod}u^p=N_1$ and $(L,F(L),\varphi ,N)\in\uL ^*$. 
\end{Prop} 

\begin{proof}
 Choose a $\c W_1$-basis $m_1,\dots ,m_s$ of $F(L)$. Then $l_1=\varphi (m_1)$, 
\dots , $l_s=\varphi (m_s)$ is a $\c W_1$-basis of $L$ and a 
$\sigma (\c W_1)$-basis of $\varphi (F(L))$. 

Let $N(l_i):=\varphi (uN_1(m_i)')\in L/u^{2p}L$, 
where $N_1(m_i)'$ are some lifts of 
$N_1(m_i)$ to $L/u^{2p}L$. Clearly, 
the elements $N(l_i)\in \varphi (F(L))\subset L/u^{2p}L$ are well-defined
 (use that $\varphi (u^{p+1}L)\subset u^{2p}L$). 

For any $l=\sum _iw_il_i\in L$, let 
$N(l):=\sum  _iN(w_i)l_i+\sum _iw_iN(l_i)$. Then $N:L\To L/u^{2p}$ is 
a $\c W_1$-differentiation and 
$N\operatorname{mod}u^p=N_1$. Clearly, $N$ is the only candidate to satisfy 
the requirements of our Proposition. 

Now suppose $m=\sum _iw_im_i\in F(L)$ with all $w_i\in \c W_1$. Then 
$N(\varphi (m))=\sum _i\sigma (w_i)l_i\operatorname{mod}u^{2p}$. 
On the other hand, 
$\varphi (uN(m))$ equals 
$$\sum _iu^p\sigma (N(w_i))l_i+
\sum _i\varphi (w_iuN(m_i))=\sum _i\sigma (w_i)l_i\operatorname{mod}u^{2p}$$
because all $\sigma (N(w_i))\in u^p\sigma (\c W_1)$. 

The proposition is proved. 
\end{proof}

\begin{remark}
 By above Proposition in the definition of objects of $\uL ^*$ one can replace 
$N:L\To L/u^{2p}L$ by 
$N_1=N\operatorname{mod}u^pL$ and use  $N$ as a unique extension of $N_1$ if 
neccessary. An example of the situation where we do need 
to extend $N_1$ is described in Proposition \ref{P1.3} below. Another situation 
is related to the definition of the truncated version 
$R^0_{st}$ of $\hat A_{st}$ in Subsection \ref{S2}. Here 
we need $N$ to be defined modulo some smaller module than $u^pL$, e.g. $u^{p+1}L$. 
Our choice was done in favour of the module $u^{2p}L$ because 
it is the smallest possible 
module where the definition of $N$ makes sense. 
\end{remark}

\begin{Prop} \label{P1.3} $\uL _0^*$ and $\uL ^*$ are 
pre-abelian categories (cf. Appendix A for the concept of pre-abelian category). 
\end{Prop} 

\begin{proof} 
Suppose $\c S$ is an additive category and $f\in\Hom _{\c S}(A,B)$, $A,B\in\c S$.
Then $i\in\Hom _{\c S}(K,A)$ is a kernel of $f$ if for any 
$D\in\c S$, the sequence of abelian groups  
$$0\To \Hom _{\c S}(D,K)\overset{i_*}\To 
\Hom _{\c S}(D,A)\overset{f_*}\To \Hom _{\c S}(D,B)$$
is exact.  
Similarly, $j\in\Hom _{\c S}(B,C)$, $B,C\in\c S$,  is a cokernel of $f$ if for any 
$D\in\c S$, the sequence 
$$0\To\Hom _{\c S}(C,D)\overset{j^*}\To \Hom _{\c S}(B,D)
\overset{f^*}\To\Hom _{\c S}(A,D)$$
is exact. 

Let $FF_{\c W_1}$ be the category of free $\c W_1$-modules with filtration. 
This category is 
pre-abelian. More precisely, consider the objects  
$\c L=(L,F(L))$ and $\c M=(M,F(M))$  
in $FF_{\c W_1}$ and let $f\in\Hom _{FF_{\c W_1}}(\c L,\c M)$. 

Then $\Ker _{FF_{\c W_1}}f$ is a natural embedding 
$i_{\c L}:\c K=(K,F(K))\To\c L$, where 
$K=\Ker (f:L\To M)$ and $F(K)=K\cap F(L)$. 
The coimage $\operatorname{Coim}_{FF_{\c W_1}}f=
\Coker _{FF_{\c W_1}}(\Ker _{FF_{\c W_1}}f)$ appears as 
a natural projection $j_{\c L}:\c L\To \c L'=(L',F(L'))$, where 
$L'=f(L)$ and $F(L')=f(F(L))$.

Similarly, 
$\Coker f$ is a natural projection $j_{\c M}:\c M\To \c C=(C,F(C))$, 
where $C=(M/L')/(M/L')_{tor}$ and $F(C)=j_{\c M}(F(M))$. 
Then the image $\operatorname{Im}_{FF_{\c W_1}}f=
\Ker _{FF_{\c W_1}}(\Coker _{FF_{\c W_1}}f)$ is 
a natural embedding $\c M'=(M',F(M'))\To\c M$, where $M'$ is the kernel of 
$j_{\c M}$ and $F(M')=F(M)\cap M'$. 

As usually, there is a natural map $\c L'\To \c M'$ induced by $L'\subset M'$. 
Note that $M/M'=C$ is free and $M'/L'$ is torsion $\c W_1$-modules and 
these properties completely characterize $M'$ as a $\c W_1$-submodule of $M$.

Now suppose $\c L=(L,F(L),\varphi )$, $\c M=(M,F(M),\varphi )$ are objects of 
$\uL ^*_0$ and $f\in\Hom _{\uL ^*_0}(\c L, \c M)$. Use the obvious 
forgetful functor $\uL ^*_0\To FF_{\c W_1}$ and the same notation for the 
corresponding images of 
$\c L$, $\c M$ and $f$. Show that $\c K=\Ker _{FF_{\c W_1}}f$ and 
$\c C=\Coker _{FF_{\c W_1}}f$ have the 
natural structures of objects of $\uL ^*_0$ 
and with respect to this structure they become 
the kernel and, resp, cokernel of $f$ 
in $\uL ^*_0$. 
Indeed, 
$$u^{p-1}K=u^{p-1}L\cap K\subset F(L)\cap K=F(K)=\Ker (f:F(L)\To F(M)).$$
Therefore, $\varphi (F(K))\subset K\cap \varphi (F(L))$ and there is a  
natural embedding $\iota :\varphi (F(K))\otimes _{\sigma \c W_1}\c W_1\subset K$. 
On the one hand, 
$$\operatorname{rk}_{\sigma \c W_1}\varphi (F(K))=\operatorname{rk}_{\c W_1}F(K)=
\operatorname{rk}_{\c W_1}K.$$
On the other hand, $F(L)/F(K)\subset L/K=L'$ have no $\c W_1$-torsion.  
This implies that the quotient 
$\varphi (F(L))/\varphi (F(K))$ has no $\sigma\c W_1$-torsion and the factor of 
$L=\varphi (F(L))\otimes _{\sigma \c W_1}\c W_1$ by 
$\varphi (F(K))\otimes _{\sigma\c W_1}\c W_1$ also has no $\c W_1$-torsion. 
So, $\iota $ becomes the equality 
$\varphi (F(K))\otimes _{\sigma\c W_1}\c W_1=K$ and 
$\c K=(K,F(K),\varphi )=\Ker _{\uL ^*_0}f$.  

The above description of $\Ker _{\uL^*_0}$ implies that $u^{p-1}L'\subset F(L')$, 
\linebreak 
$\varphi (F(L'))=\varphi (F(M))/\varphi (F(K))$ and 
$L'=\varphi (F(L'))\otimes _{\sigma\c W_1}\c W_1$. In other words, 
$\c L'=(L',F(L'),\varphi )\in\uL ^*_0$. 

Now note that for $\c M'=(M',F(M'))$, we have 
$$u^{p-1}M'=(u^{p-1}M)\cap M'\subset F(M)\cap M'=F(M')$$
and, therefore, $F(M')/F(L')$ is torsion $\c W_1$-module and 
\medskip 

$\bullet $\ $(\varphi (F(M'))\otimes _{\sigma\c W_1}\c W_1)/L'$ 
is torsion $\c W_1$-module;
\medskip 

On the other hand, $F(M)/F(M')=F(C)$ is torsion free $\c W_1$-module. This  
implies that $\varphi (F(M))/\varphi (F(M'))$ 
is torsion free $\sigma\c W_1$-module and, therefore, 
\medskip 

$\bullet $\ $M/(\varphi (F(M'))\otimes _{\sigma\c W_1}\c W_1)$ 
is torsion free $\c W_1$-module. 
\medskip 

The above two conditions completely characterize $M'$ as a submodule of $M$. 
Therefore, $\varphi (F(M'))\otimes _{\sigma\c W_1}\c W_1=M'$,  
$\c M'=(M',F(M'),\varphi )\in\uL ^*_0$ and   
$(M/M', F(M)/F(M'),\varphi )=(C,F(C),\varphi )=\c C\in\uL^*_0$. Now a formal 
verification shows that $\c C=\Coker _{\uL^*_0}f$.

Again 
$\operatorname{Coim}_{\uL ^*_0}f=(L',F(L'),\varphi )$ and 
$\operatorname{Im}_{\uL ^*_0}f=(M', F(M'),\varphi )$ together 
with their natural embedding 
$\operatorname{Coim}_{\uL ^*_0}f\To\operatorname{Im}_{\uL ^*_0}f$ in $\uL^*_0$. 
As a matter of fact, these two objects of the category 
$\uL ^*_0$ do not differ very much due to the following Lemma. 

\begin{Lem} \label{L1.4} 
 $\varphi (F(L'))\supset u^p\varphi (F(M'))$ (and, therefore, $L'\supset u^pM'$). 
\end{Lem}

\begin{proof}[Proof of Lemma] 
Otherwise, there is an $l\in\varphi (F(L'))\setminus u^p\varphi (F(L'))$ such that 
$l\in u^{2p}\varphi (F(M'))$. 

Form the sequence $l_n\in L'$ such that $l_1=l$ and for all 
$n\geqslant 2$, $l_{n+1}=\varphi (u^{a_n}l_n)$, where $a_n\geqslant 0$ is such that 
$u^{a_n}l_n\in F(L')\setminus uF(L')$. Clearly, 
all $l_n\notin uF(L')\supset u^{p}L'$. 

On the other hand, $l\in u^{2p}\varphi (F(M'))\subset u^{p+1}F(M')$ and, therefore, 
for all $n\geqslant 1$, $l_n\in \varphi ^n(u^{2p}M')\subset u^{p^n+p}M'$. So, for 
$n\gg 0$, $l_n\in u^pL'$. 
The contradiction. 
\end{proof}

Now suppose $\c L=(L,F(L),\varphi ,N)$ and $\c M=(M,F(M),\varphi ,N)$ are 
objects of $\uL ^*$ and $f\in\Hom _{\uL ^*}(\c L,\c M)$. 
Prove that the kernel $(K,F(K),\varphi )$ and the 
cokernel $(C,F(C),\varphi )$ of $f$ in the category $\uL ^*_0$ have a natural 
structure of objects of the category $\uL ^*$. 

Clearly, $N(K)\subset\Ker \left (f\operatorname{mod}u^{2p}:
L/u^{2p}L\To M/u^{2p}M\right )$. The above Lemma \ref{L1.4} 
implies that 
$u^pL'\supset u^{2p}M'$ and we obtain the following natural maps  
$$L'/u^pL'\overset{\alpha }\longleftarrow L'/u^{2p}M'\overset{\beta }
\To M'/u^{2p}M'\overset{\gamma }\To M/u^{2p}M.$$
Note that $\alpha $ is epimorphic but $\beta $ and $\gamma $ are monomorphic. 
This implies that 
$N(K)\subset \Ker (L/u^{2p}L\To L'/u^pL')$ and  
$$N(K)\operatorname{mod}u^pL\subset \Ker (L/u^pL\To L'/u^pL')=K/u^pK.$$
Therefore, by Proposition \ref{P1.2}, $N$ (as a unique lift 
of $N_1=N\operatorname{mod}u^p$) 
maps $K$ to $K/u^{2p}K$ and $(K,F(K),\varphi ,N)\in\uL ^*$. 

The above property of $\Ker _{\uL ^*}f$ implies that 
$N(L')\subset L'/u^{2p}L'$. Now use that 
$u^pM'\subset L'$, $u^{2p}L'\subset u^{2p}M'$ and 
$N(u^pM')\subset u^pM/u^{2p}M$ to deduce that 
$$N(u^pM')\subset (L'/u^{2p}M')\cap (u^pM/u^{2p}M)=u^pM'/u^{2p}M'.$$
So, $N\operatorname{mod}u^p$ maps $M'$ to $M'/u^pM'$ and again by Proposition 
\ref{P1.2} 
$N(M')\subset M'/u^{2p}M'$. This means that  the kernel of the above constructed  
$\Coker f:(M,F(M),\varphi )\To (C,F(C),\varphi )$ 
is provided with the structure of object of the category 
$\uL ^*$. Therefore,  $N$ induces the map $N:C\To C/u^{2p}C$ and 
$(C,F(C),\varphi ,N)\in\uL^*$.
The proposition is proved. 
\end{proof}

 The above proof shows that the 
kernels and cokernels in the category $\uL ^*$   
appear on the level of filtered modules as 
the kernel and cokernel of  the corresponding map of filtered modules 
$(L_1,F(L_1))$ to $(L,F(L))$ 
in the category of filtered $\W _1$-modules. Therefore, 
the category $\uL ^*$ is special, cf. Appendix A, and we can apply 
the corresponding formalism 
of short exact sequences. In particular, if we take 
another object 
$\L _2=(L_2, F(L_2), \varphi ,N)\in\uL ^*$ then 

$\bullet $\ $i\in\Hom _{\uL ^*}(\L _1, \L )$ is  strict 
monomorphism iff $i:L_1\longrightarrow L$ is injective and $i(L_1)\cap
F(L)=i(F(L_1))$;

$\bullet $\ $j\in\Hom _{\uL ^*}(\L , \L _2)$ is  strict 
epimorphism iff $j:L\longrightarrow L_2$ is epimorphic and 
$j(F(L))=F(L_2)$.
\medskip 

As usually, cf.\,Appendix A, if a monomorphism 
$i$ is strict then the monomorphism  
$j=\Coker\, i$ is  strict, and if an epimorphism $j$ is  
strict then the monomorphism $i=\Ker\,j$ is  strict, and 
under these assumptions 
$0\longrightarrow\L _1\overset{i}\longrightarrow 
\L\overset{j}\longrightarrow \L _2\longrightarrow 0$ 
is short exact sequence.  
\medskip 

With relation to the above result that the categories $\uL ^*_0$ 
and $\uL ^*$ are pre-abelian, note that the situation with the categories 
$\widetilde{\uL}^*_0$ and $\widetilde{\uL}^*$ is different. Indeed, 
let  
$FM_{\c W_1}$ be the category of filtered (not necessarily free) modules over 
$\c W_1$. This category is pre-abelian: for  
$\c M_i=(M_i,F(M_i))$, $i=1,2$, and 
$f\in\Hom _{FM_{\c W_1}}(\c M_1,\c M_2)$, one has  
$\Ker _{FM_{\c W_1}}f=(\Ker f, \Ker f\cap F(M_1))$ (together with its 
natural embedding into $\c M_1$) and 
$\Coker _{FM_{\c W_1}}f=(\Coker f, F(M_2)/(f(M_1)\cap F(M_2))$ 
(together with the natural projection from $\c M_2$). 

Now suppose that $\c M_i=(M_i, F(M_i), \varphi )\in\widetilde{\uL}_0^*$, 
$i=1,2$, and $f\in\Hom _{\widetilde{\uL}^*_0}(\c M_1,\c M_2)$. 
Then $\Ker _{\widetilde{\uL}^*_0}f$ exists (and coincides on the level 
of filtered modules with $\Ker _{FM_{\c W_1}}f$) but 
$\Coker _{\widetilde{\uL}^*_0}f$ exists 
(and coincides on the level of filtered modules with 
$\Coker _{FM_{\c W_1}}f$) only if $f(F(M_1))=f(M_1)\cap F(M_2)$. 
In particular, on the level of filtered modules the composition 
$\Coker _{FM_{\c W_1}}(\Ker _{FM_{\c W_1}}f)$ always 
makes sense and coincides with 
the natural projection $\c M_1\To (f(M_1), f(F(M_1))$. 
Therefore, one can introduce the concept of strict epimorphism  
in $\widetilde{\uL}^*_0$: $f$ is strict epimorphism iff 
$f(M_1)=M_2$ and $f(F(M_1))=F(M_2)$.  

The following situation will appear several times below. 

\begin{Lem} \label{L1.5} 
Suppose $\c M_1, \c M_2\in\widetilde{\uL} ^*_0$,   
$\iota \in\Hom _{\widetilde{\uL}_0^*}(\c M_1,\c M_2)$ 
is a strict epimorphism and   
$\Ker _{\widetilde{\uL}^*_0}\iota =(K,K,\varphi )$. Then for any 
$\c L\in\uL ^*_0$,  
$$\iota ^*:\Hom _{\widetilde{\uL} ^*_0}(\c L,\c M_1)
\To\Hom _{\widetilde{\uL} ^*_0}(\c L,\c M_2)$$ 
is epimorphic. In addition, if $\varphi |_K$ is nilpotent 
then $\iota ^*$ is bijective. 
\end{Lem} 

\begin{proof} The structure of $\c L=(L,F(L),\varphi )$ can be given by 
a vector $\bar l=(l_1,\dots ,l_s)$ and a matrix $C\in M_s(\c W_1)$ such that 

--- $l_1,\dots ,l_s$ is a $\c W_1$-basis of $L$;

--- if $\bar lC=\bar m=(m_1,\dots ,m_s)$ then $m_1,\dots ,m_s$ is a $\c W_1$-basis of $F(L)$;

--- $\bar l=\varphi (\bar m):=(\varphi (m_1),\dots ,\varphi (m_s))$. 
\medskip 

Suppose 
$M_1=(M_1,F(M_1),\varphi )$ and $M_2=(M_2,F(M_2),\varphi )$.

Any $f\in\Hom _{\widetilde{\uL}_0^*}(\c L,\c M_2)$ is given by 
$f(\bar l)\in M_2^s$ such that $f(\bar l)C\in F(M_2)^s$ and $\varphi (f(\bar l)C)=f(\bar l)$. 

Choose an $\hat f(\bar l)\in M_1^s$ such that $\hat f(\bar l)\operatorname{mod}K=f(\bar l)$. 
Then $\hat f(\bar l)C$ modulo $K$ belongs to 
$F(M_2)^s$ and, therefore, 
$\hat f(\bar l)C\in F(M_1)^s$. Clearly, $\bar k_0:=\varphi (\hat f(\bar l)C)-\hat f(\bar l) 
\in K^s$. We must prove the existence of $\bar k_1\in K^s$ such that 
$\varphi ((\hat f(\bar l)+\bar k_1)C)=\hat f(\bar l)+\bar k_1$. This is equivalent to 
$$\bar k_1-\varphi (\bar k_1C)=\bar k_0$$
and the existence of $\bar k_1$ follows from Lemma \ref{L1.1}. This proves that 
$\iota ^*$ is surjective. If $\varphi |_K$ is nilpotent then the bijectivity 
of $\iota ^*$ follows in a similar way from part b) of Lemma \ref{L1.1}.  
\end{proof}

\subsection{Standard exact sequences}\label{S1.2}  
Suppose $\L =(L,F(L),\varphi ,N)\in\uL ^*$. Introduce a $\sigma $-linear map 
$\phi :L\longrightarrow L$ by  
$\phi :l\mapsto \varphi (u^{p-1}l)$. 

\begin{definition} The object $\L $ is etale (resp., connected) 
if $\phi \,\mathrm{mod}\,u$ 
is invertible (resp., nilpotent) on $L/uL$. 
\end{definition} 

Let $\c L(0)=(\c W_1, \c W_1u^{p-1}, \varphi ,N)\in\uL ^*$, where 
$\varphi (u^{p-1})=1$ and $N(1)=u^p\operatorname{mod}u^{2p}$. Then $\c L(0)$ is etale. As a matter of fact, 
it is the simplest etale object of $\uL ^*$ due to the following Lemma. 

\begin{Lem} \label{L1.6} 
 Suppose $\c L=(L,F(L),\varphi ,N)\in\uL ^*$ is etale. Then $\c L$ is a 
product of finitely many copies of $\c L(0)$. 
\end{Lem}

\begin{proof}
 If $\widetilde{L}_0=\{l\in L/uL\ |\ \phi (l)=l\}$ then 
$L/uL=\widetilde{L}_0\otimes _{\F _p}k$. Then there is a unique 
$\F _p$-submodule $L_0$ of $L$ such that $\phi |_{L_0}=\id $ and 
$L=L_0\otimes _{\F _p}\c W_1$. 

Suppose $l\in L_0$. Then $\varphi (u^{p-1}l)=l$ and 
$N(l)=N(\varphi (u^{p-1}l))=\varphi (uN(u^{p-1}l))=\varphi (u^p(l\operatorname{mod}u^{2p})+u^pN(l))
=u^pl\operatorname{mod}u^{2p}L$. 
Therefore, if $e_1,\dots ,e_s$ is an $\F _p$-basis 
of $L_0$ then all $(\c W_1e_i, \c W_1u^{p-1}e_i)$ determine the subobjects 
$\c L_i\simeq\c L(0)$ of $\c L$ and $\c L\simeq \c L_1\times\dots\times\c L_s$. 
\end{proof}

\begin{Prop}\label{P1.7} Any $\L\in\uL ^*$ contains a unique 
maximal etale subobject 
$(\L ^{et}, i^{et})$ and  a unique maximal connected quotient object 
$(\L ^c,j^c)$ and the sequence 
$0\longrightarrow\L ^{et}\overset{i^{et}}
\longrightarrow\L\overset{j^{c}}\longrightarrow
\L ^c\longrightarrow 0$ is short exact. 
\end{Prop} 

\begin{proof} 
Let $\c L=(L,F(L),\varphi ,N)$ and, as earlier, let 
$\phi :L\To L$ be such that for any $l\in L$, $\phi (l)=\varphi (u^{p-1}l)$. 
Then for $\widetilde{L}=L/uL$, we have the $k$-linear subspaces 
$\widetilde{L}^{et}$ and $\widetilde{L}^c$ in $\widetilde{L}$ such that 
$\tilde\phi :=\phi\,\operatorname{mod}u$ is invertible on 
$\widetilde{L}^{et}$ and nilpotent on $\widetilde{L}^c$ and 
$\widetilde{L}=\widetilde{L}^{et}\oplus\widetilde{L}^c$.

Then there is a unique $\c W_1$-submodule $L^{et}$ of $L$ such that 
$\phi |_{L^{et}}$ is invertible and $L^{et}/uL^{et}=\widetilde{L}^{et}$. 
The filtered submodule $(L^{et}, u^{p-1}L^{et})$ determines 
an etale subobject $\iota ^{et}:\c L^{et}\To\c L$. Clearly, 
$u^{p-1}L^{et}\subset L^{et}\cap F(L)$. If the inverse embedding 
does not take place then there is an $l\in L^{et}\setminus uL^{et}$ such that 
$u^{p-1}l\in uF(L)$. Therefore, $\phi (l)=\varphi (u^{p-1}l)\in u^pL$ but 
$\phi |_{L^{et}}$ is invertible. So, $\iota ^{et}$ is strict monomorphism and we can consider 
$\Coker\, \iota ^{et}=j^c:\c L\To\c L^c$. Clearly, $\c L^c$ is connected. 
The maximality properties of $\c L^{et}$ and $\c L^c$ are formally  
implied by the following easy statement: 
\medskip 

{\it if $\c L_1\in\uL ^*$ is etale and $\c L_2\in\uL ^*$ is connected then 
$\Hom _{\uL ^*}(\c L_1,\c L_2)=0$.}
\medskip 
\end{proof}

Suppose $\L =(L,F(L),\varphi ,N)\in\uL ^*$. Then 
$\varphi (F(L))$ is a $\sigma (\W _1)$-module and 
$L=\varphi (F(L))\otimes _{\sigma (\c W _1)}\c W_1$. 
If $l\in L$ and for $0\leqslant i<p$, 
$l^{(i)}\in F(L)$ are such that 
$l=\sum _{0\leqslant i<p}\varphi (l^{(i)})\otimes u^i$,  
set $V(l)=l^{(0)}$. Then $V\mathrm{mod}\, u$ is a 
$\sigma ^{-1}$-linear endomorphism of the $k$-vector space $L/uL$. 

\begin{definition} The module 
$\L $ is multiplicative (resp., unipotent) 
if $\widetilde{V}:=V\,\mathrm{mod}\,u$ 
is invertible (resp., nilpotent) on $\widetilde{L}:=L/uL$. 
\end{definition} 

Let $\c L(1)=(\c W_1, \c W_1, \varphi ,N)\in\uL ^*$, where 
$\varphi (1)=1$ and $N(1)=0$. Then $\c L(1)$ is multiplicative. 
As a matter of fact, 
it is the simplest multiplicative object of 
$\uL ^*$ due to the following Lemma. 

\begin{Lem} \label{L1.8} 
Suppose $\c L=(L,F(L),\varphi ,N)\in\uL ^*$ is multiplicative,  
then $\c L$ is the product of finitely many copies of $\c L(1)$. 
 \end{Lem} 

\begin{proof} Clearly, the embedding $F(L)\To L$ 
induces the identification $F(L)/uF(L)=L/uL$ and, therefore, 
$F(L)=L$.  

 Let $\widetilde{L}_0\subset\widetilde{L}$ be such that 
$\widetilde{V}|_{\widetilde{L}_0}=\id $. If $l\in L$ is such that 
$l\operatorname{mod}\,uL\in\widetilde{L}_0$ then 
$\varphi (l)\equiv l\operatorname{mod}uL$. This implies the existence 
of a unique $l'\in L$ such that $l'\equiv l\,\operatorname{mod}uL$ and 
$\varphi (l')=l'$. 
In other words, there is an $\F _p$-submodule $L_0$ in $L$ such that 
$L=L_0\otimes _{\F _p}\c W_1$ and $\varphi |_{L_0}=\id $. 

If $l\in L_0$ then $N(l)=N(\varphi (l))=\varphi (uN(l))=u^p\varphi (N(l))=0$. 
So, if $e_1,\dots ,e_s$ is an $\F _p$-basis of $L_0$ then the 
filtered modules $(\c W_ie_i, \c W_1e_i)$ determine the 
subobjects $\c L_i\simeq\c L(1)$ of $\c L$ and 
$\c L\simeq\c L_1\times\dots\times\c L_s$.  
\end{proof}

\begin{Prop}\label{P1.9} Any $\L =(L,M,\varphi ,N)\in\uL ^*$ 
contains a unique maximal multiplicative 
quotient object 
$(\L ^{m}, j^{m})$ and a unique maximal unipotent subobject 
$(\L ^u,i^u)$ and the sequence 
$$0\longrightarrow\L ^{u}\overset{i^{u}}
\longrightarrow\L\overset{j^{m}}\longrightarrow
\L ^m\longrightarrow 0$$  
is exact. 
\end{Prop} 

\begin{proof} 
Let $\widetilde{L}=L/uL$, $\widetilde{M}=M/uM$ and 
$\widetilde{L}=\widetilde{L}^m\oplus\widetilde{L}^u$, where 
$\widetilde{V}:=V\operatorname{mod}u$ is invertible on 
$\widetilde{L}^m$ and nilpotent on $\widetilde{L}^u$. 

Note that $\varphi $ induces a  
$\sigma $-linear isomorphism $\tilde\varphi :\widetilde{M}\To\widetilde{L}$.  
Denote by $\tilde\iota :\widetilde{M}\To\widetilde{L}$ the 
$k$-linear morphism induced by 
the embedding $M\subset L$. 
With this notation, for any $l\in\widetilde{L}$, 
$\widetilde{V}(l)=\tilde\iota (\tilde\varphi ^{-1}(l))$. 

Consider the filtration $\widetilde{L}^u\supset \wt V\widetilde{L}^u
\supset\dots\supset \wt V^s\widetilde{L}^u=\{0\}$ and set for 
$1\leqslant i\leqslant s+1$, 
$\widetilde{M}_i=\tilde\varphi ^{-1}(\wt V^{i-1}\widetilde{L}^u)$. 
Then $\widetilde{M}_1\supset \widetilde{M}_2\supset\dots\supset
\widetilde{M}_s\supset\widetilde{M}_{s+1}=\{0\}$ and for 
$1\leqslant i\leqslant s$, 
\begin{equation} \label{ED1}
 \tilde\iota (\widetilde{M}_i)=\wt V^i\wt L^u=\tilde\varphi (\widetilde{M}_{i+1}).
\end{equation}

For $1\leqslant i\leqslant s+1$, introduce the $\c W_1$-submodules 
$M_i^{(0)}$ of $M$ 
such that $M_1^{(0)}\supset M_2^{(0)}\supset\dots
\supset M_s^{(0)}\supset M_{s+1}^{(0)}=0$ and  
$M_i^{(0)}/uM_i^{(0)}=\widetilde{M}_i$ with respect to the natural projection 
$M\To\widetilde{M}$. Then conditions 
\eqref{ED1} imply that for all $i$, 
$M_i^{(0)}\subset \varphi (M_{i+1}^{(0)})\otimes _{\sigma\c W_1}\c W_1+uL.$ 

Let $\widetilde{M}^m=\tilde\varphi ^{-1}(\widetilde{L}^m)$ 
and let $M^m\subset M$ be a $\c W_1$-submodule 
such that $M^m/uM^m=\widetilde{M}^m$ with respect to the natural 
projection $M\To\widetilde{M}$. Then 
\begin{equation} \label{ED2} 
M^m+uL=\varphi (M^m)\otimes_{\sigma\c W_1}\c W_1+uL
\end{equation}
and $M=M^m\oplus M_1^{(0)}$.

Prove the existence of ``more precise'' lifts $M_i^{(n)}$ of $\widetilde{M}_i$, 
where $0\leqslant i\leqslant s+1$ and $n\geqslant 1$. 

\begin{Lem} \label{L1.10} 
 For all $n\geqslant 1$ and $0\leqslant i\leqslant s+1$, 
there are $\c W_1$-modules $M_i^{(n)}$ 
such that 
\medskip 

{\rm a)}  $M_1^{(n)}\supset M_2^{(n)}\supset \dots 
\supset M_s^{(n)}\supset M_{s+1}^{(n)}=\{0\}$ and 
$M_i^{(n)}/uM_i^{(n)}=\widetilde{M}_i$ with respect to the natural 
projection $M\To\widetilde{M}$; 
\medskip 

{\rm b)} 
$M_i^{(n)}\subset \varphi (M_{i+1}^{(n)})\otimes _{\sigma\c W_1}\c W_1+
u\varphi (M_1^{(n)})\otimes _{\sigma\c W_1}\c W_1+u^{n+1}L$; 
\medskip 

{\rm c)} $M_i^{(n-1)}+u^nM=M_i^{(n)}+u^nM$.
\end{Lem}

\begin{proof} [Proof of Lemma] 
The modules $M_i^{(0)}$, $0\leqslant i\leqslant s+1$, satisfy the requirements  
a) and b) of our Lemma. Therefore, we can assume that 
the modules $M_i^{(n)}$ satisfying the requirements 
a)-c) have been already constructed for $n=N-1$, where $N\geqslant 1$.  

Note that $M=M^m\oplus M_1^{(N-1)}$ (it is known for $N=1$ and follows from c) 
for $N>1$). Therefore, \eqref{ED2} implies that  
$$L=\varphi (M^m)\otimes _{\sigma\c W_1}\c W_1+
\varphi (M_1^{(N-1)})\otimes _{\sigma\c W_1}\c W_1\qquad\qquad\qquad  $$
$$\qquad \qquad\qquad \subset 
M^m+\varphi (M_1^{(N-1)})\otimes _{\sigma\c W_1}\c W_1+uL.$$

Therefore, for $1\leqslant i\leqslant s$ (use b) for $n=N-1$),  
 
$$M_i^{(N-1)}\subset \varphi (M_{i+1}^{(N-1)})\otimes _{\sigma\c W_1}\c W_1+
u\varphi (M_1^{(N-1)})\otimes _{\sigma\c W_1}\c W_1+u^{N}M^m+u^{N+1}L$$
and we can define the submodules $M_i^{(N)}$ in such a way that the 
property c) holds for $n=N$ 
\begin{equation}\label{ED2} 
M_i^{(N)}+u^NM^m=M_i^{(N-1)}+u^NM^m
\end{equation}  
and 
\begin{equation}\label{ED3} 
M_i^{(N)}\subset \varphi (M_{i+1}^{(N-1)})\otimes _{\sigma\c W_1}\c W_1+
u\varphi (M_1^{(N-1)})\otimes _{\sigma\c W_1}\c W_1+u^{N+1}L.
\end{equation} 

Note that \eqref{ED2} implies that 
$\varphi (M_i^{(N)})+u^{Np}L=\varphi (M_i^{(N-1)})+u^{Np}L$ and, therefore, 
we can replace $\varphi (M_i^{(N-1)})$ and 
$\varphi (M_1^{(N-1)})$ by $\varphi (M_i^{(N)})$ and, 
resp. $\varphi (M_1^{(N)})$ in \eqref{ED3}. 
The lemma is proved. 
\end{proof} 

Let $M^u=\bigcap _{n\geqslant 0}(M_1^{(n)}+u^{n+1}M)$. Then 
$M^u/uM^u=\widetilde{M}^u$ with respect to the natural 
projection $M\To\widetilde{M}$ and $M=M^m\oplus M^u$. 

Let $L^u=\varphi (M^u)\otimes _{\sigma\c W_1}\c W_1$. Then 
$\operatorname{rk} _{\c W_1}L^u=\operatorname{rk} _{\c W_1}M^u$ and 
$$L^u=\bigcap _{n\geqslant 0}\left (\varphi (M_1^{(n)})\otimes _{\sigma\c W_1}
\c W_1+u^{(n+1)p}L\right )\supset M^u$$
(use Lemma \ref{L1.10}b)). On the other hand,  
$$L=\varphi (M^m\oplus M^u)\otimes _{\sigma\c W_1}\c W_1=M^m\oplus L^u$$
implies that $M^u\supset u^{p-1}L^u$ and $L^u\cap M=M^u$. Therefore, 
the filtered module $(L^u,M^u)$ defines a unipotent subobject 
$\c L^u$ of $\c L$ in the category $\uL ^*_0$ 
and the natural embedding $\c L^u\To\c L$ is strict. 

Suppose $l\in M^u$ and $N(l)\equiv l_0+l_1\operatorname{mod}u^pL$, 
where $l_0\in M^m$ and 
$l_1\in L^u$. Then $uN(l)\equiv ul_0+ul_1\in (M^m
\oplus M^u)\operatorname{mod}u^pL$ and 
$N(\varphi (l))=\varphi (uN(l))\equiv\varphi (ul_1)\operatorname{mod}u^pL$ 
implies that $N(L^u)\subset L^u\operatorname{mod}u^pL$. Then 
from Proposition \ref{P1.2} 
it follows that $\c L^u$ is a subobject of $\c L$ in the category $\uL ^*$. 
Clearly, the quotient $\c L/\c L^u:=\c L^m$ is multiplicative. 

The maximality of $\c L^{u}$ and $\c L^m$ are formally  
implied by the following easy property of objects $\c L_1, \c L_2\in\uL ^*$: 
\newline 
{\it if $\c L_1$ is unipotent and $\c L_2$ is 
multiplicative then $\Hom _{\uL ^*}(\c L_1,\c L_2)=0$.}
\medskip 
\end{proof}

Using the above results we can introduce the subcategories 
$\uL ^{*et}$, $\uL ^{*c}$, $\uL ^{*m}$, $\uL ^{*u}$ in $\uL ^*$. 
They consist of, resp., etale, connected, multiplicative and unipotent objects 
of the cattegory $\uL ^*$. 
The correspondences 
$\c L\mapsto \c L^{et}$, $\c L\mapsto\c L^c$, $\c L\mapsto \c L^m$, 
$\c L\mapsto\c L^u$ determine the natural exact functors from $\uL ^*$ to, resp., 
$\uL ^{*et}$, $\uL ^{*c}$, $\uL ^{*m}$ and $\uL ^{*u}$. 
\medskip 

\subsection{The category $\uL ^*_{cr}$} \label{S1.3}

\begin{Prop} \label{P1.11} Suppose $\L =(L,F(L),\varphi, N)\in\uL ^*$. 
Then the following conditions are equivalent:

(a) $N(F(L))\subset F(L)\operatorname{mod}u^{2p}L$;

(b)  $N(\varphi (F(L)))\subset u^pL\operatorname{mod}u^{2p}L$.
\end{Prop}

\begin{proof} $(a)\Rightarrow (b)$: if for any 
$l\in F(L)$, $N(l)\in F(L)\operatorname{mod}u^{2p}L$ then 
$N(\varphi (l))=\varphi (uN(l))=u^p\varphi (N(l))\in u^pL\operatorname{mod}u^{2p}L$.

$(b)\Rightarrow (a)$: for any $l\in F(L)$, 
$\varphi (uN(l))=N(\varphi (l))\in u^pL\operatorname{mod}u^{2p}L$; 
now use that $\varphi $
induces embedding of $F(L)/uF(L)$ into $L/u^pL$ to deduce that 
$uN(l)\in uF(L)\operatorname{mod}u^{2p}L$, i.e.  
$N(l)\in F(L)\operatorname{mod}u^{2p}L$ (use that $u^{p-1}L\subset F(L)$).
\end{proof}

\begin{definition} The category $\uL ^*_{cr}$ is a full 
subcategory of $\uL ^*$ consisting of $(L,F(L),\varphi ,N)$ such that 
$N:L\longrightarrow L$ satisfies the equivalent conditions from 
Proposition \ref{P1.11}.  
\end{definition}

\begin{remark} a) If $\c L=(L,F(L),\varphi ,N)\in\uL ^*_{cr}$ then 
 $N_1=N\operatorname{mod}u^p$ is a unique $\c W_1$-differentiation 
$N_1:L\To L/u^p$ whose restriction to $\varphi (F(L))$ is the zero map. 
Therefore, any $\c L\in\uL ^*_0$ has at most one structure 
of object of the category $\uL ^*$.

b) Any etale or multiplicative object from $\uL ^*$ belongs to $\uL ^*_{cr}$. 

c) If $f$ is a morphism in $\uL ^*_{cr}$ then $\Ker _{\uL ^*}f=\Ker _{\uL^*_{cr}}f$ and 
$\Coker _{\uL ^*}f=\Coker _{\uL ^*_{cr}}f$. In particular, we can introduce 
the full subcategories $\uL ^{*et}_{cr}$, $\uL ^{*c}_{cr}$, $\uL ^{*m}_{cr}$, 
$\uL ^{*u}_{cr}$ of, resp., etale, connected, multiplicative and unipotent 
objects of $\uL ^*_{cr}$. 
\end{remark}

\begin{Prop} \label{P1.12} Suppose $\L =(L,F(L),\varphi ,N)\in\uL ^*_{cr}$. 
Then  
there is a $\sigma (\W _1)$-basis $l_1,\dots ,l_s$ 
of $\varphi (F(L))$ and integers $0\leqslant c_i<p$, 
where $1\leqslant i\leqslant s$, such that 
$u^{c_1}l_1,\dots ,u^{c_s}l_s$ is a $\W _1$-basis of $F(L)$. 
\end{Prop}

\begin{proof} Choose a $\W _1$-basis $m_1,\dots ,m_s$ of $L$ 
such that for suitable integers $c_1,\dots ,c_s$, the elements 
$u^{c_1}m_1,\dots ,u^{c_s}m_s$ form a $\W _1$-basis of $F(L)$. 
Clearly all $0\leqslant c_i<p$. 

For $1\leqslant i\leqslant s$ and $j\geqslant 0$, 
let $l_{ij}\in \varphi (F(L))$ be such that 
$m_i=\sum_{j\geqslant 0}u^jl_{ij}$. Note that 
$\{l_{i0}\ |\ 1\leqslant i\leqslant s\}$ is a $\sigma (\W _1)$-basis of 
$\varphi (F(L))$ and it will be sufficient to prove that 
all $u^{c_i}l_{i0}\in F(L)$ because then the elements $l_i:=l_{i0}$
will satisfy the requirements of our proposition.  

For all $1\leqslant i\leqslant s$, the element 
$$N(u^{c_i}m_i)=
-\sum_{j}(j+c_i)u^{j+c_i}(l_{ij}\operatorname{mod}u^{2p}L)+\sum_{j}u^{j+c_i}N(l_{ij})$$
belongs to $F(L)\operatorname{mod}u^{2p}L$ if and only if 
$\sum _{j}(j+c_i)u^{j+c_i}l_{ij}\in F(L)$. 
(Use that $u^pL\subset uF(L)$.) 
This implies that for all integers $k\geqslant 0$, 
$\sum _{j}(j+c_i)^ku^{j+c_i}l_{ij}\in F(L)$. 
Therefore, for any $\alpha\in\Z /p\Z $, 
$$\sum _{(j+c_i)\mathrm{mod}p=\alpha }u^{j+c_i}l_{ij}\in F(L).$$
In particular, taking $\alpha =c_i\mathrm{mod}\,p$ and using that 
$u^pl_{ij}\in F(L)$, we obtain that $u^{c_i}l_{i0}\in F(L)$.
\end{proof} 

\begin{remark}
 a) Suppose $\c L=(L,F(L),\varphi )\in\uL ^*_0$ and satisfies the conclusion 
of Proposition \ref{P1.12}. Define the $\c W_1$-differentiation 
$N_1:L\To L/u^pL$ by setting $N_1(l_1)=\dots =N(l_s)=0$. 
If $N:L\To L/u^{2p}L$ is the extension of $N_1$ given by Propostion \ref{P1.2} 
then $(L,F(L),\varphi ,N)\in\uL ^*_{cr}$. In other words, 
Proposition \ref{P1.12} characterizes the objects of $\uL ^*_0$ coming from 
$\uL ^*_{cr}$.

b)  For an object $(L,F(L),\varphi ,N)\in\uL ^*_{cr}$, 
Proposition \ref{P1.12} implies
 that if $\sum _{0\leqslant i<p}u^il_i\in F(L)$, where  
all $l_i\in \varphi (F(L))$, then all $u^il_i\in F(L)$. 
\end{remark}

Consider the category of filtered Fontaine-Laffaille modules
$\underline{\mathrm{MF}}_{p-1}$ from \cite{refFL}. The objects of this category are finite
dimensional $k$-vector spaces $M$ with decreasing filtration 
of length $p$ 
by subspaces $M=M^0\supset M^1\supset \dots \supset M^{p-1}\supset
M^p=0$ and $\sigma $-linear maps 
$\varphi _i:M^i\longrightarrow M$ such that 
$\Ker\,\varphi _i\supset M^{i+1}$, where $0\leqslant i< p$,
 and $\sum _i\mathrm{Im}\,\varphi _i=M$. 
The morphisms in $\underline{\mathrm{MF}}_{p-1}$ are the morphisms of filtered 
vector spaces which commute with the corresponding morphisms $\varphi
_i$, $0\leqslant i<p$. 

The category $\MF _{p-1}$ is abelian. The object $M$ of $\MF _{p-1}$ is:

--- etale (resp., multiplicative) if $M^1=0$ (resp., $M=M^{p-1}$);

--- connected (resp., unipotent) if $M$ has no etale 
(resp., multiplicative) subquotient.

Introduce the full subcategories $\MF _{p-1}^{et}$, $\MF _{p-1}^{m}$, 
$\MF _{p-1}^{c}$ and $\MF _{p-1}^u$ of, resp., etale, multiplicative, connected and unippotent 
objects in $\MF _{p-1}$. These subcategories are closed 
under the operations of taking 
subobjects and quotient objects and, therefore, are also abelian. 
For any $M\in\MF _{p-1}$, there are standard exact sequences 
$0\To M^{et}\To M\To M^c\To 0$ and 
$0\To M^u\To M\To M^m\To 0$, 
where $M^{et}$ (resp., $M^u$) is the maximal etale (resp., unipotent) subobject and 
$M^c$ (resp., $M^m$) is the maximal connected (resp., multiplicative) quotient object. 

The categories $\uL ^*_{cr}$  and $\MF _{p-1}$ do not differ very much. 

Indeed, introduce the functor $\Md :\widetilde{\uL} ^*\To\widetilde{\uL} ^*$ 
induced on the level of filtered modules by $(L,F(L))\mapsto (L/u^pL, F(L)/u^pL)$. 
Denote by $\Md (\uL^*_{cr})$ the full subcategory of $\wt{\uL} ^*$ 
consisting of the objects $\Md (\c L)$, where $\c L\in\uL ^*_{cr}$. 

Define the functor $\c F:\MF _{p-1}\To \wt{\uL} ^*$ as follows. 
Let $M\in\MF _{p-1}$ with the corresponding filtration $M^i$ and 
$\sigma $-linear morphisms $\varphi _i$, $0\leqslant i<p$. 
Then on the level of objects, $\c F(M)=(L,F(L),\varphi ,N)$, where 
$L=M\otimes _k\c W_1/u^p\c W_1$, 
$F(L)=\sum _{0\leqslant i<p}u^{p-1-i}\c W_1(M^i\otimes 1)$ and for any $m\in M_i$, 
$\varphi (u^{p-1-i}m_i)=\varphi _i(m_i)$. One can easily 
see that $\c F$ is equivalence of the categories 
$\MF _{p-1}$ and $\Md (\uL ^*_{cr})$. 

Now the difference between the categories $\uL ^*_{cr}$ and $\MF _{p-1}$ 
is described by the following Proposition. 

\begin{Prop} \label{P1.13} 
For $\c L_1,\c L_2\in\uL ^*_{cr}$, $\Md $ induces a surjection    
from $\Hom _{\uL^*_{cr}}(\c L_1,\c L_2)$ to  
$\Hom _{\widetilde{\uL}^*}(\Md (\c L_1),\Md (\c L_2))$ 
and its kernel coincides with 
$(i_{\c L_2*}\circ j_{\c L_1}^{*})\Hom _{\uL ^*}(\c L_1^m, \c L_2^{et})$, 
where $i_{\c L_2}:\c L_2^{et}\To\c L_2$ (resp., 
$j_{\c L_1}:\c L_1\To\c L_1^m$) is the maximal etale subobject in $\c L_2$ 
(resp., multiplicative quotient object for $\c L_1$). 
 \end{Prop}

\begin{proof} For $\c L_2=(L_2, F(L_2), \varphi ,N)$, let $\phi :L_2\To L_2$ 
be such that $\phi (l)=\varphi (u^{p-1}l)$ for any $l\in L_2$. Let 
$L_2^c=\{l\in L_2\ |\ \phi (l)^n\underset{n\to\infty }\To 0\}$ and let 
$\c L_2'\in\widetilde{\uL} ^*$ be the filtered module 
$(L_2/u^pL^c, F(L_2)/u^pL_2^c)$ with $\varphi $ and $N$ induced from $\c L_2$. 
Then there are natural strict epimorphisms 
$$\c L_2\overset{\alpha }\To\c L_2'\overset{\beta }\To\Md (\c L_2),$$
where $\Ker\,\alpha $ is associated with the filtered module 
$(u^pL_2^c, u^pL_2^c)$ and $\Ker\,\beta $ --- with $(u^pL_2/u^pL_2^c, u^pL_2/u^pL_2^c)$. 

Clearly, $\varphi |_{u^pL_2^c}$ is nilpotent and then  
by Lemma  \ref{L1.5},  
$$\alpha _*:\Hom _{\uL ^*}(\c L_1,\c L_2)
\To\Hom _{\widetilde{\uL }^*}(\c L_1, \c L_2')$$
is bijective. Note that the natural embedding $L_2^{et}\To L_2$ 
induces the identification 
$u^pL_2/u^pL_2^c=u^pL_2^{et}/u^{p+1}L_2^{et}$. 
Let $\c L_2''=(u^pL_2, u^pL_2)\in\uL ^*$ 
with induced $\varphi $ and $N$. Then $\c L_2''$ is multiplicative  
and there is a natural projection $\gamma :\c L_2''\To \Ker\,\beta $ such that 
$\Ker\,\gamma $ is associated with $(u^{p+1}L_2^{et}, u^{p+1}L_2^{et})$. 
Note that   
$\varphi $ is nilpotent on $u^{p+1}L_2^{et}$. Applying Lemma \ref{L1.5}  
we obtain that 
$$\beta _*:\Hom _{\widetilde{\uL} ^*}(\c L_1, \c L_2')\To 
\Hom _{\widetilde{\uL}^*}(\c L_1, \Md (\c L_2))$$
is surjective and 
$$\Ker\,\beta _*=\Hom _{\widetilde{\uL} ^*}(\c L_1,\Ker\beta )=
\Hom _{\widetilde{\uL}^*}(\c L_1,\c L'')\simeq  
\Hom (\c L_1^m,\c L_2'').$$
 It remains to note that 
$\Hom _{\widetilde{\uL}^*}(\c L_1^m,\c L_2'')=\Hom _{\uL ^*}(\c L_1^m, \c L_2^{et})$ 
via the natural embedding of $\c L''$ into $\c L_2^{et}$. 
\end{proof} 

\begin{Cor} \label{C1.14} The functor $\Md\circ \c F^{-1}$ induces equivalence 
of the categories $\uL ^{*c}_{cr}$ (resp., $\uL ^{*u}_{cr}$) 
and $\MF _{p-1}^c$ (resp., $\MF _{p-1}^u$). 
\end{Cor}

\medskip

\subsection{Simple objects in $\uL ^*$}\label{S1.4}

\begin{definition} An object $\L $ of $\uL ^*$ is simple if 
any strict monomorphism $i:\L _1\longrightarrow \L$ in $\uL^*$ 
is either isomorphism or the zero morphism. Equivalently, 
$\L $ is simple iff any strict epimorphism $j:\L\longrightarrow\L _2$ 
is either isomorphism or the zero morphism. 
\end{definition}

All simple objects in $\uL ^*$ can be described as follows. 

Let $[0,1]_p=\{r\in\Q\ | \ 0\leqslant r\leqslant 1, v_p(r)=0\}$, 
where $v_p$ is a $p$-adic valuation. 
Then any $r\in [0,1]_p$ can be uniquely written as 
$r=\sum _{i\geqslant  1}a_ip^{-i}$, 
where the digits $0\leqslant a_i=a_i(r)<p$ form  a periodic
sequence. The minimal positive period of this sequence will be denoted
by $s(r)$. 

Let $\tilde r=1-r$. Then $\tilde r\in [0,1]_p$ and $\tilde r=
\sum _{i\geqslant 1}\tilde a_ip^{-i}$, where for all $i\geqslant 1$, 
the digits $\tilde a_i=a_i(\tilde r)$ are such that 
$a_i+\tilde a_i=p-1$. 

\begin{definition} For $r\in [0,1]_p$, let 
$\L (r)=(L(r), F(L(r)), \varphi ,N)$ be the following object 
of the category $\uL _{cr}^*$: 

$\bullet $\ $L(r)=\oplus _{i\in\Z/s(r)}\W _1l_i$;

$\bullet $\ $F(L(r))=\sum _{i\in\Z/s(r)}\c W_1 u^{\tilde a_i}l_i$;

$\bullet $\ for $i\in\Z /s(r)$, $\varphi (u^{\tilde a_i}l_i)=l_{i+1}$. 

$\bullet $\ $N$ is uniquely recovered from the condition 
$N|_{\varphi (F(L))}=0\operatorname{mod}u^p$, cf. Proposition \ref{P1.2}.   
\end{definition}

\begin{remark}
 If $r=0$ or $r=1$ we obtain the objects $\c L(0)$ and $\c L(1)$ introduced 
in Subsection \ref{S1.2}. Note also that $\c L(r)$ is connected 
iff $r\ne 0$ and 
 unipotent iff $r\ne 1$.  
\end{remark}

For $n\in\N $ and $r\in [0,1]_p$, set 
$r(n)=\sum _{i\geqslant 1}a_{i+n}(r)p^{-i}$. Extend this definition
to any 
$n\in\Z$ by setting $r(n):=r(n+Ns(r))$ for a sufficiently large
$N\in\N$.  

\begin{Prop} \label{P1.15} a) If $r\in [0,1]_p$ then $\L (r)$ is simple;

b) if $r_1,r_2\in [0,1]_p$ then 
$\L (r_1)\simeq\L (r_2)$ if and only if there is 
an $n\in\Z$ such that $r_1=r_2(n)$;

c) if $\c L$ is a simple object of the category $\uL ^*$ 
then there is an $r\in [0,1]_p$ such that $\L \simeq \L (r)$.  
\end{Prop} 

\begin{proof} 
 Lemma \ref{L1.16} below implies that the simple objects in 
the categories $\uL ^*_{cr}$ and $\uL ^*$ 
are the same. By Corollary \ref{C1.14}, 
the functor $\Md\circ\c F^{-1}$ transforms 
simple objects of $\uL ^*$ to simple objects in $\MF _{p-1}$. 
It remains to note that an analogue of 
our Proposition for the category $\MF _{p-1}$ is proved in \cite{refFL}.
\end{proof} 

\begin{Lem} \label{L1.16} For any $\c L\in\uL ^*$, there is  
an $\c L^{cr}\in\uL ^*_{cr}$ and 
a strict monomorphism 
$\iota ^{cr}\in\Hom _{\uL ^*}(\c L^{cr},\c L)$ such that  
if $\iota '\in\Hom _{\uL ^*}(\c L',\c L)$ is a strict monomorphism and 
$\c L'\in\uL ^*_{cr}$ then there is a strict monomorphism 
$\alpha :\c L'\To \c L^{cr}$ such that $\iota '=\iota ^{cr}\circ\alpha $. 
\end{Lem}

\begin{proof}[Proof of Lemma]
 Suppose $\c L=(L,F(L),\varphi ,N)$. Consider the 
$k$-linear space  $M:=\varphi (F(L))/u^p\varphi (F(L))$. Let  
$\widetilde{L}=M\otimes _k(\c W_1/u^p\c W_1)=L/u^pL$, 
$\tilde F=F(L)/u^pL$ and  
$\tilde\varphi :\widetilde{F}\To {M}$ be the map induced by $\varphi $.

Proceed by induction to define for all $i\geqslant 1$, the subspaces 
$M_i\subset M$ and the $\c W_1$-submodules 
$\widetilde{F}_i\subset\widetilde{L}$ 
as follows.  

From the definition of $N:L\To L/u^{2p}L$ it follows easily 
that  $N$ induces a $k$-linear map 
$\widetilde{N}_1:M\To M$ and 
$\widetilde{N}_1^p=0$. Therefore, $M_1:=\Ker\,\widetilde{N}_1$ 
is a non-trivial subspace in $M$. 

Suppose $i\geqslant 1$ and $M _i$ has been already  defined.  Let 
$\widetilde{F}_i$ be the submodule of the  
elements of the form $u^al$ in 
$M\otimes _k(\c W_1/u^p\c W_1)$, where 
$a\geqslant 0$, $l\in M_i$ and $u^al\in \widetilde{F}$.   
Then set 
$M_{i+1}=\tilde\varphi (\widetilde{F}_i)$. 

Verify that for all indices $i$, $M_{i+1}\subset M_i$. If $i=1$ we must prove that 
$\wt N_1(M_2)=0$. Indeed, $M_2$ is spanned by $\varphi (u^al)$, where 
$l\in M_1$ and $u^al\in\wt F_1$. But 
$N(\varphi (u^al))=\varphi (uN(u^al))=\varphi (-u^{a+1}l+u^{a+1}N(l))\in u^pL$. 
If $i>1$ then we can assume by induction that $M_{i-1}\subset M_i$. 
This implies that $\wt F_{i-1}\subset \wt F_i$ and $M_i\subset M_{i+1}$. 

We obtained a decreasing sequence of non-trivial finite dimensional 
$k$-linear spaces $\{M_i\ |\ i\geqslant 1\}$.  
For $i\gg 1$, these spaces become a 
non-trivial constant space    
$M^{cr}\subset M$ such that if 
$\widetilde{F}^{cr}=\{u^al\in\widetilde{F}\ |\ a\geqslant 0, 
l\in M^{cr}\}$ then $\tilde\varphi (\widetilde{F}^{cr})=M^{cr}$. 
This subspace $M^{cr}$ 
has the maximality property: if $M'\subset M$ is such that for 
$\widetilde{F}'=\{u^al\in\widetilde{F}\ |\ a\geqslant 0, l\in M'\}$, 
$\tilde\varphi (\widetilde{F}')=M'$ then $M'\subset M^{cr}$. 
Indeed, show as earlier that $M'\subset M_1$ and then proceed 
by induction proving that $M'\subset M_i$ for all $i\geqslant 1$. 

Now in notation from Subsection \ref{S1.3}, there is an $\c L^{cr}\in\uL ^*_{cr}$ 
such that $\Md (\c L^{cr})=(M^{cr}\otimes _k(\c W_1/u^p\c W_1), 
\widetilde{F}^{cr}, \tilde\varphi ,\widetilde{N})$, where $\widetilde{N}|_{M^{cr}}=0$. 
Then from proposition \ref{P1.13} it follows the existence of a strict monomorphism 
$\iota ^{cr}:\c L^{cr}\To\c L$. 
If $\iota ':\c L'=(L', F(L'), \varphi ,N)\To\c L$ is strict monomorphism 
and $\c L'\in\uL ^*_{cr}$ then 
$\Md (\c L')$ is associated with the filtered module 
$(M'\otimes _k(\c W_1/u^p\c W_1), \widetilde{F}')$ and by the above maximality 
property of $M^{cr}$, $M'$ is a subspace in $M^{cr}$ 
and $\Md (\c L')$ is a strict subobject of 
$\Md (\c L^{cr})$. This gives the required strict embedding $\alpha $. 
The Lemma is proved. 
\end{proof}

\subsection{Extensions in $\uL ^*$}\label{S1.5} \ 
Suppose $r_1,r_2\in [0,1]_p$. 
Choose an $s\in\N $ which is divisible by 
$s(r_1)$ and $s(r_2)$ and  
introduce the objects  
$\L _1=(L_1,F(L_1),\varphi ,N)$ and $\L _2=(L_2,F(L_2),\varphi ,N)$ 
of the category $\uL ^*_{cr}$ as follows: 

  $L_1=\oplus _{i\in\Z /s}\W _1l_i^{(1)}$, 
$F(L_1)=\sum _{i\in\Z /s}
\W _1u^{\tilde a_i}l_i^{(1)}$, where 
$r_1=\sum_{i\geqslant 1}a_ip^{-i}$ with the 
digits $0\leqslant a_i<p$, 
$\tilde a_i=(p-1)-a_i$ 
and for all $i\in\Z/s$, 
$\varphi (u^{\tilde a_i}l_i^{(1)})=l_{i+1}^{(1)}$; 

 $L_2=\oplus _{j\in\Z/s}\W _1l_j^{(2)}$, 
$F(L_2)=\sum _{j\in\Z/s}
\W _1u^{\tilde b_j}l_j^{(2)}$, where 
$r_2=\sum_{j\geqslant 1}b_jp^{-j}$ with the digits 
$0\leqslant b_j<p$, $\tilde b_j=(p-1)-b_j$, 
and for all $j\in\Z/s$, 
$\varphi (u^{\tilde b_j}l_j^{(2)})=l_{j+1}^{(2)}$. 

\begin{Lem} \label{L1.17} For $\kappa =1,2$, $\L _\kappa $ is isomorphic to the product of 
$s/s(r_\kappa )$ copies of the (simple) object $\L (r_\kappa )$. 
\end{Lem}

\begin{proof} Take $\kappa =1$. 
For $\gamma\in\F _{p^s}$ and $\bar\imath \in\Z /s(r_1)$, 
let $m_{\bar\imath }(\gamma )=
\sum _{i\operatorname{mod}s(r_1)=\bar\imath }\sigma ^i(\gamma )l_i^{(1)}$ 
and $M(\gamma )=\sum _{\bar\imath \in\Z /s(r_1)}\c W_1
m_{\bar\imath }(\gamma )\subset L_1$. 
Then all $\c M(\gamma )=(M(\gamma ), M(\gamma )\cap F(L_1), \varphi , N)$ 
with induced $\varphi $ and $N$ 
are subobjects of $\c L_1$ isomorphic to $\c L(r_1)$. 
If $\gamma _1,\dots ,\gamma _d$ is an $\F _{p^{s(r_1)}}$-basis 
of $\F _{p^s}$ then  
$\c M(\gamma _1)\times\dots \times\c M(\gamma _d)$ is isomorphic to $\c L_1$. 
(Use that $d=s/s(r_1)$ and $\operatorname{det}(\sigma ^i(\gamma _j))\ne 0$, where 
for a given $\bar\imath $, $i$ is such that 
$i\operatorname{mod}s(r_1)=\bar\imath $ and $1\leqslant j\leqslant d$.) 
\end{proof}

If $\c L=(L,F(L),\varphi ,N)\in\uL ^*$ then we shall use the same notation 
$\c L$ for the image $(L,F(L),\varphi )$ of $\c L$ under the forgetful 
functor from $\uL ^*$ to $\uL^*_0$. Clearly, this forgetful 
functor induces a group homomorphism $\Ext _{\uL ^*}(\c L_2,\c L_1)
\To \Ext _{\uL^*_0}(\c L_2,\c L_1)$. 

Suppose 
$\L =(L, F(L), \varphi )\in\Ext _{\uL _0^*}(\L _2,\L _1)$. 
Consider a $\sigma(\W _1)$-linear section 
$S:l_j^{(2)}\mapsto l_j$, $j\in\Z /s$, of the corresponding 
epimorphic map 
$\varphi (F(L))\longrightarrow\varphi (F(L_2))$. Then
\medskip 

a) \ $L=L_1\oplus \left (\oplus _{j\in\Z/s}
\W _1l_j\right )$;
\medskip 

b)\ for all indices $j\in\Z/s$, there are unique 
elements $v_j\in L_1$, such that 
$F(L)=F(L_1)+\sum _{j\in\Z/s} \c W_1(u^{\tilde b_j}l_j+v_j)$ 
and $\varphi (u^{\tilde b_j}l_j+v_j)=l_{j+1}$; 
\medskip 

c)\ $F(L)\supset u^{p-1}L$ if and only if for 
all $j\in\Z/s$, $u^{b_j}v_j\in F(L_1)$;
\medskip 

d)\ if $S':l_j^{(2)}\mapsto l_j'=l_j+\varphi (w_{j-1})$, where
$j\in\Z/s$ and $w_{j-1}\in F(L_1)$, is another section of the epimorphism 
$\varphi (F(L))\longrightarrow\varphi (F(L_2))$ then for the
corresponding elements $v'_j\in L_1$, it holds  
$v'_j-v_j=w_{j}-u^{\tilde b_j}\varphi (w_{j-1})$.
\medskip 

The constructions from 
above items a)-d) can be summarized as follows. 

\begin{Lem} \label{L1.18}
Let 
$Z(\c L_2, \c L_1)=\{(v_j)_{j\in\Z /s}\in L_1^s\  |\ u^{b_j}v_j\in F(L_1)\}$   
be a subgroup in  
$L_1^s$ and let 
$$B(\c L_2, \c L_1)=\{(w_j-u^{\tilde b_j}\varphi (w_{j-1}))_{j\in\Z /s}\ |
\ w_j\in F(L_1)\}$$  
be a subgroup of $Z(\c L_2, \c L_1)$. Then there is a natural 
isomorphism of abelian groups 
$Z(\c L_2,\c L_1)/B(\c L_2,\c L_1)
\simeq \Ext _{\uL ^*_0}(\c L_2,\c L_1)$.  
\end{Lem} 
\medskip

\begin{Prop} \label{P1.19} Any $\L \in\Ext _{\uL_0^*}(\L _2,\L _1)$  
appears from a  
system of factors $(v_j)_{j\in\Z /s}\in Z(\c L_2, \c L_1)$ 
satisfying the following
normalization condition  

{\bf (C1)} if $v_j=\sum _{i,t}\gamma _{ijt}u^tl_i^{(1)}$ with  
$\gamma _{ijt}\in k$, then $\gamma _{ij\tilde b_j}=0$. 
\end{Prop} 
\begin{proof} Choose a section $S$ of the projection 
$\varphi ‎(F(L))\longrightarrow \varphi (F(L_2))$ with the minimal 
 set 
$\gamma (S)=\{(i,j,\tilde b_j)\ |\ \gamma _{ij\tilde b_j}\ne 0\}$. 
Suppose $\gamma (S)\ne\emptyset $ (otherwise, the proposition 
is proved) and let $(v_j)_{j\in\Z/s}$ be the corresponding system 
of factors. 

Suppose $(i_0,j_0, \tilde b_{j_0})\in \gamma (S)$ and   
$\gamma =\gamma _{i_0j_0\tilde b_{j_0}}$. Replace 
$(v_j)_{j\in\Z /s}$ 
by an equivalent system $(v'_j)_{j\in\Z /s}$ via the elements 
$w_j\in F(L_1)$ such that $w_j=0$ if $j\ne j_0-1$ and 
$w_{j_0-1}=\sigma ^{-1}(\gamma )u^{\tilde a_{i_0-1}}l_{i_0-1}^{(1)}$. 

If   
$v'_{j}=\sum _{i,t}\gamma '_{ijt}u^tl_i^{(1)}$ then   
\medskip 

--- $\gamma '_{i_0j_0\tilde b_{j_0}}=0$;
\medskip 

--- $\gamma '_{i_0-1,j_0-1,\tilde a_{i_0-1}}=\sigma ^{-1}
(\gamma )+\gamma _{i_0-1,j_0-1,\tilde a_{i_0-1}}$;  
\medskip 

--- for all remaining indices $\gamma '_{ijt}=\gamma _{ijt}$. 
\medskip 

Then 
$\gamma (S')\subset \gamma (S)\setminus \{(i_0,j_0,\tilde b_{j_0})\}
\cup\{(i_0-1,j_0-1,\tilde a_{i_0-1})\}$ and the minimality condition for $S$ 
implies $(i_0-1,j_0-1,\tilde a_{i_0-1})\in\gamma (S')\setminus\gamma (S)$. 
Therefore, 
$\tilde a_{i_0-1}=\tilde b_{j_0-1},\ \gamma _{i_0-1,j_0-1,\tilde a_{i_0-1}}=0,\   
\gamma '_{i_0-1,j_0-1,\tilde b_{j_0-1}}=\sigma ^{-1}(\gamma )$,   
and the new section $S'$ again satisfies the minimality condition. 

Repeating the above procedure we obtain for all $n\in\Z /s$, that 
$\tilde a_{i_0-n}=\tilde b_{j_0-n}$, that is 
$\tilde r_1(i_0)=\tilde r_2(j_0)$. 

Choose $\beta \in k$ such that 
$\sigma ^s(\beta )-\beta =\gamma $ and consider 
$w_j\in F(L_1)$ such that for all $0\leqslant n<s$, 
$w_{j_0+n}=\sigma ^n(\beta )u^{\tilde b_{j_0+n}}l_{i_0+n}^{(1)}$. 
Then for the corresponding new system of factors $(v'_j)_{j\in\Z /s}$, where 
$$v'_j=v_j+w_{j}-u^{\tilde b_{j}}\varphi (w_{j-1})=
\sum_{i,t}\gamma '_{ijt}u^tl_i^{(1)},$$
it holds $\gamma '_{i_0,j_0,\tilde b_{j_0}}=0$, and 
$\gamma _{ijt}=\gamma '_{ijt}$ if 
$(i,j,t)\ne (i_0,j_0,\tilde b_{j_0})$. 
This contradicts to the minimality condition for $S$. 
\end{proof} 

\begin{Prop}\label{P1.20} Any $\L \in\Ext _{\uL^*}(\L _2,\L _1)$ 
can be described via  a system of factors 
$(v_j)_{j\in\Z/s}$, satisfying the above 
condition $(\bf C1)$ and the 
normalization condition
\medskip  

{\bf (C2)} the coefficients $\gamma _{ijt}=0$ if 
$t>\tilde a_i$. 
\end{Prop} 

\begin{proof} Suppose 
$v^{(0)}=(v_j)_{j\in\Z /s}$ is such that 
$v_{j_0}=\gamma u^{t_0}l_{i_0}^{(1)}$ with $\gamma\in k$, 
$t_0>\tilde a_{i_0}$ and 
for $j\ne j_0$, $v_j=0$. It will be sufficient to prove 
that any such system of factors is trivial.

Take the elements $w_j^{(0)}$, $j\in\Z /s$, such that 
$w^{(0)}_{j_0}=-\gamma u^{t_0}l_{i_0}^{(1)}$ and $w_j^{(0)}=0$ if $j\ne j_0$. 
Then the corresponding equivalent system 
$(v_j^{(1)})_{j\in\Z /s}$ is such that 
$v_j^{(1)}=0$ if $j\ne j_0+1$, and 
$v^{(1)}_{j_0+1}=\gamma ^pu^{t_1}l_{i_0+1}^{(1)}$, where 
$t_1=\tilde b_{j_0+1}+(t_0-\tilde a_{i_0})p$. 
This implies that $t_1\geqslant p>\tilde a_{i_0+1}$,  
$t_1-\tilde a_{i_0+1}\geqslant t_0-\tilde a_{i_0}$, 
and $t_1-\tilde a_{i_0+1}>t_0-\tilde a_{i_0}$ unless $\tilde b_{j_0+1}=0$, 
$t_1=p$ and $\tilde a_{i_0+1}=p-1$. 

Repeat this procedure by using   
for all $n\geqslant 0$, the appropriate elements 
$w_j^{(n)}$, $j\in\Z /s$, to obtain the equivalent systems of factors 
$(v_j^{(n)})_{j\in\Z /s}$ such that 
$v_j^{(n)}=0$ if $j\ne j_0+n$, and 
$v_{j_0+n}^{(n)}=\gamma ^{p^n}u^{t_n}l_{i_0+n}^{(1)}$. 

If $(\tilde r_2,\tilde r_1,t_0)\ne (0,1,p)$ then 
$t_n\to\infty $ and we can use the elements 
$w_j=\sum_{n\geqslant 0}w_j^{(n)}$, $j\in\Z /s$,  
to trivialize the original system of factors $v^{(0)}$. 

If $(\tilde r_2,\tilde r_1,t_0)=(0,1,p)$, we can trivialize 
$v^{(0)}$ via the elements $w_j$, $j\in\Z /s$, where 
for $0\leqslant n<s$, $w_{j_0+n}=\kappa ^{p^n}u^pl_{i_0+n}^{(1)}$ 
and $\kappa\in k$ is such that $\sigma ^s(\kappa )-\kappa =\gamma $. 
\end{proof}

\begin{Prop} \label{P1.21} 
Suppose $\L =(L,F(L),\varphi )\in\Ext _{\uL_0^*}(\c L_2,\c L_1)$ 
is given via a system of factors  
$(v_j)_{j\in\Z/s}$ satisfying  the normalization condition $\bf{(C1)}$. 
Then $\c L$ comes from $\uL ^*_{cr}$ 
if and only if all $v_j\in F(L_1)$. 
\end{Prop} 

\begin{proof} Let $N_1:L\To L/u^pL$ be a $\c W_1$-differentiation such that 
for all $j\in \Z/s$, $N_1(l_j)=0$ (and, of course, $N_1(l_j^{(1)})=0$). 
If all $v_j\in F(L_1)$, $F(L)$ is generated by 
the elements $u^{\tilde b_j}l_j$ and 
$u^{\tilde a_j}l_j^{(1)}$, $j\in\Z /s$. 
If $m$ is any of these elements then the basic identity 
$N_1(\varphi (m))=\varphi (uN_1(m))$ is, clearly, satisfied.  
By Proposition \ref{P1.2}, $N_1$ can be extended to a unique   
$\c W_1$-differentiation $N:L\To L/u^{2p}$ and  
$\c L=(L,F(L),\varphi ,N)\in\uL ^*_{cr}$. 

Suppose now that $\L=(L,F(L),\varphi ,N)\in\uL ^*_{cr}$ 
and for all $j\in\Z /s$,  
$v_j=\sum_{i,t}\gamma _{ijt}u^tl_i^{(1)}$ with  
$\gamma _{ij\tilde b_j}=0$. 
Consider the following congruence 
(use that $-u^{\tilde b_j}l_j\equiv v_j\operatorname{mod}F(L)$)
\begin{equation}\label{E1.1} 
N(u^{\tilde b_j}l_j+v_j)\equiv 
\sum _{i,t}\gamma _{ijt}(\tilde b_j-t)u^tl_i^{(1)}+u^{\tilde b_j}N(l_j)
\mathrm{mod}F(L).
\end{equation}
The condition $\L\in\uL^*_{cr}$ implies that 
$N(u^{\tilde b_j}l_j+v_j)\in F(L)\operatorname{mod}u^{2p}L$ and 
$N(l_j)\in u^pL\operatorname{mod}u^{2p}L\subset F(L)\operatorname{mod}u^{2p}L$. 
This means that all $(\tilde b_j-t)\gamma _{ijt}u^tl^{(1)}_i\in F(L_1)$. 
Therefore, for $t\ne \tilde b_j$, 
$\gamma _{ijt}u^tl_i^{(1)}\in F(L_1)$, and $v_j\in F(L_1)$. 
The proposition is proved. 
\end{proof}  

\begin{definition} A pair 
$(i_0,j_0)\in (\Z/s)^2$ 
is $(r_1,r_2)_{cr}$-admissible if 
$\tilde a_{i_0}\ne \tilde b_{j_0}$ 
and there is an $m_0=m_{cr}(i_0,j_0)\in\N$ such that 
for $1\leqslant m<m_0$, $\tilde a_{i_0+m}=\tilde b_{j_0+m}$ but 
$\tilde a_{i_0+m_0}>\tilde b_{j_0+m_0}$. 
\end{definition}

\begin{remark}
 For any $(r_1,r_2)_{cr}$-admissible pair of indices $(i_0,j_0)$, 
it holds $\tilde r_1(i_0)>\tilde r_2(j_0)$ (or, equivalently, 
$r_1(i_0)<r_2(j_0)$). 
\end{remark}

\begin{definition} For $(i_0,j_0)\in (\Z /s)^2$ 
and $\gamma\in k$, denote by 
$E_{cr}(i_0,j_0,\gamma )$ the extension 
$\L\in\Ext_{\uL ^*_{cr}}(\L_2,\L_1)$ 
given by the system of factors $(v_j)_{j\in\Z/s}$ such that 
$v_{j_0}=\gamma u^{\tilde a_{i_0}}l_{i_0}^{(1)}$ and 
$v_j=0$ if $j\ne j_0$. 
\end{definition}

\begin{Prop}\label{P1.22} Any element 
$\L\in\Ext _{\uL^*_{cr}}(\L _2, \L_1)$ can be obtained 
as a sum of $E_{cr}(i,j,\gamma _{ij})$, where 
$(i,j)\in (\Z /s)^2$ runs over the set of  
$(r_1,r_2)_{cr}$-admissible 
pairs and all coefficients $\gamma _{ij}\in k$. 
\end{Prop}

\begin{proof} Propositions \ref{P1.19}-\ref{P1.21} imply that any 
$\L =(L,F(L),\varphi ,N)$ from the group 
$\Ext _{\uL ^*_{cr}}(\L _2,\L _1)$ 
can be presented as a sum of extensions 
$E_{cr}(i,j,\gamma _{ij})$, where 
$i,j\in\Z/s$ are such that $\tilde a_i\ne\tilde b_j$, 
and $\gamma _{ij}\in k$. 

If $m_0\in\N$ is such that  
for $1\leqslant m<m_0$, it holds $\tilde a_{i+m}=\tilde b_{j+m}$ 
but $\tilde a_{i+m_0}<\tilde b_{j+m_0}$, then the 
extension $E_{cr}(i,j,\gamma _{ij})$ is trivial, cf. the proof 
of Proposition \ref{P1.20}.  
The proposition is proved. 
\end{proof}

The above proposition describes 
the subgroup $\Ext _{\uL ^*_{cr}}(\L _2,\L _1)$ 
of \linebreak  
$\Ext _{\uL ^*}(\L _2,\L _1)$. In particular, working 
modulo this subgroup we can describe the extensions in 
the whole category $\uL^*$ via the systems of factors 
$(v_j)_{j\in\Z/s}\in Z(\c L_2, \c L_1)$ such that all  
$v_j=\sum_{i,t}\gamma _{ijt}u^tl_i^{(1)}$ 
satisfy the normalization conditions {\bf (C1)} and 
\medskip 

{\bf (C3)} {\it  if $t\geqslant\tilde a_i$ then $\gamma _{ijt}=0$}.
\medskip 

\begin{Prop} \label{P1.23}Suppose the system of factors 
$(v_j)_{j\in\Z/s}$ 
satisfies the conditions ${\bf (C1)}$ and ${\bf (C3)}$. 
If it determines $\L=(L,F(L),\varphi )\in 
\Ext _{\uL^*_0}(\c L_2,\c L_1)$ from the image of 
$\Ext _{\uL^*}(\c L_2,\c L_1)$ then: 
\medskip

{\rm a)}\ $\gamma _{ijt}=0$ if $t<\tilde a_i-1$;
\medskip

{\rm b)}\ if $t=\tilde a_i-1$ and there is an 
$m_0\in\N$ such that 
for all $1\leqslant m<m_0$, 
$\tilde a_{i+m}-1=\tilde b_{j+m}$
but $\tilde a_{i+m_0}-1>\tilde b_{j+m_0}$, then $\gamma _{ijt}=0$;
\medskip 

{\rm c)}\ if $t=\tilde a_i-1$ and for all $m\in\Z/s$, 
$\tilde a_{i+m}-1=\tilde b_{j+m}$ then $\gamma _{ijt}=0$. 
\end{Prop}

\begin{proof} Suppose $\c L=(L,F(L),\varphi ,N)\in\Ext _{\uL ^*}(\c L_2,\c L_1)$ 
and $(v_j)_{j\in\Z /s}$ describes the image of $\c L$ in 
$\Ext _{\uL ^*_0}(\c L_2,\c L_1)$. 
By the definition of $N$,  $uN(u^{\tilde b_j}l_j+v_j)\in F(L)$, and this implies that   
$\gamma _{ijt}=0$ if $t<\tilde a_i-1$, $t\ne\tilde b_j$  
(use congruence \eqref{E1.1}). This proves a).  

Now we can set for all indices $i$ and $j$,  
$\gamma _{ij}:=\gamma _{i,j,\tilde a_i-1}$. 

Let $\kappa _{ij}\in k$ be  
such that $N(l_j)\equiv \sum _i\kappa _{ij}l_i^{(1)}\mathrm{mod}u^pL$ 
and suppose $\gamma _{ij}\ne 0$ 
(this implies that $\tilde b_j\ne \tilde a_i-1$). 
For $m\geqslant 0$, consider the relations 
\begin{equation} \label{E1.3}N(l_{j+m+1})=
\varphi (uN(u^{\tilde b_{j+m}}l_{j+m}+v_{j+m})).
\end{equation}

If $m=0$ then \eqref{E1.3} implies 
$\kappa _{i+1,j+1}=
\gamma _{ij}^p(\tilde b_j-\tilde a_i+1)$. 
Suppose that there is an $m_0\geqslant 0$ such that 
for all $1\leqslant m<m_0$, $\tilde a_{i+m}-1=\tilde b_{j+m}$ but  
$\tilde a_{i+m_0}-1\ne \tilde b_{j+m_0}$. Then \eqref{E1.3} 
together with   
\eqref{E1.1} (where $j$ is replaced by $j+m$) imply that for 
$1\leqslant m<m_0$, 
$$\kappa _{i+m+1,j+m+1}=\kappa _{i+m,j+m}^p=
\gamma _{ij}^{p^{m+1}}(\tilde b_j-\tilde a_i+1).$$

In particular, $N(l_{j+m_0})\mathrm{mod}\,u^pL$ contains 
$l^{(1)}_{i+m_0}$ with the coefficient 
$\gamma _{ij}^{p^{m_0}}(\tilde b_j-\tilde a_i+1)$. 
Therefore, 
$uN(u^{\tilde b_{j+m_0}}l_{j+m_0}+v_{j+m_0})
\mathrm{mod}\,u^pL$ contains 
$l^{(1)}_{j+m_0}$ with the coefficient 
$ u^{\tilde b_{j+m_0}+1}\gamma _{ij}^{p^{m_0}}(\tilde b_j-\tilde a_i+1)$. 
But this monomial must belong to 
$F(L_1)$. This proves that if $\gamma _{ij}\ne 0$ then  
$\tilde b_{j+m_0}+1>\tilde a_{i+m_0}$.

Finally, suppose that for all $m\geqslant 1$, $\tilde a_{i+m}-1=\tilde b_{j+m}$. 
Then $\tilde a_i-1=\tilde a_{i+s}-1=\tilde b_{j+s}=\tilde b_{j}$ and 
$\gamma _{ij}=\gamma _{i,j,\tilde a_i-1}=\gamma _{i,j,\tilde b_j}=0$.
\end{proof}

\begin{remark}
 With notation from the proof of above proposition the elements 
$v_j=\sum _i\gamma _{ij}u^{\tilde a_i-1}l_i^{(1)}$ determine 
a system of factors from $Z(\c L_2,\c L_1)$ iff $\gamma _{ij}=0$ 
when either $\tilde a_i=0$ or $\tilde b_j=p-1$ 
(in this case $v_j$ should belong to $F(L_1)$). 
\end{remark}

\begin{definition} A pair 
$(i_0,j_0)\in (\Z/s)^2$ 
is $(r_1,r_2)_{st}$-admissible if: 

$\bullet $\ $\tilde b_{j_0}\ne p-1$ and $\tilde a_{i_0}\ne 0$, cf. above remark;  

$\bullet $\ 
$\tilde a _{i_0}-1\ne \tilde b_{j_0}$; 

$\bullet $\ 
there is an $m_0=m_{st}(i_0,j_0)\in\N$ such that  
for $1\leqslant m<m_0$, $\tilde a_{i_0+m}-1=\tilde b_{j_0+m}$ 
but $\tilde a_{i_0+m_0}-1<\tilde b_{j_0+m_0}$. 
\end{definition} 

\begin{definition} A pair $(i_0,j_0)\in (\Z/s)^2$ is 
$(r_1,r_2)_{sp}$-admissible if 
$i_0=0$ and for all $m\in\Z/s$, 
$\tilde a_{m}-1=\tilde b_{j_0+m}$. 
\end{definition}

\begin{Prop} \label{P1.24} 
 {\rm a)} If $(i_0,j_0)$ is an  $(r_1,r_2)_{st}$-admissible pair  
then  
$r_1(i_0)+1/(p-1)>r_2(j_0)$; 
\medskip 

{\rm b)} if $(0,j_0)$ is  
an $(r_1,r_2)_{sp}$-admissible pair then  
$r_1+1/(p-1)=r_2(j_0)$. 
\end{Prop}

\begin{proof}
a) Here for $1\leqslant m<m_0$, $a_{i_0+m}+1=b_{j_0+m}$ 
and $a_{i_0+m_0}\geqslant b_{j_0+m_0}$. Therefore, 
$$r_1(i_0)+1/(p-1)>\sum _{1\leqslant m\leqslant m_0}(a_{i_0+m}+1)p^{-m}>$$
$$\sum _{1\leqslant m\leqslant m_0}b_{j_0+m}p^{-m}+
\sum _{m>m_0}(p-1)p^{-m}\geqslant r_2(j_0).$$
The part b) can be obtained similarly. 
\end{proof}

Using the calculations from the proof of 
Proposition \ref{P1.23} we obtain the following two statements. 

\begin{Prop} \label{P1.25} Suppose $(i_0,j_0)\in (\Z/s)^2$ 
is $(r_1,r_2)_{st}$-admissible and $\gamma\in k$. 
Then there is a unique $E_{st}(i_0,j_0,\gamma )\in\Ext _{\uL^*}(\L_2,\L_1)$ 
given by the system of factors 
$(v_j)_{j\in\Z/s}$ such that 
$v_{j_0}=\gamma u^{\tilde a_{i_0}-1}l_{i_0}^{(1)}$ and 
$v_j=0$ if $j\ne j_0$, and the map $N$, which is uniquely determined 
by the condition:

$\bullet $\  if $j=j_0+m$ with 
$1\leqslant m\leqslant m_{st}(i_0,j_0)$ then 
$$N(l_{j_0+m})\equiv \gamma ^{p^m}(\tilde b_{j_0}-
\tilde a_{i_0}+1)l_{i_0+m}^{(1)}
\mathrm{mod}\,u^pL$$ 
and, otherwise, $N(l_j)\equiv 0\,\mathrm{mod}\,u^pL$.
\end{Prop}

\begin{Prop} \label{P1.26} Suppose $(0,j_0)\in (\Z/s)^2$ 
is $(r_1,r_2)_{sp}$-admissible and $\gamma\in \F _q$, $q=p^s$. 
Then there is a unique $E_{sp}(j_0,\gamma )\in\Ext _{\uL^*}(\L_2,\L_1)$ 
given by the zero system of factors 
and the map $N$, which is uniquely determined 
by the condition: 

$\bullet$ \ $N(l_{j_0+m})\equiv\gamma ^{p^m}l_{m}^{(1)}
\mathrm{mod}(u^pL)$, $m\in\Z /s$. 
\end{Prop}

The following proposition gives the uniqueness 
property of the decomposition of elements of $\Ext _{\uL ^*}(\c L_2,\c L_1)$ 
into a sum of standard extensions.

\begin{Prop} \label{P1.27} Any element 
$\L\in\Ext _{\uL^*}(\L _2, \L_1)$ 
appears as a unique  
sum of the extensions $E_{cr}(i,j,\gamma ^{cr}_{ij})$, 
$E_{st}(i,j,\gamma ^{st}_{ij})$ 
and $E_{sp}(j,\gamma ^{sp}_{0j})$, where 
all $\gamma ^{cr}_{ij}, \gamma ^{st}_{ij}\in k$ but $\gamma ^{sp}_{0j}\in\F _q$, and 
$\gamma ^{cr}_{ij}=0$, resp. $\gamma ^{st}_{ij}=0$, $\gamma ^{sp}_{0j}=0$, if 
the corresponding pair of lower indices  is not 
$(r_1,r_2)_{cr}$-admissible, resp. $(r_1,r_2)_{st}$-admissible, 
$(r_1,r_2)_{sp}$-admissible. 
\end{Prop}

\begin{proof} By Proposition \ref{P1.23}, any  
$\c L\in\Ext _{\uL ^*}(\c L_2, \c L_1)$ can be 
decomposed as a sum of the above special extensions. 
To prove the uniqueness of such decomposition, assume that $\c L$ 
represents a trivial element of $\Ext _{\uL ^*}(\c L_2,\c L_1)$ and prove that all 
involved coefficients $\gamma ^{cr}_{ij}$, 
$\gamma ^{st}_{ij}$ and $\gamma ^{sp}_{0j}$ are 
equal to 0.

The image of $\c L$ in $\Ext _{\uL ^*_0}(\c L_2,\c L_1)$ is given 
by the system of factors 
$(v^{cr}_j+v^{st}_j)_{j\in\Z /s}$ such that 

--- $v^{cr}_j=\sum_{i}\gamma ^{cr}_{ij}u^{\tilde a_i}l_i^{(1)}$; 
\medskip 

--- $v^{st}_j=\sum_{i}\gamma ^{st}_{ij}u^{\tilde a_i-1}l_i^{(1)}$.  
\medskip 

Let $w_j\in F(L_1)$ be 
such that for all $j$, $v_j=w_j-u^{\tilde b_j}\varphi (w_{j-1})$. 

If  
$w_j\equiv\sum _{i}\kappa _{ij}u^{\tilde a_i}l_i^{(1)}\mathrm{mod}\,uF(L)$ 
with $\kappa _{ij}\in k$, then  for all $i$ and $j$, 
\begin{equation}\label{E1.2} \gamma _{ij}^{cr}u^{\tilde a_i}+
\gamma _{ij}^{st}u^{\tilde a_i-1}\equiv 
\kappa _{ij}u^{\tilde a_{i}}-\kappa _{i-1,j-1}^pu^{\tilde b_j}
 \mathrm{mod}\,u^{\tilde a_i+1}.
\end{equation}

Suppose $(i_0,j_0)$ is $(r_1,r_2)_{st}$-admissible. 
Then $\tilde a_{i_0}-1\ne \tilde b_{j_0}$ 
and comparing the coefficients for $u^{\tilde a_{i_0}-1}$ 
in \eqref{E1.2} we deduce that 
$\gamma _{i_0j_0}^{st}=0$. Therefore, all $\gamma _{ij}^{st}=0$.  

Suppose $(i_0,j_0)$ is $(r_1,r_2)_{cr}$-admissible. Then for 
$m_0=m_{cr}(i_0,j_0)$, $\tilde a_{i_0}\ne \tilde b_{j_0}$, 
$\tilde a_{i_0+m}=\tilde b_{j_0+m}$ if 
$1\leqslant m<m_0$, and $\tilde a_{i_0+m_0}>\tilde b_{j_0+m_0}$. Then \eqref{E1.2} 
implies that $\gamma ^{cr}_{i_0j_0}=\kappa _{i_0j_0}$, 
$\kappa _{i_0+m,j_0+m}=\kappa _{i_0+m-1,j_0+m-1}^p$ for $1\leqslant m<m_0$, and 
$\kappa _{i_0+m_0-1,j_0+m_0-1}=0$. Therefore, $\gamma _{i_0j_0}^{cr}=0$.

Finally, $\c L$ is the trivial element of the group 
$\Ext _{\uL ^*}(\c L_2,\c L_1)$ 
and, therefore, for all $j$, $N(l_j)\in u^pL$. Then from the description of 
standard extensions $E_{sp}(j,\gamma _{0j}^{sp})$ in Proposition \ref{P1.26} 
it follows that all $\gamma _{0j}^{sp}=0$. 
\end{proof}

\section{The functor $\CV ^*:\uL ^*\longrightarrow\uC _F$}
\label{S2}

\subsection{ The object $\c R ^0_{st}\in\widetilde{\uL }^*$}
\label{S2.1}

Let $R=\underset{n}\varprojlim (\bar O/p)_n$ be Fontaine's ring;  
it has a natural structure of 
$k$-algebra via the map $k\longrightarrow R$ given by  
$\alpha\mapsto\varprojlim ([\sigma ^{-n}\alpha ]\mathrm{mod}\,p)$, 
where for any $\gamma\in k$, $[\gamma ]\in W(k)\subset\bar O$ is the 
Teichm\"uller representative of $\gamma $. 
Let $\m _R$ be the maximal ideal of $R$. 

Choose  
$x_0=(x_0^{(n)}\,\mathrm{mod}\,p)_{n\geqslant 0}\in R$ and 
$\varepsilon =(\varepsilon ^{(n)}\mathrm{mod}\,p)_{n\geqslant 0}$ 
such that for all $n\geqslant 0$, $x_0^{(n+1)p}=x_0^{(n)}$ and 
$\varepsilon ^{(n+1)p}=\varepsilon ^{(n)}$ with 
$x_0^{(0)}=-p$, $\varepsilon ^{(0)}=1$ but $\varepsilon ^{(1)}\ne 1$. 
We shall denote by $v_R$ the valution on $R$ such that $v_R(x_0)=1$.

Let $Y$ be an indeterminate. 

Consider  
the divided power envelope $R\langle Y\rangle $ of $R[Y]$ with 
respect to the ideal $(Y)$. If for $j\geqslant 0$, 
$\gamma _j(Y)$ is the $j$-th divided power of $Y$ then 
$R\langle Y\rangle =\oplus _{j\geqslant 0}R\gamma _j(Y)$.  
Denote by $R_{st}$ the completion $\prod _{j\geqslant 0}R\gamma _j(Y)$ 
of $R\langle Y\rangle $ and  set, 
$\Fil ^p R_{st}=\prod _{j\geqslant p}R\gamma _j(Y)$. 
Define the $\sigma $-linear morphism 
of the $R$-algebra $R_{st}$ by the correspondence $Y\mapsto x_0^pY$;  
it will be denoted below by the same symbol $\sigma $. 

Introduce a $\W _1$-module structure on $R_{st}$ by the $k$-algebra morphism 
$\W _1\longrightarrow R_{st}$ such that 
$u\mapsto \iota (u):=x_0\exp (-Y)=x_0\sum _{j\geqslant 0}(-1)^j\gamma _j(Y)$. 
Set $F(R_{st})=\sum _{0\leqslant i<p}x_0^{p-1-i}R\gamma _i(Y)+\Fil ^pR_{st}$. 
Define the continuous $\sigma $-linear morphism of $R$-modules  
$\varphi\,:F(R_{st})\longrightarrow R_{st}$ by setting 
for  $0\leqslant i<p$,  
$\varphi (x_0^{p-1-i}\gamma _i(Y))=\gamma _i(Y)(1-(i/2)x_0^pY)$, 
and for $i\geqslant p$, $\varphi (\gamma _i(Y))=0$. 
Let $N$ be a unique $R$-differentiation of $R_{st}$ such that 
$N(Y)=1$.

\begin{Prop}\label{P2.1} 

{\rm a)} If $a\in R_{st}$ 
and $b\in F(R_{st})$ then  
$$\varphi (ab)=\sigma (a)\varphi (b)\operatorname{mod}x_0^{2p}R_{st};$$

{\rm b)} $\varphi \operatorname{mod}x_0^{2p}R_{st}$ 
is a $\sigma $-linear morphism of $\c W_1$-modules; 

{\rm c)} for any $b\in R_{st}$ and $w\in \c W_1$, 
$N(wb)=N(w)b+wN(b)$;

{\rm d)} for any $l\in F(R_{st})$,  
$uN(l)\in F(R_{st})$ and 
$$N(\varphi (l))=\varphi (uN(l))\operatorname{mod}x_0^{2p}R_{st}.$$ 
\end{Prop}

\begin{proof} a) It is sufficient to verify it for 
$a=Y$ and $b=x_0^{p-1-i}\gamma _i(Y)$, $0\leqslant i<p$.

b) Use that the multiplication by $\sigma (u)=u^p$ 
comes as the multiplication by 
$\iota (u) ^p=x_0^p\equiv x_0^p\exp (-x_0^pY)=
\sigma (\iota (u))\operatorname{mod}x_0^{2p}R_{st}$. 

c) Use that $N(\iota (u))=-\iota (u)$.

d) It will be enough to check the identity for 
$l=x_0^{p-1-i}\gamma _i(Y)$ with $1\leqslant i<p$. 
Then 
$N(\varphi (l))=\gamma _{i-1}(Y)(1-(1/2)(i+1)x_0^pY)$. On the other hand,   
$uN(l)=
x_0^{p-1-(i-1)}\gamma _{i-1}(Y)\exp (-Y)$ and 
$\varphi (uN(l))$ is equal to 
$$\gamma _{i-1}(Y)\left (1-\frac{i-1}{2}x_0^pY\right )
\exp (-x_0^pY)\equiv\gamma _{i-1}(Y)\left (1-\frac{i+1}{2}x_0^pY\right )
\operatorname{mod}x_0^{2p}.$$ 
\end{proof}

Introduce a $\Gamma _F$-action on 
$R_{st}\operatorname{mod}x_0^{p^2/(p-1)}R_{st}$ as follows.

For any 
$\tau\in\Gamma _F$, let $k(\tau )\in\Z$ be such that 
$\tau (x_0)=\varepsilon ^{k(\tau )}x_0$ and 
let $\widetilde{\log}(1+X)=X-X^2/2+\dots -X^{p-1}/(p-1)$ 
be the truncated logarithm. 
For any $\tau\in\Gamma _F$, define 
a linear map $\tau :R_{st}\To R_{st}$ by extending the natural action of 
$\tau $ on $R$ and  setting 
for $\tau\in\Gamma _F$ and $j\geqslant 0$,  
$$\tau (\gamma _j(Y)):=\sum _{0\leqslant i\leqslant\min\{j,p-1\} }\gamma _{j-i}(Y)
\gamma _i(\wt{\log}\varepsilon).$$

Note that the cocycle relation 
$\varepsilon ^{k(\tau _1)}(\tau _1\varepsilon )^{k(\tau )}=
\varepsilon ^{k(\tau _1\tau )}$, where 
$\tau _1,\tau\in \Gamma _F$, implies the cocycle relation 
$$k(\tau _1)\widetilde{\log}\varepsilon +k(\tau )\widetilde{\log}
(\tau _1(\varepsilon ))\equiv k(\tau _1\tau )\widetilde{\log}\varepsilon 
\operatorname{mod}x_0^{p^2/(p-1)}.$$
(Use that $\widetilde{\log}(1+X)^k\equiv k\wt{\log}(1+X)
\operatorname{mod}(X^p)$ and 
$\varepsilon\equiv 1\operatorname{mod}x_0^{p/(p-1)}$.) 
In addition, for any $k\in\Z $, the obvious congruence 
$$(1+X)^k=\exp (k\log (1+X))\equiv \wt{\exp}(k\wt{\log}(1+X))\operatorname{mod}(X^p)$$
implies that for any $\tau\in\Gamma _F$, $\tau (x_0\exp (-Y))\equiv 
x_0\exp (-Y)\operatorname{mod}x_0^{p^2/(p-1)}$. 

Therefore, the correspondences $\gamma _j(Y)\mapsto \tau (\gamma _j(Y))$ 
induce a $\Gamma _F$-action on $\c W_1$-algebra 
${R}_{st}\operatorname{mod}x_0^{p^2/(p-1)}R_{st}$, which extends the natural 
$\Gamma _F$-action on $R$.   
 
\begin{Prop} \label{P2.2} For any $\tau\in\Gamma _F$, 

{\rm a)} $\tau (F(R_{st}))= F(R_{st})$;
\medskip 

{\rm b)} for any $a\in F(R_{st})$, 
$\tau (\varphi (a))\equiv \varphi (\tau (a)
\operatorname{mod}x_0^{p+1/(p-1)}R_{st}$;
\medskip 

{\rm c)} for any $b\in R_{st}$, $\tau (N(b))=N(\tau (b))$. 
\end{Prop}
 
\begin{proof} The proof is straightforward in cases 
a) and c). Part b) follows by direct calculation from 
the following Lemma.
\end{proof}

\begin{Lem} \label{L2.3} 
 $\sigma (\wt{\log}\varepsilon )/x_0^p\equiv \wt{\log}
\varepsilon \operatorname{mod}x_0^{p+1/(p-1)}R$. 
\end{Lem}

\begin{proof} Consider Fontaine's element 
$$t^+=\log [\varepsilon ]=\sum _{n\geqslant 1} (-1)^{n-1}\frac{([\varepsilon ]-1)^n}{n}
=\sum _{m\in\Z }p^m[\eta _m]\in A_{cr}$$
where all $\eta _m\in R$. Then $t^+\in\Fil ^1A_{cr}$ and $\sigma t^+=pt^+$. This 
implies for all $m\in\Z $, that $\eta _m=\sigma ^{-m}\eta _0$. 

Consider $\c H\subset A_{cr}$ consisting of the elements 
of the form $\sum _{m\in\Z }p^m[r_m]$ such that for $m\leqslant 0$, 
$v_R(r_m)\geqslant p^2/(p-1)$ (this is automatic for $m\leqslant -2$), 
$v_R(r_1)\geqslant p^2/(p-1)-1$ and $v_R(r_2)\geqslant p^2/(p-1)-2$. 
Then $\c H$ is an additive subgroup in $A_{cr}$. 

Verify that 
\medskip 

$\bullet $\ {\it for all $n\geqslant p$, $([\varepsilon ]-1)^n/n\in\c H$.}
\medskip  

Indeed, 
 the congruence 
$[\varepsilon ]\equiv 1+[a_0] \operatorname{mod}\,pW(R)$ (where $a_0=\varepsilon -1$) 
implies that 
$[\varepsilon ]=\lim _{m\to\infty }(1+[\sigma ^{-m}a_0])^{p^m}$. 
Therefore, 
$$[\varepsilon ]-1=\sum _{m\geqslant 0}[a_m]p^m=
[a_0]\left (1+\sum _{m\geqslant 1}p^m[b_m]\right ),$$
where $v_R(a_m)=p^{1-m}/(p-1)$, $v_R(b_1)=-1$ and $v_R(b_2)=-1-1/p$. 

If $n\not\equiv 0\operatorname{mod}p$ then $([\varepsilon ]-1)^n\equiv 
[a_{0n}]+p[a_{1n}]+p^2[a_{2n}]\operatorname{mod}p^3W(R)$ with 
$v_R(a_{0n})=v_R(a_{1n})+1=v_R(a_{2n})+2=pn/(p-1)$. 
This proves that $([\varepsilon ]-1)^n/n\in\c H$ for 
all $n\not\equiv 0\operatorname{mod}p$, $n>p$. 

As for all remaining $n\geqslant p$, just note that for all $M\geqslant 1$, 
$$([\varepsilon ]-1)^{p^M}\equiv [a_0]^{p^M}
(1+p^{M+1}[b_1]+p^{M+2}[b_{2M})])\operatorname{mod}p^{M+3}W(R),$$
where $v_R(b_{2M})=-2$. 
\medskip 
 
The above calculations mean that 
$t^+\equiv\wt{\log}[\varepsilon ]\operatorname{mod}\c H$. Therefore,  if 
$$\wt{\log}[\varepsilon ]=[\omega _0]+p[\omega _1]+p^2[\omega _2]
\operatorname{mod}p^3W(R)$$
then $\omega _0=\wt{\log}\varepsilon \equiv \eta _0\operatorname{mod}x_0^{p^2/(p-1)}R$, 
$$\omega _1\equiv\eta _1\equiv \sigma ^{-1}\eta _0\equiv \sigma ^{-1}\wt{\log}\varepsilon 
\operatorname{mod}x_0^{p^2/(p-1)-1}R$$ 
and 
$\omega _2\equiv\eta _2\operatorname{mod}x_0^{p^2/(p-1)-2}R$. 

Now note that $\wt{\log}[\varepsilon ]\in \Fil ^1A_{cr}\bigcap W(R)$, that is 
$\wt{\log}[\varepsilon ]$ is divisible by $[x_0]+p$ in $W(R)$. The division algorithm 
gives $(\omega _1-\omega _0/x_0)/x_0\equiv \omega _2\operatorname{mod}x_0^{1/(p-1)}R$. 
Therefore, $\sigma (\omega _1)\equiv \sigma (\omega _0)
/x_0^p\operatorname{mod}x_0^p\sigma (\omega _2)R$. 
The lemma is proved. 
\end{proof}

By above results  we can introduce 
$\c R^0_{st}=(R^0_{st}, F(R^0_{st}), \varphi ,N)\in\wt{\uL}^*$, where  
$R^0_{st}=R_{st}\operatorname{mod}x_0^p\m_R$ and $F(R^0_{st})=
F(R_{st})\operatorname{mod}x_0^p\m _R$ 
with induced $\sigma $-linear map $\varphi $ and 
$\c W_1$-differentiation $N$. The above defined $\Gamma _F$-action 
on $R^0\operatorname{mod}x_0^p\m _R$ 
respects the structure of  
$\c R^0_{st}$ as an object of the category $\wt{\uL}^*$.  
In our setting the filtered Galois module $R^0_{st}$ plays a role 
of Fontaine's ring $\hat A_{st}$.

\subsection{The functor $\c V^*$} \label{S2.2}

If $\L=(L,F(L),\varphi ,N)\in\widetilde{\uL}^*$ then the triple 
$(L,F(L),\varphi )$ is an object of $\widetilde{\uL} ^*_0$ 
which will be denoted  below by the same symbol $\L $. 

\begin{definition} Let  
$\c R^0=(R^0,F(R^0),\varphi )\in\widetilde{\uL} ^*_0$, where 
$R^0=R/x_0^p\m _R$, $F(R^0)=x_0^{p-1}R^0$, 
the $\c W_1$-module structure on $R^0$ is given via  
$u\mapsto x_0$ and $\phi $ is induced by the map $r\mapsto r/x_0^{p(p-1)}$, 
$r\in R$. 
\end{definition}

For any $\c L =(L,F(L),\varphi ,N)\in\uL ^*$, consider the $\Gamma _F$-module 
$\c V^*(\c L )=\Hom _{\widetilde{\uL }^*}(\c L ,\c R^0_{st})$. 
If $f\in \c V^*(\c L )$ and $i\geqslant 0$, introduce the $k$-linear morphisms 
$f_i:L\longrightarrow R^0$ such that for any $l\in L$, 
$f(l)=\sum _{i\geqslant 0}f_i(l)\gamma _i(Y)$. The correspondence 
$f\mapsto f_0$ gives the homomorphism of abelian groups  
$\pr _0:\c V^*(\c L )\longrightarrow \c V^*_0(\c L ):=
\Hom _{\widetilde{\uL }_0^*}(\c L ,\c R^0)$.

\begin{Prop} \label{P2.4} $\pr _0$ is isomorphism of abelian groups. 
\end{Prop}

\begin{proof} Clearly, $\pr _0$ is additive. 
Suppose $f\in\Ker\, \pr _0$. Then for all $i\geqslant 0$ and 
$l\in L$,  
$f_i(l)=f_0(N^i(l)))=0$, i.e. $f=0$.

Suppose $g\in \Hom _{\widetilde{\uL}^*_0}(\c L ,\c R^0)$. 
This means that $g:L\To R^0$ is a $\sigma$-linear morphism of 
$\c W_1$-modules, $g(F(L))\subset F(R^0)$ and for any 
$l\in F(L)$, $g(\varphi (l))=(g(l)/x_0^{p-1})^p$. 

Set for any $l\in L$, 
$f(l)=g(l)+g(Nl)\gamma _1(Y)+\dots +g(N^il)\gamma _i(Y)+\dots $. 
Then for any $l\in L$, $f(N(l))=N(f(l))$ and our Proposition 
is implied by the following Lemma.  
\end{proof}

\begin{Lem} \label{L2.5}
{\rm a)} 
For any $l\in L$, $f(ul)=x_0\exp (-Y)f(l)$;

{\rm b)} for any $l\in F(L))$, \ $\varphi (f(l))=f(\varphi (l))$. 
\end{Lem}

\begin{proof}[Proof of Lemma] 
a) For any $l\in L$, 
$f(ul)=
\sum _{i\geqslant 0}g(N^i(ul))\gamma _i(Y)= $  
$$x_0\sum _{i\geqslant 0}
g((N-\id )^il)\gamma _i(Y) =
x_0\sum _{i,s}(-1)^{i-s}\binom {i}{s}g(N^sl)\gamma _s(Y)\gamma _j(Y)$$
$$=
x_0\sum _{j,s}(-1)^jg(N^sl)\gamma _s(Y)\gamma _j(Y)= 
x_0\exp (-Y)f(l).$$

b) Let $l\in L$. Prove by induction on $i\geqslant 1$ that 
$$N^i(\varphi (l))=\varphi ((uN)^i(l))=
-\frac{i(i-1)}{2}u^p\varphi (u^{i-1}N^{i-1}(l))
+\varphi (u^iN^i(l)).$$
Then  
$$g(N^i(\varphi (l)))=-\frac{i(i-1)}{2}x_0^p
\left (\frac{g(u^{i-1}N^{i-1}(l))}{x_0^{p-1}}\right )^p
+\left (\frac{g(u^iN^i(l))}{x_0^{p-1}}\right )^p$$
and 
$f(\varphi (l))$ is equal to 
$\sum _{i\geqslant 0}g(N^i(\varphi (l))\gamma _i(Y)=$
$$\sum _{i\geqslant 0}\left (\frac{g(N^il)}{x_0^{p-1-i}}\right )^p
\left (\gamma _i(Y)-\frac{i(i+1)}{2}x_0^p\gamma _{i+1}(Y)\right )=
\varphi (f(l)).$$
\end{proof}

\begin{Cor} \label{C2.6} 
{\rm a)}\ If $\mathrm{rk} _{\c W_1}L=s$ then 
$|\c V^*(\c L )|=p^s$;
 
{\rm b)}\ the correspondence $\c L\mapsto \c V^*(\c L )$ 
induces an exact functor $\c V^*$ from 
$\uL ^*$ to the category of $\F _p[\Gamma _F]$-modules.
\end{Cor}

\begin{proof} a) Proceed as in \cite{refAb1, refAb3}. Suppose the structure of 
the filtered $\varphi $-module $\L $ is given by a choice of 
a $\c W_1$-basis $m_1,\dots ,m_s$ of $F(L)$ and a non-degenerate matrix 
$A\in M_s(\c W_1)$ such that 
$(m_1,\dots ,m_s)=(\varphi (m_1),\dots ,\varphi (m_s))A$.  
Let $\bar X=(X_1,\dots ,X_s)$ be a vector with 
 $s$ independent variables and let $R_0=\mathrm{Frac}\,R$. 
Consider the quotient $A_{\L }$ of the polynomial ring 
$R_0[\bar X]$ by the ideal generated by the coordinates 
of the vector $(\bar XA)^{(p)}-x_0^{p(p-1)}\bar X$. (For a matrix $C$ 
the matrix $C^{(p)}$ is obtained by raising all 
elements of $C$ to $p$-th power.) 
Then $A_{\L }$ is etale $R_0$-algebra of rank $p^s$ 
(use that $(u^{p-1}I_s)A^{-1}\in M_s(\c W_1)$) 
and all its $\bar R_0$-points  give rise to 
elements of the group $\Hom _{\widetilde{\uL}^*_0}(\c L,\c R)$, 
where $\c R=(R,x_0^{p-1}R,\varphi )\in\widetilde{\uL}^*_0$ is such that 
for any $r\in R$, $\varphi (x_0^{p-1}r)=r^p$.  It remains 
to note that $\varphi |_{x_0^p\m _R}$ is nilpotent, by Lemma \ref{L1.5}, 
the natural projection $\c R\To\c R^0$ induces bijection from 
$\Hom _{\widetilde{\uL}^*_0}(\c L,\c R)$ to 
$\Hom _{\widetilde{\uL}^*_0}(\c L,\c R^0)=\c V^*_0(\c L)$ and by 
Proposition \ref{P2.4}, $|\c V^*_0(\c L)|=|\c V^*(\c L)|$. 

b) This follows from a) because the functor 
$\L\mapsto \c V^*_0(\L)$ is left exact. 
\end{proof}

\medskip 

Introduce the ideal 
$\wt{J}=\sum _{0\leqslant i<p}x_0^{p-i}
\m _R\gamma _i(Y)+\Fil ^pR^0_{st}$ 
of $R^0_{st}$. Then $F(R^0_{st})\supset \wt{J}$ 
and $\varphi |_{\wt{J}}$ is nilpotent.  
For    
$\wt{\c R}_{st}^0=(R^0_{st}/\wt{J}, F(R^0_{st})/\wt{J},
\varphi\operatorname{mod}\wt{J})\in\wt{\uL}^*_0$,  
there is a natural projection 
$\c R^0_{st}\To\wt{\c R}_{st}^0$ in $\wt{\uL}^*_0$ and for any 
$\c L\in\uL ^*_0$, $\Hom _{\wt{\uL}^*_0}(\c L, \c R^0_{st})=
\Hom _{\wt{\uL}^*_0}(\c L,\wt{\c R}_{st}^0)$. 
This implies the following description  
of the $\Gamma _F$-modules $\c V^*(\c L )$ where $\c L\in\uL ^*$  
(use the identification $\pr _0$ 
of Proposition \ref{P2.4}). 

\begin{Cor}\label{C2.8} 
$$\c V^*(\c L )=\left\{\sum _{0\leqslant i<p}
N^{*i}(f_0)\gamma _i(Y)\,\mathrm{mod}\,\wt{J}\ |\ f_0
\in\Hom _{\widetilde{\uL}_0^*}(\L ,\c R^0)\right\}$$
\end{Cor}

\begin{remark} a) In the above description of 
$\c V^*(\c L)$, for any $l\in L$, $N^*(f_0)(l)=f_0(N(l))$. In addition, 
all $N^{*i}(f_0)\gamma _i(Y)$ depend just on $N_1=N\operatorname{mod}u^pL$. 

b) If $\c L\in\uL ^{*u}$ then  
in the above Corollary we can replace ${\c R}^0$ and 
$\widetilde{J}$ by, respectively,  
${\c R}^u=(R/x_0^pR, x_0^{p-1}R/x_0^pR,
\varphi )\in\widetilde{\uL} ^*_0$ 
and the ideal 
$\widetilde{J}^u=\sum _{0\leqslant i<p}Rx_0^{p-i}\gamma _i(Y)+\Fil ^pR^0_{st}$. 
In particular, for unipotent modules the whole theory can be developed in the context of 
$k[u]/u^p$-modules.   
\end{remark}
\medskip

\subsection{The category $\uC _F$ and the functor $\CV ^*$}\label{S2.3}

\begin{definition} The objects of the category $\uC _F$ are the triples 
$\c H=(H,H^0,j)$, where $H,H^0$ are finite $\Z _p[\Gamma _F]$-modules, 
$\Gamma _F$ acts trivially on $H^0$ and $j:H\longrightarrow H^0$ 
is an epimorphic map of $\Z _p[\Gamma _F]$-modules. 
If $\c H_1=(H_1,H_1^0,j_1)\in\uC _F$ then 
$\Hom _{\uC _F}(\c H_1,\c H)$ consists of the couples 
$(f,f^0)$, where $f:H_1\longrightarrow H$ and 
$f^0:H_1^0\longrightarrow H^0$ are morphisms of 
$\Gamma _F$-modules such that $jf=f^0j_1$. 
\end{definition} 

The category $\uC _F$ is pre-abelian, 
cf. Appendix \ref{A}, and its objects 
have a natural group structure. 
In particular, with above notation, 
$\Ker (f,f^0)=(\Ker f,j_1(\Ker f))$ together with the 
natural embedding to $\c H_1$. Similarly, 
$\Coker (f,f^0)=(H/f(H_1), H^0/j(f(H_1)))$. For example, the map 
$(\id ,0):(H,H)\To (H,0)$ has the trivial kernel and cokernel. In addition, 
the monomorphism 
$(f_1,f_1^0):\c H_1\longrightarrow\c H$ is strict if and only if 
$f_1(\Ker \,j_1)=f_1(H_1)\cap\,\Ker\, j$. Suppose 
$\c H_2=(H_2,H_2^0,j_2)$ and 
$(f_2,f_2^0):\c H\longrightarrow\c H_2$ is an epimorphism. 
Then it is strict if and only if $f_2^0$  induces epimorphic map 
from $\Ker j$ to $\Ker\, j_2$. In $\uC _F$ we can use  
formalism of short exact sequenes and the corresponding 
6-terms $\Hom _{\underline{\CM} _F}-\Ext _{\underline{\CM} _F}$ 
exact sequences, cf. Appendix \ref{A}.

\begin{definition} Suppose $\L\in\uL ^*$ and 
$i^{et}:\L ^{et}\longrightarrow\L$ is 
the maximal etale subobject. Then 
$\CV ^*:\uL^*\longrightarrow\uC _F$ is the functor such that 
$\CV ^*(\L)=(\c V ^*(\L ),\c V^*(\L ^{et}), \c V^*(i^{et}))$.
\end{definition}

The simple objects in $\uC _F$ are of the form either 
$(H,0,0)$, 
where $H$ is a simple $\Z _p[\Gamma _F]$-module, 
or $(\F _p, \F _p,\id )$, where 
$\F _p$ is provided with the trivial $\Gamma _F$-action. 
In this context it will be very convenient to use the following formalism.

For $s\in\N$, consider Serre's fundamental characters 
$\chi _s:\Gamma _F\longrightarrow k^*$. Here for $\tau\in\Gamma _F$, 
$\chi _s(\tau )=(\tau x_s)/x_s\,\mathrm{mod}\,x_0^p$, where 
$x_s\in R$ is such that $x_s^{p^s-1}=x_0$. 
If $\chi $ 
is any continuous (1-dimensional) character of $\Gamma _F$ then 
there are $s,m\in\N $ such that $0<m\leqslant p^s-1$ and 
$\chi =\chi _s^m$. Set $r(\chi )=m/(p^s-1)$. Then 
$r(\chi )$ depends only on $\chi $ and the correspondence 
$\chi\mapsto r(\chi )$ gives 
a bijection from the set of all 
continuous (1-dimensional) 
characters of $\Gamma _F$ with values in $k^*$ to  
the set $[0,1]_p\setminus\{0\}$. 

For $r\in[0,1]_p$, $r\ne 0$, introduce the $\Gamma _F$-module $\F (r)$ 
such that $\F (r)=\F _{p^{s(r)}}$, where $s(r)$ 
is the period of the $p$-digit expansion of $r$, cf. Subsection \ref{S1.2}, 
 with the $\Gamma _F$-action given by 
the character $\chi $ such that $r(\chi )=r$. We have:

--- all $\F (r)$ are simple $\Z _p[\Gamma _F]$-modules;

--- $\Gamma _F$-modules $\F (r_1)$ and $\F (r_2)$ are isomorphic if and 
only if there is an $n\in\Z $ such that $r_1=r_2(n)$;

--- any simple $\Z _p[\Gamma _F]$-module 
is isomorphic to some $\F (r)$.

It will be natural to set $\c F(r):=(\F (r),0,0)$ for 
all $r\in (0,1]_p$, 
and to set separately $\c F(0):=(\F _p, \F _p,\id )$. 

With above notation we have the following property, 
where the objects $\L (r)$ were introduced in Subsection \ref{S1.3}.

\begin{Prop}\label{P2.7} For any $r\in[0,1]_p$, $\CV^*(\L (r))=\c F(r)$. 
\end{Prop}

\begin{proof} 
 The proof goes along the lines of Subsection 4.2 of \cite{refAb1}, 
cf. also the beginning of Subsection \ref{S2.4} below. 
\end{proof}

\subsection{A criterion.} \label{S2.4}

Suppose $\c L_1, \c L_2$ are given in notation of 
Subsection \ref{S1.4}  and $q=p^s$. 
Then for $i=1,2$, $\c V^*(\L _i)=V_i$ are 1-dimensional 
vector spaces over $\F _q$ with $\Gamma _F$-action given by 
the character $\chi _i:\Gamma _F\longrightarrow k^*$ 
such that $r(\chi _i)=r_i$. (Note that $(q-1)r_i\in\Z $ and, 
therefore, $\chi _i(\Gamma _F)\subset\F _q^*$.)
Choose $\pi _s\in\bar F$ such that $\pi _s^{q-1}=-p$. Then 
$F_s=F(\pi _s)$ is a tamely ramified extension of $F$ of 
degree $q-1$ and all points of $V_i$ are defined over $F_s$. 
We can identify $V_i$ with the $\F _p[\Gamma _F]$-module 
$\F _q\bar\pi _s^{(q-1)r_i}\subset \bar O/p\bar O$, where 
$\bar\pi _s=\pi _s\,\mathrm{mod}\,p$. These identifications allow us 
to fix the points $h_i^0:=\bar\pi _s^{(q-1)r_i}\in V_i$ 
and to identify  
$V_i$ with the $\Gamma _F$-module $\{\alpha h_i^0\ |\ \alpha\in\F _q\}$. 

Suppose $h_1\in V_1$. Define the homomorphism 
$$F_{h_1}:\Ext _{\F _p[\Gamma _F]}(V_1,V_2)
\longrightarrow Z^1(\Gamma _{F_s},\F _q)=\Hom (\Gamma _{F_s},\F _q),$$
where $\Gamma _{F_s}=\Gal (\bar F/F_s)$, as follows. 
If $V\in\Ext _{\F _q[\Gamma _F]}(V_1,V_2)$ and 
$h\in V$ is a lift of $h_1$ then for any $\tau\in\Gamma _F$,  
$F_{h_1}(V)(\tau )=a_{\tau }\in\F _q$, where  
$\tau h-h=a_{\tau }h_2^0$. 

Clearly, $F_{h_1}(V)$ does not depend on a choice of $h$ and 
it is the zero function if and only if the projection 
$V\To V_1$ admits a $\Gamma _F$-equivariant section. In other words, 
we have the following criterion. 

\begin{Prop} \label{P2.8} $V$ is the trivial extension if 
and only if  for all 
$h_1\in V_1$, it holds $F_{h_1}(V)=0$. 
\end{Prop}

\subsection{Galois modules $\c V^*(E_{cr}(i_0,j_0,\gamma ))$}\label{S2.5}

Suppose we have an object 
$\L =(L,F(L),\varphi ,N)$ of the category $\uL ^*_{cr}.$  
Then there is a 
special $\sigma (\c W_1)$-basis $l_1,\dots ,l_s$ of 
$\varphi (F(L))$ such that for some integers 
$0\leqslant c_1,\dots ,c_s<p$ and a matrix $A\in\GL _s(k)$, the elements 
$u^{c_1}l_1,\dots ,u^{c_s}l_s$ form a $\W _1$-basis of $F(L)$ and 
$(\varphi (u^{c_1}l_1),\dots ,\varphi (u^{c_s}l_s))=
(l_1,\dots ,l_s)A $. 

 For $1\leqslant i\leqslant s$, set 
$\tilde c_i=(p-1)-c_i$. The following Proposition 
is a special case of Corollary \ref{C2.8} (remind that $R^0=R/x_0^p\m _R$).

\begin{Prop} \label{P2.9} With above notation, 
$\c V^*(\L)$ is the $\F _p[\Gamma _F]$-module 
of all $(\theta _1,\dots ,\theta _s)\,
\mathrm{mod}\,x_0^p\m _R\in (R^0)^s$ such that 
$$(\theta _1^p/x_0^{p\tilde c_1},\dots ,\theta _s^p/x_0^{p\tilde c_s})=
(\theta _1,\dots ,\theta _s)A.$$
\end{Prop}

\begin{remark} In \cite{refAb1,refAb2} it was proved 
(in the context of the Fontaine-Laffaille theory) that the family of 
 $\F _p[\Gamma _F]$-modules $\c V^*(\L )$, 
where $\L\in\uL ^*_{cr}$, coincides with the family of all killed by $p$ 
subquotients of crystalline representations of $\Gamma _F$ with weights 
from $[0,p)$. This result can be also extracted from 
Subsection \ref{S4}, where we establish 
that the family of 
 $\F _p[\Gamma _F]$-modules $\c V^*(\L )$, 
where $\L\in\uL ^*$, coincides with the family of all killed by $p$ 
subquotients of semi-stable representations 
of $\Gamma _F$ with weights from $[0,p)$.
\end{remark}

For an $(r_1,r_2)_{cr}$-admissible pair $(i_0,j_0)\in (\Z /s)^2$ and 
$\gamma\in k$, use the description of $E_{cr}(i_0,j_0,\gamma )$ from 
Subsection \ref{S1.4}. 
Then by Corollary \ref{C2.8}, $V=\c V^*(E_{cr}(i_0,j_0,\gamma ))$ is 
identified with the additive group of all 
taken modulo $x_0^p\m _R$ solutions in $R$ of the following 
system of equations 
\[
\begin{array}{lccc} X_i^{(1)p}/x_0^{pa_i}&=
&X_{i+1}^{(1)}, & \text{for all}\ i\in\Z /s; \\ 
X_j^p/x_0^{pb_j}&=&X_{j+1}-\delta _{jj_0}
\gamma ^pX_{i_0+1}^{(1)}, & \text{for all}\ j\in\Z /s 
\end{array}
\]
Note that the first group of equations describes  
$V_1=\c V^*(\L _1)$ and the correspondences 
$X_i^{(1)}\mapsto 0$ and $X_j\mapsto X_j^{(2)}$ with $i,j\in\Z /s$, 
define the map $V\longrightarrow V_2$, where $V_2=\c V^*(\L _2)$ is associated 
with all taken modulo $x_0^p\m _R$ solutions in $R$ of the equations 
$X_j^{(2)p}/x_0^{pb_j}=X^{(2)}_{j+1}$, $j\in\Z /s$. As it was noted  
in Subsection \ref{S2.2}, 
the corresponding $\Gamma _F$-action on $V,V_1$ and $V_2$ comes from 
the natural $\Gamma _F$-action on $R^0$. 

Take $x_s\in R$ such that $x_s^{q-1}=x_0$ and 
$x_s\mapsto \pi _s\mathrm{mod}\,p$ under the natural identification 
$R/x_0^pR\simeq\bar O/p\bar O$. (This identification is given by the correspondence 
$r=\underset{n}\varprojlim (r_n\,\mathrm{mod}\,p)
\mapsto r^{(1)}:=\underset{n\to\infty }\lim r_{n+1}^{p^n}$.)
For $i,j\in\Z /s$, set 
$x_0^{r_1(i)}:=x_s^{(q-1)r_1(i)}$ and 
$x_0^{r_2(j)}:=x_s^{(q-1)r_2(j)}$, and 
introduce the variables $Z_i^{(1)}=x_0^{-pr_1(i)}X_i^{(1)}$, 
$Z_j=x_0^{-pr_2(j)}X_j$, $Z_j^{(2)}=x_0^{-pr_2(j)}X_j^{(2)}$. 
Then the elements of $V$ appear as the taken modulo $\m _R$ 
solutions in $R_0:=\mathrm{Frac}(R)$ of 
the following system of equations 

$$
\begin{array}{cccc} Z_i^{(1)p}&=&Z_{i+1}^{(1)}, & \text{for all}\ i\in\Z /s;\\ 
Z_j^p&=&Z_{j+1}, & \text{for all}\ j\ne j_0+1;\\
Z_{j_0+1}-Z_{j_0+1}^q&=&\gamma ^pZ_{i_0+1}^{(1)}x_0^{p(r_1(i_0)-r_2(j_0))}&
\end{array} 
$$

Note that for the points $h_1^0\in V_1$ and $h_2^0\in V_2$ 
chosen in Subsection \ref{S2.4}, it holds 
$Z_i^{(1)}(h_1^0)=Z_i^{(2)}(h_2^0)=1$, where $i\in\Z /s$. 

Suppose $\alpha\in\F _q$ and $h_1=\alpha h_1^0\in V_1$. 

Let $\c F_s=k((x_s))\subset R_0=\mathrm{Frac}\,R$. The 
 field-of-norms functor gives a natural embedding of the absolute Galois group 
$\Gamma _{\c F_s}$ of $\c F_s$ into $\Gamma _{F_s}$, 
where $F_s=F(\pi _s)$. Then the restriction 
$F_{h_1}(V)|_{\Gamma _{\c F_s}}$ of the cocycle 
$$\{F_{h_1}(V)(\tau )=
A _{\tau ,\alpha }(i_0,j_0,\gamma )\in\F _q\ |\ \tau\in\Gamma _{F_s}\}$$
from Subsection \ref{S2.4} can be described as follows. 

Let $U\in R_0$ be such that 
$U-U^q=\gamma x_0^{r_1(i_0)-r_2(j_0)}$. 
Then for any $\tau\in\Gamma _{\c F_s}$, 
$\sigma ^{j_0}(A _{\tau ,\alpha }(i_0,j_0,\gamma ))=
\sigma ^{i_0}(\alpha )(\tau (U)-U)$ and therefore 
$$A_{\tau ,\alpha }(i_0,j_0,\gamma )=\sigma ^{i_0-j_0}
(\alpha )\sigma ^{-j_0}(\tau U-U).$$
The following Lemma is an immediate consequence of the definition of 
$(r_1,r_2)_{cr}$-admissible pairs.

\begin{Lem} \label{L2.10} With above notation let 
$C=-(q-1)(r_1(i_0)-r_2(j_0))$. Then 
$C$ is a prime to $p$ integer and $1\leqslant C\leqslant q-1$. 
\end{Lem}

\subsection{Galois modules $\c V^*(E_{st}(i_0,j_0,\gamma ))$}
\label{S2.6}

For an $(r_1,r_2)_{st}$-admissible pair $(i_0,j_0)\in (\Z /s)^2$
and $\gamma\in k$, 
use the description of $E_{st}(i_0,j_0,\gamma )$ from Subsection \ref{S1.5}.

By Subsection \ref{S2.2}, $V=\c V^*(E_{st}(i_0,j_0,\gamma ))$ is identified 
(as an abelian group) 
with the solutions 
$\left (\{X_i^{(1)}\ |\ i\in\Z /s\},\{X_j\ |\ j\in\Z /s\}\right )\in R^{2s}$ 
of the following system of equations 
\begin{equation}\label{E2.2}
\begin{array} {ccc}
X_i^{(1)p}/x_0^{pa_i}&=&X_{i+1}^{(1)}, \ \text{for all}\ i\in\Z /s;\\
X_j^{p}/x_0^{pb_j}+\delta _{jj_0}
\gamma ^pX_{i_0}^{(1)p}/x_0^{pa_{i_0}+p}&
=&X_{j+1}, \ \text{for all}\ j\in\Z /s
\end{array}
\end{equation}

The structure of $V$ as an element of 
$\Ext _{\F _p[\Gamma _F]}(V_1,V_2)$ can be described 
along the lines of Subsection \ref{S2.5}. 
The action of $\Gamma _F$ on $V$ comes from the natural 
$\Gamma _F$-action on 
$\widetilde{\c R}^0_{st}$, 
and the embedding of $V$ into $({R}^0_{st})^{2s}$ 
given by the following correspondences:
\newline 
-- if $i\in\Z /s$ then $X_i^{(1)}
\mapsto X_i^{(1)}\operatorname{mod}x_0^{p}\m _R$;
\newline 
-- if $j\notin\{j_0+1,\dots ,j_0+m_0\}$ then 
$X_j\mapsto X_j\operatorname{mod}x_0^{p}\m _R$;
\newline 
-- for $1\leqslant m\leqslant m_0$, $X_{j_0+m}
\mapsto X_{j_0+m}+\gamma ^{p^m}(\tilde b_{j_0}-
\tilde a_{i_0}+1)X^{(1)}_{i_0+m}Y\operatorname{mod}x_0^{p}\m _R$.
\medskip 

Similarly to Subsection \ref{S2.5}, introduce new 
variables by the relations 
$Z_i^{(1)}=x_0^{-pr_1(i)}X_i^{(1)}$, 
$Z_i=x_0^{-pr_2(i)}X_i$ and $Z_i^{(2)}=x_0^{-pr_2(i)}X_i^{(2)}$, $i\in \Z /s$, 
and rewrite system of equations \eqref{E2.2} in the following form: 
$$ 
\begin{array} {cccc} 
Z_i^{(1)p}&=&Z_{i+1}, & \text{for all}\ i\in\Z /s;\\
Z_j^p\ \ &=&Z_{j+1}, &\text{for all}\ j\ne j_0+1;\\
Z_{j_0+1}-Z_{j_0+1}^q&=&\ \ \ \ \ \gamma ^pZ_{i_0+1}^{(1)}
x_0^{p(r_1(i_0)-r_2(j_0)-1)}&
\end{array}
$$

If $\alpha\in\F _q$ and $h_1=\alpha h_1^0\in V_1$, then the 
restriction to $\Gamma _{\c F_s}$ of the cocycle  
$\{F_{h_1}(V)(\tau )=A _{\tau ,\alpha }(i_0,j_0,\gamma )\ |\ \tau\in\Gamma _{F_s}\}$ 
can be described as follows. 
Let $U\in R_0$ be such that 
$$U-U^q=\gamma x_0^{r_1(i_0)-r_2(j_0)-1}.$$
Then for any $\tau\in\Gamma _{\c F_s}$, 
$\sigma ^{j_0}(A _{\tau ,\alpha }(i_0,j_0,\gamma ))
=\sigma ^{i_0}(\alpha )(\tau U-U)$. Thus 
$$A_{\tau ,\alpha }(i_0,j_0,\gamma )=\sigma ^{i_0-j_0}(\alpha )\sigma ^{-j_0}(\tau U-U).$$

The following Lemma is a direct consequence of the definition of 
$(r_1,r_2)_{st}$-admissible pairs, cf. also Proposition \ref{P1.24}

\begin{Lem} \label{L2.11} Let $C=-(q-1)(r_1(i_0)-r_2(j_0)-1)$. 
Then $C$ is a prime to $p$ 
 integer such that $1\leqslant C<(q-1)(1+1/(p-1))$. 
\end{Lem}
\medskip 

\subsection{Galois modules $E_{sp}(j_0,\gamma )$}
\label{S2.7}

In this subsection $(0,j_0)$ is some $(r_1,r_2)_{sp}$-admissible 
pair (i.e. $r_1+1/(p-1)=r_0(j_0)$) and $\gamma\in\F _q$. 
Then  $V=\c V^*(E_{sp}(j_0,\gamma ))$ is identified as an abelian group 
with the solutions 
$$(\{X_i^{(1)}\ |\ i\in\Z /s\},\{X_j^{(2)}\ |\ j\in\Z /s\})\in R^{2s}$$ 
of the following system of equations 
$$
\begin{array} {ccc} 
 X_i^{(1)p}/x_0^{pa_i}&=&X_{i+1}^{(1)}, \text{for all } i\in\Z /s,\\
X_j^{(2)p}/x_0^{pb_j}&=&X_{j+1}^{(2)}, \text{for all }j\in\Z /s.
\end{array}
$$
The corresponding $\Gamma _F$-action comes 
from the natural $\Gamma _F$-action on 
${\c R}^0_{st}$ and the embedding of $V$ into $(R^0_{st})^{2s}$ 
given by the following correspondences:
\newline 
-- if $i\in\Z /s$ then $X_i^{(1)}\mapsto X_i^{(1)}
\operatorname{mod}x_0^{p}\m _R$;
\newline 
-- if $m\in\Z /s$ then $X_{j_0+m}^{(2)}\mapsto X_{j_0+m}^{(2)}+
\gamma ^{p^m}X_m^{(1)}Y\operatorname{mod}x_0^{p}\m _R$.
\medskip 

If $\alpha\in\F _q$ and $h_1=\alpha h_1^0\in V_1$ then the cocycle 
$$\{F_{h_1}(V)(\tau )=A_{\tau ,\alpha }^{sp}(j_0,\gamma )\ |\ \tau\in\Gamma _{F_s}\}$$ 
can be described as follows. 
Note that the point $h_1$ corresponds to the collection 
$(\{\sigma ^i(\alpha )x _0^{pr_1(i)}\ |\ i\in\Z /s\},
\{\sigma ^{i-j_0}(\alpha\gamma )
x _0^{pr_1(i-j_0)}Y\ |\ i\in\Z /s\})$. 
Then for $\tau\in\Gamma _{F_s}$, $\tau (h_1)$ corresponds to the collection 
$$(\{\sigma ^i(\alpha )x _0^{pr_1(i)}\ |\ i\in\Z /s\},\{\sigma ^{i-j_0}(\alpha\gamma )
x _0^{pr_1(i-j_0)}(Y+k(\tau )\widetilde{\log}\varepsilon )\ |\ i\in\Z /s \}).$$
Therefore, $\tau (h_1)-h_1$ corresponds to the collection 
$$(\{\ 0\ |\ i\in\Z /s\},\{\sigma ^{i-j_0}(\alpha\gamma )
x _0^{pr_2(i)}k(\tau )\ |\ i\in\Z /s \}),$$
which corresponds to $\sigma ^{-j_0}(\alpha\gamma )h_2^0$. Therefore, 
$A_{\tau ,\alpha }^{sp}(j_0,\gamma )=\sigma ^{-j_0}(\alpha \gamma )k(\tau )$. 

Notice that for any $\tau\in\Gamma _{\c F_s}\subset\Gamma _{F_s}$, 
$A^{sp}_{\tau ,\alpha }(j_0,\gamma )=0$.
\medskip

\subsection{Fully faithfulness of $\CV ^*$.} 
\label{S2.8}

\ \

In this subsection we prove the following important property.

\begin{Prop} \label{P2.12} The functor $\CV ^*$ is fully faithful. 
\end{Prop} 

\begin{proof} We must prove that for all $\L _1,\L _2\in\uL ^*$, 
the functor $\CV ^*$ induces a bijective map 
$$\Pi (\L _1,\L _2):\Hom _{\uL ^*}(\L _2,\L _1)
\longrightarrow\Hom _{\uC _F}(\CV ^*(\L_1),\CV ^*(\L _2)).$$

By induction on lengths of composition series for $\L_1$ and $\L_2$ it 
will be sufficient to verify that for any two 
simple objects $\L _1$ and $\L_2$:
\medskip 

$\bullet $\ $\Pi (\L_1,\L_2)$ is bijective;
\medskip 

$\bullet $\ the functor $\CV^*$ induces injective map 
$$\mathrm{E\Pi }\,(\L_1,\L_2):\Ext _{\uL^*}(\L_2,\L_1)
\longrightarrow\Ext _{\uC _F}(\CV^* (\L_1),\CV ^*(\L_2)).$$
\medskip 

The first fact has been already checked in Subsection \ref{S2.3}. 

In order to 
verify the second property, notice that for 
any two objects $\L _1,\L _2\in\uL ^*$, the natural map 
$$\Ext _{\uC _F}(\CV ^*(\L_1), \CV ^*(\L _2))
\longrightarrow \Ext _{\uM _F}(\c V ^*(\L _1)),\c V ^*(\L _2))$$
is injective. Therefore, we can prove injectivity of 
$\mathrm{E\Pi }(\L_1,\L_2)$ on the level of functor $\c V^*$. 
In addition, 
for $n_1,n_2\in\N $, 
$\Ext _{\uL^*}(\L ^{n_1}_2,\L ^{n_2}_1)=\Ext _{\uL ^*}(\L _2,\L _1)^{n_1n_2}$ 
(the formation of $\Ext $ is compatible with direct sums). 
So, by Lemma \ref{L1.17}, we can replace $\L_1$ and $\L_2$ 
by the objects introduced in Subsection \ref{S1.5} (where they are denoted also by 
$\c L_1$ and $\c L_2$). 

By Proposition \ref{P1.27}, 
any element of $\Ext _{\uL^*}(\L_2,\L_1)$ 
appears as a sum of standard extensions of the form $E_{cr}(i,j,\gamma _{ij})$, 
$E_{st}(i,j,\gamma _{ij})$ and $E_{sp}(j,\gamma ^{sp}_{j})$. 
Here: 
a) 
$(i,j)\in (\Z /s)^2$ is either $(r_1,r_2)_{cr}$-admissible or 
$(r_1,r_2)_{st}$-admissible and all $\gamma _{ij}\in k$; 
b)  
$j\in\Z /s$ is such that $(0,j)$ is $(r_1,r_2)_{sp}$-admissible  
and $\gamma _{j}^{sp}\in\F _q$. 

\begin{remark} A couple $(i,j)$ can't be both $(r_1,r_2)_{cr}$-admissible and 
$(r_1,r_2)_{st}$-admissible, but it can be 
$(r_1,r_2)_{cr}$-admissible and 
$(r_1,r_2)_{sp}$-admissible at the same time.
\end{remark} 

By Subsections \ref{S2.5}\,-\ref{S2.7}, we can attach 
to these standard extensions the 
1-cocycles $A_{\tau ,\alpha }(i,j,\gamma _{ij})$ and 
$A^{sp}_{\tau ,\alpha }(j,\gamma _j^{sp})$, 
where $\tau\in\Gamma _{F_s}$. It remains to prove that the sum 
of these cocycles is trivial only if all corresponding coefficients 
$\gamma _{ij}$ and $\gamma ^{sp}_{j}$ are equal to 0. 

First, we need the following lemma.

\begin{Lem} \label{L2.13} Suppose for all $(i,j)\in (\Z /s)^2$, 
the elements $U_{ij}\in R_0=\,\mathrm{Frac}\,R$ are such that 
$U_{ij}-U_{ij}^q=\gamma _{ij}x_s^{-C_{ij}}$, 
where all $\gamma _{ij}\in k$ and all 
$C_{ij}$ are prime to $p$ natural numbers. For $\tau\in\Gamma _{\c F_s}$, let 
$B _{\tau }(i,j,\gamma _{ij})=\tau (U_{ij})-U_{ij}\in\F _q$. If for 
all $\alpha\in\F _q$ and all $\tau\in\Gamma _{\c F_s}$,
\begin{equation}\label{E2.3} 
\sum _{i,j,\in\Z /s}\sigma ^{i-j}(\alpha )\sigma ^{-j}B_{\tau }(i,j,\gamma _{ij})=0
\end{equation} 
then all $\gamma _{ij}=0$.
\end{Lem}

\begin{proof} [Proof of Lemma] For different 
prime to $p$ natural numbers $C_{ij}$ the extensions $\c F_s(U_{ij})$ 
behave independently. Therefore, we can assume that all $C_{ij}=C$ are the same. 

Let $j_0=j_0(j)$ be such that $0\leqslant j_0<s$ and 
$j_0\equiv -j\,\mathrm{mod}\,s$.
Then \eqref{E2.3} means that for any $\alpha\in\F _q$, 
$$B _{\alpha }:=\sum_{i,j\in\Z /s} 
\sigma ^{i-j}(\alpha )\sigma ^{j_0}(U_{ij})\in\c F_s.$$

Then 
$$B_{\alpha }-B_{\alpha }^q=\sum _{j\in\Z /s}
\left (\sum_{i\in\Z /s}\sigma ^{i-j}(\alpha )
\gamma _{ij}^{p^{-j}}\right )x_s^{-p^{j _0}C}.$$
Looking at the Laurent series of $B _{\alpha }\in\c F_s$ 
we conclude that all $B _{\alpha }\in\F _q$. 
This means that for all $j\in\Z /s$ and $\alpha\in\F _q$, 
$\sum _{i\in\Z /s}\sigma ^i(\alpha )\gamma _{ij}=0$ and, therefore, 
all $\gamma _{ij}=0$. The lemma is proved
\end{proof}

Now suppose that for all $\alpha\in\F _q$ and 
 $\tau\in\Gamma _{F_s}$, 
the sum of cocycles $A_{\tau ,\alpha }(i,j,\gamma _{ij})$ 
and $A^{sp}_{\tau ,\alpha }(j,\gamma ^{sp}_{j})$ is zero. 
Restrict this sum to the subgroup $\Gamma _{\c F_s}$. Then all 
$sp$-terms will disappear and by above Lemma \ref{L2.13} all $\gamma _{ij}=0$. 
So, for all $\tau\in\Gamma _{F_s}$ and $\alpha\in\F _q$, 
$\sum _{j\in\Z /s}\sigma ^{-j}(\alpha\gamma _{j}^{sp})=0$, 
and this implies that all $\gamma _{j}^{sp}=0$.
\end{proof}

\begin{Cor}\label{C2.14} The functor 
 $\c V^*$ is fully faithful on the subcategories of 
unipotent objects $\uL ^{*u}$ and 
of connected objects $\uL ^{*c}$.
\end{Cor}

\begin{proof} Indeed, on both categories 
the map $\mathrm{\Pi }(\L_1,\L_2)$ 
is already bijective on the level of functor $\c V^*$. 
\end{proof}

\subsection{Ramification estimates}\label{S2.9}

Suppose $\L\in\uL ^*$ and $H=\c V^*(\L )$. 
For any rational number $v\geqslant 0$, denote by $\Gamma _F^{(v)}$ 
the ramification subgroup of $\Gamma _F$ in upper numbering, \cite{refSer}.

\begin{Prop} \label{P2.15}If $v>2-\frac{1}{p}$ then 
$\Gamma _F^{(v)}$ acts trivially on $H$. 
\end{Prop}

 A proof can be obtained along the lines of the paper \cite{refHat} 
(which adjusts Fontaine's approach from \cite{refFo1}). Alternatively, one can apply  
author's method from \cite{refAb3}: if $\tau\in\Gamma ^{(v)}$ 
with $v>2-1/p$ then there is an automorphism $\psi $ of $R$ such that 
$\psi (x_0)=\tau (x_0)$ and $\psi $ induces the trivial action on 
$H$; therefore we can assume that $\tau $ comes from 
the absolute Galois group of $k((x_0))$ and the characteristic $p$ 
approach from \cite{refAb3} 
gives the ramification estimate which coincides with the required 
by the theory of field-of-norms. 

\begin{Cor} \label{C2.16} If $\widetilde{F}$ is the common field-of-definition 
 of points of $\F _p[\Gamma _F]$-modules $\c V(\L )$ for all 
$\L\in\uL^*$, then $v_p(\c D(\widetilde{F}/F))<3-\frac{1}{p}$, where 
$\c D(\widetilde{F}/F)$ is the different of the field extension $\widetilde{F}/F$.
\end{Cor}
\medskip

\section{ Semistable representations with weights from $[0,p)$ 
and filtered $\c W$-modules} \label{S3}

\subsection {The ring $S$ }\label{S3.1}

Let $v=u+p\in\c W$ and let 
$S$ be the $p$-adic closure of the divided power envelope of 
$\c W$ with respect to the ideal generated by $v$. 
Use the same symbols $\sigma $ and $N$ for  natural continuous 
extensions of $\sigma $ and $N$ from $\c W$ to $S$. For $i\geqslant
0$, denote by $\Fil ^iS$ the $i$-th divided power of the ideal $(v)$ in $S$. 
Then for $0\leqslant i<p$, there are $\sigma $-linear morphisms 
$\phi _i=\sigma /p^i:\Fil ^iS\longrightarrow S$. Note that  
$\phi _0=\sigma $ and agree to use the notation $\varphi $
for $\phi _{p-1}$. 
One can see also that 
$S$ is the $p$-adic closure of 
$W(k)[v_0,v_1,\dots ,v_n,\dots ]$, where $v_0=v$ and for all 
$n\geqslant 0$, $v_{n+1}^p/p=v_n$. 

Consider the ideals $\m _S=(p,v,v_1,\dots ,v_n,\dots )$,
$I=(p,v_1,v_2,\dots )$ and $J=(p,v_1v,v_2,\dots ,v_n,\dots )$ of
$S$. Then

--- $\m _S$ is the maximal ideal in $S$;

--- $I=\Fil ^pS+pS\supset J$;

--- $\varphi (I)\subset S$ and $\varphi (J)\subset pS$;

--- $\varphi (v^{p-1})\equiv 1-v_1(\mathrm{mod}\,J)$ and 
$\varphi (v_1)\equiv 1(\mathrm{mod}\,J)$. 
\subsection {The ring of semi-stable periods $\hat A_{st}$} 
\label{S3.2}

Let $R$ be Fontaine's ring and 
let $x_0,\varepsilon \in R$ be the elements chosen in Subsection 
\ref{S2.1}.

Denote by $A_{cr}$ the Fontaine crystalline ring. 
It is the $p$-adic closure 
of the divided power envelope of $W(R)$ with respect to the ideal
$([x_0]+p)$ of $W(R)$, where $[x_0]\in W(R)$ is the Teichm\"uller
representative of $x_0$. Then for $i\geqslant 0$, 
$\Fil ^iA_{cr}$ is the $i$-th divided power of the ideal $([x_0]+p)$ 
in $A_{cr}$. 
Denote by $\sigma :A_{cr}\longrightarrow A_{cr}$  the natural 
morphism induced by the $p$-th power on $R$. Then for 
$0\leqslant i<p$, there are $\sigma $-linear maps 
$\phi _i=\sigma /p^i:\Fil ^iA_{cr}\longrightarrow A_{cr}$. 
We shall often use the simpler
notation $\varphi =\phi _{p-1}$ and $F(A_{cr})=\Fil ^{p-1}A_{cr}$. Notice that 
$A_{cr}$ is provided with the natural continuous $\Gamma _F$-action. 

Let $X$ be an indeterminate. Then $\hat A_{\st }$ is the $p$-adic
closure of the ring 
$A_{cr}[\gamma _i(X)\ |\ i\geqslant 0]\subset 
A_{cr}[X]\otimes _{\Z _p}\Q _p$, where 
for all $i\geqslant 0$, $\gamma _i(X)=X^i/i!$. The ring 
$\hat A_{st}$ has the
following additional structures:

$\bullet $\  the $S$-module structure given by the natural 
$W(k)$-algebra structure and the correspondence $u\mapsto [x_0]/(1+X)$; 

$\bullet $\ the ring endomorphism $\sigma $, 
which is the extension of the above defined endomorphism 
$\sigma $ of $A_{cr }$ via 
the condition $\sigma (X)=(1+X)^p-1$;

$\bullet $\  the continuous $A_{cr}$-derivation 
$N:\hat A_{\st}\longrightarrow \hat A_{\st }$ such that $N(X)=X+1$;

$\bullet $\  for any $i\geqslant 0$, the ideal 
$\Fil ^i\hat A_{\st }$, which is the closure of the ideal 
$\sum _{i_1+i_2\geqslant i}
\left (\Fil ^{i_1}A_{cr }\right )\gamma _{i_2}(X)$;

$\bullet $\  the action of $\Gamma _F$, which is the extension of 
the $\Gamma _F$-action on $A_{cr}$ such that for all
$\tau\in\Gamma _F$, $\tau (X)=[\varepsilon ]^{k(\tau )}(X+1)-1$. Here 
all $k(\tau )\in\mathbb{Z} _p$ are such that $\tau (x_0)=\varepsilon
^{k(\tau )}x_0$. 
\medskip 

Note that for $0\leqslant m<p$, 
$\sigma (\Fil ^m\hat A_{st})\subset p^m\hat A_{st}$ 
and, as earlier, we can set 
$\phi _m=p^{-m}\sigma |_{\Fil ^m\hat A_{st}}$ 
and introduce the simpler notation 
$\varphi =\phi _{p-1}$ and $F(\hat A_{st})=\Fil ^{p-1}\hat A_{st}$.
\medskip 

\subsection{Construction of semi-stable representations of
$\Gamma _F$ with weights from $[0,p)$} \label{S3.3}

For $0\leqslant m< p$, consider the category $\widetilde{\c S}_m$ 
of quadruples $\c M=(M,\Fil ^mM,\phi _m ,N)$, where 
$\Fil ^mM\subset M$ are $S$-modules, $\phi _m:\Fil ^mM\longrightarrow M$ 
is a $\sigma $-linear map and $N:M\longrightarrow M$ is 
a $W(k)$-linear endomorphism 
such that for any $s\in S$ and $m\in M$, $N(sx)=N(s)x+sN(x)$ 
The morphisms of the category $\widetilde{\c S}_m$ are $S$-linear 
morphisms of filtered modules commuting with the corresponding
morphisms $\phi _m$ and $N$. 
Notice that for $0\leqslant m<p$, 
$\hat A_{\st }$ has a natural structure of the object of
the category $\widetilde{\c S}_m$. As earlier, we shall use the simpler notation 
$\varphi =\phi _{p-1}$ and $F(M)=\Fil ^{p-1}M$. 

For $0\leqslant m<p$, the Breuil category $\c S _m$ of strongly divisible 
$S$-modules of 
weight $\leqslant m$  is a full subcategory of 
$\widetilde{\c S}_m$ consisting of the objects 
$\c M=(M,\Fil ^mM,\phi _m,N)$ 
such that 
\medskip 

(1) $M$ is a free $S$-module of finite rank;

(2) $(\Fil ^{m}S)M\subset \Fil ^mM$;

(3) $(\Fil ^mM)\cap pM=p\Fil ^mM$;

(4) $\phi _m (\Fil ^mM)$ spans $M$ over $S$;

(5) $N\phi _m =p\phi _m N$;

(6) $(\Fil ^1S)N(\Fil ^{m}M)\subset\Fil ^mM $.
\medskip

For $\c M\in\c S _m$, let $T^*_{st}(\c M)$ be 
the $\Gamma _F$-module of all $S$-linear and 
commuting with $\phi _m $ and $N$, 
maps 
$f:M\longrightarrow\hat A_{\st }$ such that 
$f(\Fil ^mM)\subset \Fil ^{m}\hat A_{\st }$. Then one has the
following two basic facts:
\medskip 

$\bullet $ 
$T^*_{\st }(\c M)$ is a continuous $\Z _p[\Gamma _F]$-module 
without $p$-torsion, its $\Z _p$-rank equals 
$\mathrm{rk}\,_SM$, and 
$V^*_{\st }(\c M)=T^*_{\st}(\c M)\otimes _{\Z _p}\Q _p$  
is semi-stable $\Gamma _F$-module 
with Hodge-Tate weights from $[0,m]$;
\medskip 

$\bullet $ any semi-stable representation of $\Gamma _F$  with 
Hodge-Tate weights from $[0,m]$, $0\leqslant m<p$, appears in the form 
$V^*_{\st }(\c M)$ for a suitable $\c M\in\c S_m$. 
\medskip

By Theorem 1.3 \cite{refBr2} these facts follow from the  
existence of strongly divisible lattices in $S\otimes _{\c W}F$-modules associated with 
weakly admissible $(\phi _0,N)$-modules with filtration of length $m$. 
Breuil proved this for all $m\leqslant p-2$ but his method can 
be easily extended to cover the case $m=p-1$ as well, cf. also \cite{refBr3}.  
\medskip 
\medskip

\subsection{The category $\uL ^f$}
\label{S3.4}

In this section we introduce $\c W$-analogues of Breuil's 
$S$-modules from the category $\c S _{p-1}$ 
and prove that they can be also used to construct 
semi-stable representations of $\Gamma _F$ with Hodge-Tate weights 
from $[0,p)$. 
\medskip 

\begin{definition} Let $\widetilde{\uL}$ be the category of 
$\L =(L,F(L),\varphi , N_S)$, where $L\supset F(L)$ are 
$\c W$-modules, 
 $\varphi :F(L)\longrightarrow L$ 
is a $\sigma $-linear 
morphism of $\c W$-modules and 
$N_S:L\longrightarrow L_S:=L\otimes _{\c W}S$ is such 
that for all $w\in\c W$ and $l\in L$, 
$N_S(wl)=N(w)l+(w\otimes 1)N_S(l)$. 
For $\L _1=(L_1,F(L_1),\varphi ,N_S)\in \widetilde{\uL}$, the morphisms 
$\Hom _{\widetilde{\uL}}(\L ,\L _1)$ are  
$\c W$-linear $f:L\longrightarrow L_1$ such that 
$f(F(L))\subset F(L_1)$, $f\varphi =\varphi f$ and $fN_S=N_S(f\otimes 1)$. 
\end{definition}
\medskip 

 Let $\c A_{st}=(\hat A_{st}, F(\hat A_{st}), \varphi ,N_S)$, 
where $N_S=N\otimes 1$. 
Then $\c A_{st}$ is an object of the category $\widetilde{\uL}$.
\medskip 

Suppose $\L =(L,F(L),\varphi ,N_S)\in\widetilde{\uL }$. 

Set $L_S:=L\otimes _{\c W}S$, $F(L_S)=(F(L)\otimes 1)S+(L\otimes 1)\Fil ^pS$, 
and $\varphi _S:F(L_S)\longrightarrow F(L_S)$ is 
a unique 
$\sigma $-linear map such that $\varphi _S|_{F(L)\otimes 1}=\varphi \otimes 1$ and for any $s\in\Fil ^pS$ and $l\in L$, 
$\varphi _S(l\otimes s)=(\varphi (v^{p-1}l)\otimes 1)\varphi (s)/\varphi (v^{p-1})$.
\medskip 

\begin{definition} Denote by $\uL ^f$ the full subcategory in $\widetilde{\uL}$ 
consisting of the quadruples $\L =(L,F(L),\varphi ,N_S)$ such that 

$\bullet $\ $L$ is a free $\c W$-module of finite rank;

$\bullet $\ $v^{p-1}L\subset F(L)$, $F(L)\cap pL=pF(L)$ and 
$L=\varphi (F(L))\otimes _{\sigma\c W}\c W$;

$\bullet $\ for any $l\in F(L)$, $vN_S(l)\in 
F(L_S)$  
and $\varphi _S(vN(l))=cN_S(\varphi (l))$, 
where $c=1+u^p/p$. 
\end{definition}
\medskip 

It can be easily seen that for $\L =(L,F(L),\varphi ,N_S)\in\uL ^f$ and the map 
$N=N_S\otimes 1:L_S\longrightarrow L_S$, the quadruple 
$\L _S=(L_S,F(L_S),\varphi _S,N)$ is the object of the category $\c S_{p-1}$

The main result of this Subsection is the following statement.

\begin{Prop}\label{P3.1} For any  $\c M=(M,F(M),\varphi ,N)\in\c S_{p-1}$, 
there is an $\L =(L,F(L),\varphi ,N_S)\in\uL ^f$ such that 
$\c M=\L _S$. 
\end{Prop}

\begin{Cor}\label{C3.2} a) If $\L\in\uL ^f$ and $T^*_{st}(\L )=
\Hom _{\widetilde{\uL}}(\L ,\hat A_{st})$ with the induced structure 
of $\Z _p[\Gamma _F]$-module then 
$V^*_{st}(\L )=T^*_{st}(\L )\otimes _{\Z _p}\Q _p$ is a semi-stable 
$\Q _p[\Gamma _F]$-module with Hodge-Tate weights from $[0,p)$ and $\dim _{\Q _p}V^*_{st}(\L )=\mathrm{rk}_{\c W}L$.

b) For any semi-stable $\Q _p[\Gamma _F]$-module $V^*_{st}$ 
with Hodge-Tate weights from $[0,p)$, there is 
an $\L\in\uL ^f$ such that $V^*_{st}\simeq V^*_{st}(\L )$.
\end{Cor}

\begin{proof} [Proof of Proposition \ref{P3.1}]

Let $d$ be a rank of $M$ over $S$. If $L\subset M$ is 
a free $\c W$-submodule of rank $d$ and $M$ is generated 
by the elements of $L$ over $S$ we say 
that $L$ is $\c W$-structural (with respect to $M$). 

Let $F(L)=F(M)\cap L$.

\begin{Lem}\label{L3.3}  If $L$ is $\c W$-structural 
for $M$ then 
\medskip 

{\rm a)} $F(L)\supset v^{p-1}L$;
\medskip 

{\rm b)} $F(L)\cap pL=pF(L)$;
\medskip 

{\rm c)} $F(L)$ is a free 
$\c W$-module of rank $d$.
\end{Lem} 

\begin{proof} 

a) $v^{p-1}L\subset (\Fil ^{p-1}S)M\cap L\subset F(M)\cap L=F(L)$.
\medskip 

b) $F(L)\cap pL=L\cap F(M)\cap pL=F(M)\cap pL=
F(M)\cap pM\cap pL=pF(M)\cap pL=pF(L)$.
\medskip 

c) $F(L)$ has no $p$-torsion. Therefore, it will be sufficient to prove that 
$F(L)/pF(L)$ is a free $k[[u]]$-module of rank $d$. 
Consider the following natural embeddings 
of $k[[v]]$-modules
$$L/pL\supset F(L)/pF(L)\supset v^{p-1}L/pv^{p-1}L\simeq L/pL$$
(Use b) and that $pL\cap v^{p-1}L=pv^{p-1}L$.) It remains to note that 
$L/pL$ is free of rank $d$ over $k[[v]]$. 

The Lemma is proved.
\end{proof}

 Suppose $L$ is $\c W$-structural for $M$. 

\begin{Lem} \label{L3.4} 
If $L$ is $\c W$-structural then $\varphi (F(L))$ spans $M$ over $S$. 
\end{Lem} 

\begin{proof} The equality $S=\c W+\Fil ^pS$ implies that 
 $M=L+(\Fil ^pS)L=L+(\Fil ^pS)M$. Therefore, 
$$F(M)=F(M)\cap L+(\Fil ^pS)M=F(L)+(\Fil ^pS)L$$
(use that $F(M)\supset (\Fil ^pS)M$) and in notation of Subsection 
\ref{S3.1} it holds 
$$F(M)=F(L)+v_1L+JM.$$
This implies that $\varphi (F(L))$, 
$\varphi (v_1L)$ and $\varphi (JM)$ span $M$ over $S$. 
But for any $l\in L$, 
$\varphi (v_1l)=\varphi (v_1)\varphi (v^{p-1}l)/\varphi (v^{p-1})
=(1-v_1)^{-1}\varphi (v^{p-1}l)\equiv \varphi (v^{p-1}l)\mathrm{mod}\,\m _SM$.
For similar reasons, $\varphi (JM)\subset pM\subset \m _SM$. This means that 
$\varphi (F(L))$ spans $M$ modulo $\m _SM$. The lemma is proved.
\end{proof}

By above lemma it remains to prove the existence of 
a $\c W$-structural $L$ for $M$ such that $\varphi (F(L))\subset L$. 

 Let $\phi _0$ be a $\sigma $-linear endomorphism of the 
$S$-module $M\in\c S_{p-1}$ such that for all $m\in M$, 
$\phi _0(m)=\varphi (v^{p-1}m)/\varphi (v^{p-1})$. Clearly, 
$\phi _0 (\m _SM)\subset\m _SM$ and, therefore, 
it induces a $\sigma $-linear endomorphism $\sigma _0$ 
of the $k$-vector space $M_k=M/\m_SM$. 
\medskip

\begin{Lem} \label{L3.5} Suppose $n\in\Z _{\geqslant 0}$, $L$ is $\c W$-structural and 
$\varphi (F(L))\subset L+p^nM$. Then there is a $\c W$-structural 
$L'$ for $M$ such that $\varphi (F(L'))\subset L'+p^nJM$.
\end{Lem}

\begin{proof} Denote by $F(L)_k$ the image of $F(L)$ 
in the $k$-vector space 
\linebreak 
$M/\m _SM=L/(\m _S\cap\c W)L=L_k$. 
Let $s=\dim _kF(L)_k$, then $s\leqslant d=\dim _kL_k$. Choose a 
$\c W$-basis $e^{(1)},\dots ,e^{(d)}$ of $L$ and a $\c W$-basis 
$f^{(1)},\dots ,f^{(d)}$ of $F(L)$ such that 
\medskip 

$\bullet $\  for $1\leqslant i\leqslant s$, $f^{(i)}=e^{(i)}$ and for 
$s<i\leqslant d$, $f^{(i)}\in vL$. 
\medskip 

It will be convenient to use 
the following vector notation: $\bar e=(\bar e_1,\bar e_2)$, where 
$\bar e_1=(e^{(1)},\dots ,e^{(s)})$ and 
$\bar e_2=(e^{(s+1)},\dots ,e^{(d)})$, and $\bar f=(\bar f_1,\bar f_2)$, where 
$\bar f_1=\bar e_1$ and $\bar f_2=(f^{(s+1)},\dots ,f^{(d)})$. 

Then in obvious notation it holds 
$(\varphi (\bar f_1),\varphi (\bar f_2))=(\bar e_1,\bar e_2)C$, 
where $C\in\GL _d(S)$. Clearly, $C\equiv C_0+p^nv_1C_1\,\mathrm{mod}\,p^nJ$ with 
$C_0\in\GL _d(\c W)$ and $C_1\in \mathrm{M}_d(\c W)$. 
Clearly, $\varphi (F(L))\subset L+p^nJM$ iff 
$C_1\equiv 0\,\mathrm{mod}\,\m _S$. 
Choose $\bar g=(\bar g_1,\bar g_2)\in L^d$ and set 
$$\bar e_1'=(e^{\prime (1)},\dots ,e^{\prime (s) })
=\bar e_1+p^n(v_1-v^{p-1})\bar g_1$$ 
$$\bar e_2'=(e^{\prime (s+1)},\dots ,e^{\prime (d)})
=\bar e_2+p^n(v_1-v^{p-1})\bar g_2$$

Clearly, the coordinates of $\bar e'=(\bar e_1',\bar e_2')$ 
give an $S$-basis of $M$ and we can introduce 
the structural $\c W$-module $L'=\sum _i\c We^{\prime (i)}$ for $M$. 

Prove that the elements $e^{\prime (i)}$, $1\leqslant i\leqslant s$, and 
$f^{(i)}, s<i\leqslant d$, generate $F(L')\,\mathrm{mod}\,p^nJM$. 
Indeed, we have 
\begin{equation}\label{E3.1} L+p^nIM=L'+p^nIM
\end{equation}
and this implies that the image $F(L)_k$ of $F(L)$ in $L_k$ 
coincides with its analogue $F(L')_k$. 
In addition, for $1\leqslant i\leqslant s$, 
$$e^{\prime (i)}\in L'\cap (F(L)+p^nIM)\subset L'\cap F(M)=F(L').$$
Therefore, it would be sufficient to prove that 
$(vL')\cap F(L')\,\mathrm{mod}\,p^nJM$ is generated by 
the images of $ve^{\prime (i)}$, $1\leqslant i\leqslant s$, and 
$f^{(s+1)},\dots ,f^{(d)}$. But relation \eqref{E3.1} implies that 
$vL+p^nJM=vL'+p^nJM$ and 
$$(vL')\cap F(L')\,\mathrm{mod}\,p^nJM=(vL)\cap F(L)\,\mathrm{mod}\,p^nJM.$$
It remains to note that for $1\leqslant i\leqslant s$, 
$ve^{\prime (i)}\equiv ve^{(i)}\,\mathrm{mod}\,p^nJM$.

Therefore, we can define special bases for $L'$ and $F(L')$ 
by the relations  
 $\bar f_1'=\bar e_1'$ and $\bar f_2'=\bar f_2$
and obtain that  
$$(\varphi (\bar f_1'),\varphi (\bar f_2'))=(\varphi (\bar f_1),
\varphi (\bar f_2))+p^nv_1(\sigma _0\bar g_1,\bar 0)\,\mathrm{mod}\,p^{n}JM$$
and 
$$(\varphi (\bar f_1'), \varphi (\bar f_2'))\equiv 
(\bar e_1',\bar e_2')C_0+p^nv^{p-1}(\bar g_1,\bar g_2)C_0+\qquad\qquad\qquad $$
$$\qquad\qquad +p^nv_1((\bar e_1,\bar e_2)C_1 -(\bar g_1,\bar g_2)C_0+
(\sigma \bar g_1,\bar 0))\,\mathrm{mod}\,p^nJM$$

So, $\varphi (F(L'))\subset L'+p^nJM$ if and only if 
there is an $\bar g=(\bar g_1,\bar g_2)\in L^d$ such that 
$(\sigma _0\bar g_1,\bar 0)\equiv (\bar g_1,\bar g_2)
C_0+\bar h\,\mathrm{mod}\,(\m _S\cap\c W)L$, 
where $\bar h=(\bar e_1,\bar e_2)C_1\in L$ 
and $C_0\,\mathrm{mod}\,\m _S\in\GL _d(k)$. 
The existence of such vector $\bar g$ is implied by Lemma 
\ref{L3.6} below. 
\end{proof} 

\begin{Lem} \label{L3.6} Suppose $V$ is a $d$-dimensional  
vector space over $k$ with a $\sigma $-linear  endomorphism 
$\sigma _0:V\longrightarrow V$ and $\bar a=(\bar a_1,\bar a_2)\in V^d$, where 
$\bar a_1\in V^s$ and $\bar a_2\in V^{d-s}$. Then 
for any $C\in\GL _d(k)$ there is an 
$\bar g=(\bar g_1,\bar g_2)\in V^d$ with 
$\bar g_1\in V^s$ and $\bar g_2\in V^{d-s}$ 
such that 
\begin{equation}\label{E3.2} (\sigma _0\bar g_1,\bar 0)=\bar gC+\bar a. 
\end{equation} 
\end{Lem} 

\begin{proof} 

Let $C^{-1}=\left (\begin{matrix} D_{11} & D_{12} \cr
D_{21} & D_{22} \end{matrix}\right )$ 
with the block matrices of sizes $s\times s$, 
$(d-s)\times s$, $s\times (d-s)$ and $(d-s)\times (d-s)$. 
Then the equality \eqref{E3.2} can be rewritten as 
$$\begin{array}{lcl}  (\sigma _0\bar g_1)D_{11} & = & \bar g_1+\bar a_1' \\ 
        (\sigma _0\bar g_1)D_{21}         &= & \bar g_2+\bar a_2' 
\end{array} $$
where $(\bar a_1', \bar a_2')=\bar aC^{-1}$. Clearly, 
it will be sufficient to solve the first equation in $\bar g_1$, 
but this is a special case of Lemma \ref{L1.1}. 
\end{proof}

\begin{Lem} \label{L3.7} Suppose $n\geqslant 0$ and 
$L$ is $\c W$-structural for $M$ such that 
$\varphi (F(L))\subset L+p^nJM$. Then there is a 
$\c W$-structural $L'$ for $M$ such that 
$\varphi (F(L'))\subset L'+p^{n+1}M$. 
\end{Lem} 

\begin{proof} Suppose the coordinates of 
$\bar e\in M^d$ form a $\c W$-basis of $L$ and $D\in\M _d(\c W)$ is such that 
the coordinates of $\bar f=\bar eD$ form a $\c W$-basis 
of $F(L)$.  Then $\varphi (\bar f)=\bar e+p^n\bar h$, where 
$\bar h\equiv \bar 0\operatorname{mod}JM$. 
Let $\bar e'=\bar e+p^n\bar h$ and let $L'$ be a $\c W$-submodule in $M$ spanned by 
the coordinates of $\bar e'$. Clearly, $L'$ is  $\c W$-structural. 

Prove that $F(L')$ is spanned by the coordinates of $\bar e'D$. 
Indeed, suppose $\bar e$ and $\bar e'$ have the coordinates 
$e^{(i)}$ and, resp., $e^{\prime (i)}$, $1\leqslant i\leqslant s$. 
Then for all $i$, 
$e^{\prime (i)}=e^{(i)}+p^nh^{(i)}$, where 
$h^{(i)}\in JM\subset (\Fil ^pS)M$. This means that a 
$\c W$-linear combination of $e^{(i)}$ belongs to $F(M)$ 
if and only if the same linear combination of $e^{\prime (i)}$ 
belongs to $F(M)$. This implies that $\bar e'D$ spans $F(L')$ 
over $\c W$  because $\bar eD$ spans $F(L)$ over $\c W$.
Then $\varphi (F(L'))\subset L'+p^{n+1}M$ 
because $\varphi (\bar h)\in pM$ (use that $\varphi (J)\subset pS$) and 
$$\varphi (\bar e'D)=
\varphi (\bar eD+p^n\bar hD)=\bar e+p^n\bar h+p^n\varphi (\bar h)\sigma (D)\equiv 
\bar e'\mathrm{mod}p^{n+1}M$$
\end{proof} 
It remains to notice that applying above Lemmas 
\ref{L3.6} and \ref{L3.7} one 
after another we shall obtain a sequence of $\c W$-structural 
modules $L_n$ such that for all $n\geqslant 0$, 
$L_n+p^{n+1}M=L_{n+1}+p^{n+1}M$, where $L_0\otimes _{\c W}S=M$. 
Therefore, $L=\underset{n}\varprojlim L_n/p^n$ is 
$\c W$-structural and $\varphi (L)\subset L$. 

The proposition is completely proved.
\end{proof}

\subsection{The categories $\uL ^t$ and $\uL ^{ft}$}\label{S3.5}\ 
\medskip

\begin{definition} 
 $\c W$-module $L$ is $p$-strict if it is isomorphic 
to $\oplus _{1\leqslant i\leqslant s}\c W/p^{n_i}$, where $n_1,\dots ,n_s\in\N $. 
\end{definition}

In particular, if $L$ is $p$-strict and $pL=0$ then $L$ is a 
free $\c W_1$-module. 
The $p$-strict modules can be efficiently studied via devissage due 
to the following property.

\begin{Lem} \label{L3.8} 
$L$ is $p$-strict if and only if $pL$ and $L/pL$ are $p$-strict.
\end{Lem} 

\begin{proof} 
Specify Breuil's proof of a similar statement but for more 
complicated ring $S=\c W ^{DP}$ from \cite{refBr2}.
\end{proof}

\begin{definition} Denote by $\uL ^t$ the full subcategory in $\widetilde{\uL}$ 
consisting of the quadruples $\L =(L,F(L),\varphi ,N_S)$ such that 

$\bullet $\ $L$ is $p$-strict;

$\bullet $\ $v^{p-1}L\subset F(L)$, $F(L)\cap pL=pF(L)$ and 
$L=\varphi (F(L))\otimes _{\sigma\c W}\c W$; 

$\bullet $\ for any $l\in F(L)$, $vN_S(l)\in 
F(L_S)$  
and $\varphi _S(vN_S(l))=cN_S(\varphi (l))$, 
where $c=1+u^p/p$. 
\end{definition}
\medskip

\begin{definition}
 Denote by $\uL ^t[1]$ the full subcategory in 
$\uL ^t$, which consists of objects killed by $p$.  
\end{definition}

The category 
$\uL ^t[1]$ is not very far from the category 
$\uL ^*$ introduced  in  Section \ref{S1}. 
Indeed, suppose $\L =(L,F(L),\varphi ,N_S)\in\uL ^t[1]$. 
Note that $N_S(L)\subset L_{S_1}:=L\otimes _{\W _1}S_1=
L/u^pL\oplus (L\otimes 1)\Fil ^pS_1$. (Remind that 
$S_1=S/pS=\c W_1/u^p\c W_1\oplus \Fil ^pS_1$.) 
With this notation we have the following property.

\begin{Prop} \label{P3.9} There is  a unique 
 $N:L\longrightarrow L/u^{2p}$ such that 

a) for any $l\in L$, $N(l)\otimes 1=cN_S(l)$ in 
$L _{S_1}$, where $c=1+u^p/p\in S^*$; 

b) $(L,F(L),\varphi ,N)\in\uL ^*$.
\end{Prop}

\begin{proof} Let $N_1:=cN_S:L\longrightarrow L_{S_1}$. 
 Then for any $w\in\W _1$ and $l\in L$, it holds 
$N_1(wl)=N(w)l+wN_1(l)$ (use that $N(c)=0$ in $S_1$) and 
there is a commutative diagram (use that $\sigma (c)=1$ in $S_1$) 

$$\xymatrix{F(L)\ar[r]^{\varphi }\ar[d]_{uN_1}  & L\ar[d]^{N_1} 
\\ F(L)_S \ar[r]^{\varphi _S} & L_{S_1} }
$$

Prove that $N_1(\varphi (F(L))\subset L/u^pL$ and, therefore, 
$N_1(L)\subset L/u^pL$. 

Indeed, 
$(uN_1)(F(L))\subset uN_1(L)\cap F(L)_S\subset 
\left (uL/u^pL\oplus (uL)\Fil ^pS_1\right )$
\linebreak 
$\cap \left (F(L)/u^pL\oplus L\Fil ^pS_1\right )
\subset F(L)/u^pL\oplus (uL)\Fil ^pS_1$.
This implies that 
$N_1(\varphi (F(L))\subset\varphi _S(uN_1(F(L)))\subset L/u^pL$ 
because $\varphi _S(u\Fil ^pS_1)=0$. 
So, by Proposition \ref{P1.3} there is 
a unique $N:L\longrightarrow L/u^{2p}$ such that 
$N\mathrm{mod}\,u^p=N_1$ and $(L,F(L),\varphi ,N)\in\uL ^*$. 
\end{proof}

\begin{Cor} \label{C3.10} With above notation the correspondence 
$$(L,F(L),\varphi ,N_S)\mapsto (L,F(L),\varphi ,N)$$ 
induces the equivalence of categories 
$\Pi :\uL ^t[1]\longrightarrow\uL ^*$.
\end{Cor} 

\begin{proof}
We must verify that our correspondence is surjective on objects 
and bijective on morphisms. The first holds because 
$N_S=c^{-1}N\operatorname{mod}u^p$ and the second --- because 
a $\c W_1$-linear map $f$ commutes with $N$ iff it commutes with 
$N\operatorname{mod}u^p$ (use Proposition \ref{P1.2}) 
iff $f\otimes _{\c W_1}S_1$  commutes with $N_S$. 
\end{proof}

\begin{Cor} \label{C}
 The category $\uL ^t$ is preabelian. 
\end{Cor}

\begin{proof} Corollary \ref{C3.10} 
and Proposition \ref{P1.3} imply that 
$\uL ^t[1]$ is pre-abelian. This can be extended then to the whole 
category $\uL ^t$ by Breuil's method from 
 \cite{refBr2} via above Lemma \ref{L3.8}. 
\end{proof}

Note that if 
$\c L=(L,F(L),\varphi ,N_S)$ and $\c M=(M,F(M),\varphi ,N_S)$ are objects of 
$\uL ^t$ and $f\in\Hom _{\uL}(\c L,\c M)$ then:
\medskip 

$\bullet $\ $\Ker f=(K,F(K), \varphi ,N_S)$, where 
$K=\Ker (f:L\To M)$ and $F(K)=F(L)\cap K$ 
with induced $\varphi $ and $N_S$;
\medskip 

$\bullet $\ $\Coker f=(C,F(C),\varphi ,N_S)$, where 
$C=M/M'$,  $M'$ is equal to  $(f(L)\otimes _{\c W}\c W[u^{-1}])\cap M$ 
and $F(C)=F(M)/(M'\cap F(M))$  with induced $\varphi $ and $N_S$;
\medskip 

$\bullet $\ $f$ is strict monomorphic means that $f:L\To M$ is monomorphism  
of $\c W$-modules, $\left (f(L)\otimes _{\c W}\c W[u^{-1}]\right )\cap M=f(L)$ 
(or, equivalently, $M/f(L)$ is $p$-strict) and $f(F(L))=L\cap F(M)$;
\medskip 

$\bullet $\ $f$ is strict epimorphic means that $f$ 
is epimorphism of $p$-strict modules 
and $f(F(L))=F(M)$. 
\medskip 

According to Appendix A, we can use the concept of $p$-divisible group 
$\{\c L^{(n)}, i_n\}_{n\geqslant 0}$ in $\uL ^t$. In this case 
$\c L^{(n)}=(L_n, F(L_n), \varphi ,N_S)$, where all 
$L_n$ are free $\c W/p^n$-modules of the same rank equal to 
the height of this $p$-divisible group. We have obvious 
equivalence of the category $\uL ^f$ and the category of 
$p$-divisible groups of finite height in $\uL ^t$. 
\medskip 

\begin{definition} Denote by $\uL ^{ft}$ the full subcategory  in 
$\uL ^t$, which consists of strict subobjects of 
$p$-divisible groups in $\uL ^t$. By $\uL ^{ft}[1]$ we denote 
the full subcategory in $\uL ^{ft}$ consisting of all objects killed by $p$.
\end{definition} 

It is easy to see that $\uL ^{ft}$ contains all 
strict subquotients of the corresponding $p$-divisible groups. 
Contrary to the case 
of filtered modules coming from 
crystalline representations, the categories $\uL ^{ft}$ and $\uL ^t$ 
do not coincide but they have the same simple objects.

Note that the functor $\Pi $ from Corollary \ref{C3.10} 
identifies simple objects of the categories $\uL ^t$ and $\uL^*$ 
and for any two objects $\L _1,\L _2\in\uL ^t[1]$, we have 
a natural isomorphism $\Ext _{\uL ^t[1]}(\L _1,\L _2)=
\Ext _{\uL^*}(\Pi (\L _1),\Pi (\L _2))$. One can use the methods 
of Subsection \ref{S1.2} to extend the  
concepts of etale, connected, unipotent and multiplicative objects 
to the whole category $\uL ^t$. The starting point for this extension 
is the case of $W(k)$-modules, which is well-known from the classical Dieudonne 
theory \cite{refDM}. Then we obtain the following standard properties:
\medskip 

$\bullet $\ for any $\L\in\uL ^t$, there are a unique maximal etale subobject 
$(\L ^{et},i^{et})$ and a unique maximal connected quotient object 
$(\L ^c,j^c)$ in $\uL ^t$ such that the sequence 
$0\longrightarrow\L ^{et}\overset{i^{et}}
\longrightarrow\L\overset{j^c}\longrightarrow\L ^c\longrightarrow 0$ 
is exact and the correspondences $\L\mapsto\L ^{et}$ and 
$\L\mapsto\L ^c$ are functorial; 
if $\c L\in\uL ^{ft}$ then $\c L^{et}$ and $\c L^c$ are 
also objects of $\uL ^{ft}$; 
\medskip 

$\bullet $\ for any $\L\in\uL ^t$, there are a unique maximal unipotent subobject 
$(\L ^{u},i^{u})$ and a unique maximal multiplicative quotient object 
$(\L ^m,j^m)$ in $\uL ^t$ such that the sequence 
$0\longrightarrow\L ^{u}\overset{i^{u}}
\longrightarrow\L\overset{j^m}\longrightarrow\L ^m\longrightarrow 0$
is exact and the correspondences $\L\mapsto\L ^{u}$ 
and $\L\mapsto\L ^m$ are functorial; if $\c L\in\uL ^{ft}$ then 
$\c L^u$ and $\c L^u$ are also objects of $\uL ^{ft}$. 
\medskip 

Denote by $\uL ^{et,t}$, $\uL ^{c, t}$, $\uL ^{u, t}$ 
and $\uL ^{m, t}$ the full 
subcategories in $\uL ^t$ consisting of, resp., etale, 
connected, unipotent and multiplicative objects. 
We have also the corresponding full subcategories 
$\uL ^{et, ft}$, $\uL ^{c, ft}$, $\uL ^{u, ft}$ 
and $\uL ^{m, ft}$  
in $\uL ^{ft}$.

The results of Subsection 1.5 and  Appendix A imply that in the category $\uL ^{ft}$: 
\medskip 

$\bullet $\ there is a unique etale $p$-divisible group  
$\c L^{\infty }(0):=\{\c L^{(n)}(0),i_n\}_{n\geqslant 0}$ 
of height 1  such that 
$\c L^{(1)}(0)=\c L(0)$; 
\medskip 

$\bullet $\ there is a unique multiplicative $p$-divisible group 
of height 1, 
\linebreak     
$\c L^{\infty }(1):=
\{\c L^{(n)}(1),i_n\}_{n\geqslant 0}$ such that  
$\c L^{(1)}(1)=\c L(1)$; 
\medskip 

$\bullet $\ for any $p$-divisible group $\c L^{\infty }$ there are 
functorial exact sequences of $p$-divisible groups 
$$0\To \c L^{\infty ,et}\To\c L^{\infty }\To \c L^{\infty ,c}\To 0$$ 
$$0\To\c L^{\infty ,u}\To\c L^{\infty }\To\c L^{\infty ,m}\To 0$$
Here $\c L^{\infty ,et}$ and $\c L^{\infty ,m}$ are products of 
several copies of $\c L^{\infty }(0)$ and, resp.,  $\c L^{\infty }(1)$, 
and $\c L^{\infty ,c}$ and $\c L^{\infty ,u}$ are $p$-divisible groups in 
the categories  $\uL ^{c,ft}$ and, resp.,  $\uL ^{u,ft}$. 

\newpage

\section{Semistable modular representations with weights $[0,p)$}
\label{S4}

In this section we prove that all killed by $p$ subquotients of 
Galois invariant lattices of semistable $\Q _p[\Gamma _F]$-modules with Hodge-Tate 
weights $[0,p)$ can be obtained via the functor $\c V^*$ from Section \ref{S2}. 

\subsection{The functor $\c V^t:\uL^t\longrightarrow\uM _F$}
\label{S4.1}

For $n\geqslant 1$, introduce the objects  
$\c A_{st,n}=
(\hat A_{st ,n}, F(\hat A_{st .n}), \varphi ,N_S)$ 
of the category $\widetilde{\uL }$, 
with $\hat A_{st,n}=\hat A_{st}/p^n\hat A_{st}$, 
$F(\hat A_{st ,n})=F(\hat A_{st })/p^nF(\hat A_{st})$ 
and induced $\varphi $ and $N_S$. 
Let $\c A_{st,\infty}=(A_{st ,\infty }, F(A_{st ,\infty }),\varphi ,N_S)$ 
be the inductive limit of all $\c A_{st,n}$. 

For $\L\in\uL ^t$, set $\c V^t(\L)=\Hom _{\widetilde{\uL}}(\L ,\c A_{st,\infty })$ 
with the induced structure of  $\Gamma _F$-module. This gives the functor 
$\c V^t:\uL ^t\longrightarrow \uM _F$. We shall use the same notation 
for its  
restriction to the category $\uL ^{ft}$. 

\begin{Prop} \label{P4.1} Suppose $\L =(L,F(L),\varphi ,N_S)\in\uL ^t$. Then 
 $N_S|_{\varphi (F(L))}$ is nilpotent. 
\end{Prop}

By devissage and Corollary \ref{C3.10} this is implied by  
the following statement for the objects of the category $\uL ^*$.
 
\begin{Lem} \label{L4.2} If $\L =(L,F(L),\varphi ,N)\in\uL^*$ then 
 $N^p(\varphi (F(L))\subset u^pL$.
\end{Lem}

\begin{proof} For any $l\in F(L)$, $N(\varphi (l))=\varphi (uN(l))$. 
Use induction to prove that for $1\leqslant m\leqslant p$, 
$N^m(\varphi (l))\equiv \varphi (u^mN^m(l))\mathrm{mod}\,u^pL$ 
and use then that $\varphi (u^pN^p(l))\in\varphi (uF(L))\subset u^pL$.
\end{proof}

\begin{Prop} \label{P4.3} For $n\geqslant 1$, 
$\oplus _{j\geqslant 0}A_{cr,n}\gamma _j(\log (1+X))$ 
is the maximal $W(k)$-submodule of $\hat A_{st ,n}$ where 
$N$ is nilpotent.
\end{Prop} 

\begin{proof} For any $j\geqslant 1$, it holds 
$N(\gamma _j(\log (1+X))=\gamma _{j-1}(\log (1+X))$ 
and $N$ is nilpotent on 
$\oplus _{j\geqslant 0}A_{cr,n}\gamma _j(\log (1+X))$. 
Therefore, it will be sufficient to prove that 
\[ \Ker \left (N^p|_{\hat A_{st, 1}}\right )=
 \underset{0\leqslant j<p}\oplus A_{cr,1}\gamma _j(\log (1+X)).
\]

Let $C=\F _p\langle X\rangle $ be the divided power envelope of 
$\F _p[X]$ with respect to the ideal $(X)$. 
Then $C=\F _p[X_0,X_1,\dots ,X_n,\dots ]_{<p}$ is the ring of polynomials in 
$X_i:=\gamma _{p^i}(X)$, where for all $i\geqslant 0$, $X_i^p=0$. 

Let $\m _C$ be the maximal ideal of $C$ and $Y=\log (1+X)\in C$. 
Then $Y\equiv X_0-X_1\,\mathrm{mod}\,\m _C^2$ and 
for all $j\geqslant 0$, 
$\gamma _{p^j}(Y)\equiv X_j-X_{j+1}\,\mathrm{mod}\,\m ^2_C$. 
This implies that with $Y_j=\gamma _j(Y)$ for all $j\geqslant 0$, 
$$C=\F _p[X_0,Y_0,\dots ,Y_n,\dots ]_{<p}=\F _p\langle Y\rangle [X]_{<p}=
\underset{0\leqslant i<p}\oplus\F _p\langle Y\rangle \gamma _i(X).$$
So, $\hat A_{st,1}=\underset{\substack{j\geqslant 0\\ 0\leqslant i<p}}\oplus 
A_{cr,1}\gamma _i(X)\gamma _j(Y)$. 
Remind $N(X)=X+1$ and for $j\geqslant 1$, 
$N(\gamma _j(Y))=\gamma _{j-1}(Y)$. Using 
that $N^p$ is an $A_{cr,1}$-derivation we obtain that for any 
$P=\sum _{i,j}\alpha _{ij}X^i\gamma _j(Y)
\in\F _p\langle Y\rangle [X]_{<p}$ with $\alpha _{ij}\in\F _p$, 
it holds (use that $N^p(X)=X+1$ and $N^p(\gamma _{j+p}(Y))=\gamma _j(Y)$)
$$N^p(P)=\sum_{i,j}\alpha _{ij}iX^i\gamma _j(Y)+
\sum _{i,j}(i+1)\alpha _{i+1,j}\gamma _j(Y)X^i+
\sum _{i,j}\alpha _{i,j+p}X^i\gamma _j(Y).$$
If $P\in\Ker N^p$ then for all involved indices $i,j$, 
$$i\alpha _{ij}+(i+1)\alpha _{i+1,j}+\alpha _{i,j+p}=0.$$ 
This  implies that 
$\alpha _{ij}=0$ if either $i\ne 0$ or $j\geqslant p$. 

Indeed, 
take $i=p-1$. Then $-\alpha _{p-1,j}+\alpha _{p-1,j+p}=0$. 
Because for $j\gg 0$, $\alpha _{p-1,j}=0$ it implies that all $\alpha _{p-1,j}=0$. 
Then proceed similarly with $i=p-2$ and so on. This proves that all 
$\alpha _{ij}=0$ if $i\ne 0$. It remains to note that for $i=0$, our relations give 
$\alpha _{0,j+p}=0$ for all $j\geqslant 0$. 
\end{proof}

As earlier, consider the category $\widetilde{\uL }_0$. Remind that its objects are 
the triples $(L,F(L),\varphi )$, where $L\supset F(L)$ are $\c W$-modules 
and 
$\varphi :F(L)\longrightarrow L$ is a $\sigma $-linear morphism. 
For any object $\L =(L,F(L),\varphi , N_S)\in\widetilde{\uL }$, 
agree to use the same notation $\L $ for the corresponding object 
$(L,F(L),\varphi )\in\widetilde{\uL }_0$. 

For all $n\geqslant 0$, set $\c A_{cr ,n}=(A_{cr ,n}, F(A_{cr ,n}),
\varphi )\in\widetilde{\uL }_0$ 
with $A_{cr,n}=A_{cr}/p^nA_{cr}$, $F(A_{cr ,n})=F(A_{cr })/p^nF(A_{cr })$ 
and induced $\varphi $.  
Here the $\c W$-module structure on $A_{cr ,n}$ is defined by the 
morphism of $W(k)$-algebras $\c W\longrightarrow A_{cr ,n}$ such that 
$u\mapsto [x_0]$. Denote by  $\c A_{cr ,\infty }$ 
the inductive limit of all $\c A_{cr ,n}$. 

Suppose $\L\in\uL ^t$ and 
$f\in\Hom _{\widetilde{\uL }}(\L ,\c A_{st ,n})$. Then by Propositions  
\ref{P4.1} and \ref{P4.3}, 
$$f(\varphi (F(L)))\subset \underset{j\geqslant 0}
\oplus A_{cr ,n}\gamma _j(\log (1+X)).$$
Consider the formal embedding of the algebra 
$A_{st ,n}$ into the completion 
$\prod _{j\geqslant 0} A_{cr ,n}\gamma _j(\log (1+X))$ 
of $\oplus _{j\geqslant 0}A_{cr ,n}\gamma _j(\log (1+X))$ such that 
$X\mapsto \sum _{j\geqslant 1}\gamma _j(\log (1+X))$. Then any element of 
$A_{st ,n}$ can be uniquely written in the form 
$\sum _{j\geqslant 0}a_j\gamma _j(\log (1+X))$, where all $a_j\in A_{cr ,n}$. 
Note that the $\c W$-module structure on $A_{st ,n}$ is given via the map
$$u\mapsto [x_0]/(1+X)=[x_0]
\sum _{j\geqslant 0}(-1)^j\gamma _j(\log (1+X)).$$

For $j\geqslant 0$, introduce 
the $W(k)$-linear maps $f_j\in\Hom (L,A_{cr ,n})$ such that 
for any $l\in L$, it holds $f(l)=\sum _{j\geqslant 0}f_j(l)\gamma _j(\log (1+X))$. 
Then using methods from \cite{refBr2} obtain the following property.

\begin{Prop} \label{P4.4} 
 a) The correspondence $f\mapsto f_0$ induces isomorphism of 
abelian groups $\c V^t(\L )=\Hom _{\widetilde{\uL}_0}(\L,\c A_{cr ,n})$;

b) for any $j\geqslant 0$ and $l\in L$, $f_j(l)=f_0(N^j(l))$.
\end{Prop}

\begin{Cor} \label{C4.5} The functor $\c V^t$ is exact. 
\end{Cor} 
\begin{proof} 
Let $\uL ^t_0$ be the full subcategory of $\widetilde{\uL} _0$ consisting 
of the triples $(L,F(L),\varphi )$ coming from 
all 
$\L =(L,F(L),\varphi ,N)\in\uL ^t$. By Proposition 
\ref{P4.4} it will be sufficient 
to prove that the 
functor $\c V ^t_0:\uL ^t_0\longrightarrow (Ab)$, such that 
$\c V ^t_0(\L )=
\Hom _{\widetilde{\uL }_0}(\L ,\c A_{cr ,\infty })$, is exact. 
The verification can be done by 
devissage  along the lines of paper    
\cite{refFL}. 
\end{proof}

\begin{remark} One can simplify 
the verification of above corollary   
by replacing $\c A_{cr ,1}$ by the corresponding 
object $\widetilde{\c A}_{cr ,1}$ related to the module 
$\widetilde{A}_{cr ,1}=(R/x_0^p)T_1\oplus (R/x_0^p)$ 
introduced in 
Subsection \ref{S4.2} below. 
\end{remark}

\begin{Cor} \label{C4.6} 
For $\L\in\uL ^f$, let $\{\L ^{(n)},i_n\}_{n\geqslant 0}$ 
be the corresponding $p$-divisible group in the category 
$\uL ^{ft}$. Then in notation of 
Corollary {\rm \ref{C3.2}}, 
$T^*_{st}(\L )=\underset{n}\varprojlim\c V ^{t}(\L ^{(n)})$.
\end{Cor}

\subsection{The functor $\c V[1]^*$}
\label{S4.2} 

Note the following case of Proposition \ref{P4.4}. 

\begin{Prop} \label{P4.7} 
Suppose $\L =(L,F(L),\varphi ,N)\in\uL ^t[1]$. Then 
there is an isomorphism of abelian groups 
$\c V ^t(\L )\simeq \Hom _{\widetilde{\uL}_0}(\L,\c A_{cr ,1})$ and 
$\Gamma _F$ acts on $\c V^t(\L )$ via its natural action  on $A_{st ,1}$ 
and the identification 
$$\iota _{\L }:
\Hom _{\widetilde{\uL}_0}(\L ,\c A_{cr ,1})
\longrightarrow 
\Hom _{\widetilde{\uL}}(\L, \c A_{st ,1}).$$
such that if 
$f_0\in\Hom _{\widetilde{\uL}_0} 
(\L ,\c A_{cr ,1})$ then 
for any $l\in L$, 
$$\iota _{\L}(f_0)(l)=
\sum _{j\geqslant 0}f_0(N^j(l))\gamma _j(\log (1+X))$$
\end{Prop}

Introduce the functor 
$\c V[1]^*:= \c V^t|_{\uL ^t[1]}\circ\Pi ^{-1}:\uL^*\longrightarrow\uM _F$, 
where $\Pi :\uL ^t[1]\To\uL ^*$ is the equivalence of categories from 
Corollary \ref{C3.10}. 

\begin{Prop} \label{P4.8} On the subcategory of unipotent objects $\uL^{*u}$ of $\uL^*$ 
 the functors $\c V [1]^*$ and $\c V^*$ coincide.
\end{Prop} 

\begin{proof} The definition of $A_{cr}$ implies that 
 $A_{cr ,1}=(R/x_0^p)[T_1,T_2,\dots ]_{<p}$, where for 
all indices $i\geqslant 1$, $T_i$ comes from $\gamma _{p^i}([x_0]+p)$ and 
 $T_i^p=0$. 
Set $F(A_{cr ,1})=\Fil ^{p-1}A_{cr ,1}=
(x_0^{p-1}R/x_0^pR)\oplus (R/x_0^p)I_1$, where the ideal 
$I_1$ is generated by all $T_i$. Then the corresponding map $\varphi :
F(A_{cr ,1})\longrightarrow A_{cr ,1}$ is uniquely determined by the conditions 
$\varphi (x_0^{p-1})=1-T_1$, $\varphi (T_1)=1$ and 
$\varphi (T_i)=0$ if $i\geqslant 2$. In particular, 
$\varphi (A_{cr ,1})\subset (R/x_0^p)T_1\oplus (R/x_0^p)$. 

Let $\widetilde{A}_{cr ,1}=A_{cr ,1}/J_1$ with the induced structure of 
filtered $\varphi $-module $\widetilde{\c A}_{cr ,1}$, 
where the ideal $J_1$ of $A_{cr ,1}$ is generated 
by the elements $T_1x_0^p$ and $T_i$ with $i\geqslant 2$. Then the projection 
$A_{cr ,1}\longrightarrow \widetilde{A}_{cr ,1}$ induces for any object 
$\L =(L,F(L),\varphi ,N)$ of the category $\uL ^*$, the identification 
(use that $\varphi |_J=0$)
$$\Hom _{\widetilde{\uL}_0}(\L ,\c A_{cr ,1})=
\Hom _{\widetilde{\uL }_0}(\L ,\widetilde{\c A}_{cr ,1}).$$

Introduce $a_0,a_{-1}\in\Hom (L,R/x_0^p)$ such that for any $m\in L$, 
$f_0(m)=a_{-1}(m)T_1+a_0(m)$. Note that $a_0$ and $a_{-1}$ are $\c W_1$-linear, 
where the multiplication by $u$ on $L$ correspondes to the multiplication by 
$x_0$ in $R/x_0^p$. 

Then for any $m\in F(L)$, the requirement 
$f_0(\varphi (m))=\varphi (f_0(m))$ is equivalent to the conditions

\begin{equation}\label{E4.1}
\begin{array} {lcl} a_0(\varphi (m)) & = & a_{-1}(m)^p+
   \dfrac{a_0(m)^p}{x_0^{p(p-1)}}\\
a_{-1}(\varphi (m))& = & -\dfrac{a_0(m)^p}{x_0^{p(p-1)}}
  \end{array}
\end{equation}

Note that these conditions depend only on 
$\bar m=m\operatorname{mod}u^pL$. 
\medskip 

Consider the operator $V:L\longrightarrow L$ from Subsection 
\ref{S1.5}. 
Clearly,  $V(u^pL)\subset uF(L)$ and for $\bar L:=L/u^pL$, 
we obtain the induced operator 
$\bar V:\bar L\longrightarrow\bar L$ (use that $F(L)/uF(L)\subset L/u^pL$). 

For any $m\in\bar L$, relations \eqref{E4.1} can be 
rewritten as follows:

$$\begin{array}{lcl}a_0(m)&=&\dfrac{a_0(\bar Vm)^p}{x_0^{p(p-1)}}+
a_{-1}(\bar Vm)^p\\
a_{-1}(m)&=&-\dfrac{a_0(\bar Vm)^p}{x_0^{p(p-1)}}
\end{array}
$$
Therefore, if $\L $ is unipotent then for any $m\in\bar L$, 
$$a_{-1}(m)=-a_0(m)+a_{-1}(\bar Vm)^p=
-a_0(m)+a_{-1}(\bar V^2m)^{p^2}=\dots =-a_0(m).$$

This implies that for any $m\in F(\bar L)$, 
$a_0(\varphi (m))=a_0(m)^p/x_0^{p(p-1)}$. 
In other words, we have a natural 
identification 
$$\Hom _{\widetilde{\uL }_0}(\L ,\widetilde{\c R}^u)=
\Hom _{\widetilde{\uL }_0}(\L ,\widetilde{A}_{cr ,1})
$$
coming from the map of filtered $\varphi $-modules 
$\widetilde{\c R}^u\longrightarrow\widetilde{\c A}_{cr ,1}$ 
given by the $R$-linear map 
$R/x_0^p\longrightarrow \widetilde{A}_{cr ,1}
=(R/x_0^p)T_1\oplus (R/x_0^p)$ such that 
for any $r\in R/x_0^p$, 
$r\mapsto (-rT_1,r)$. 
(For the definition of $\wt{\c R}\in\uL ^*_0$ cf. Subsection \ref{S2.2}.) 

This implies that for all unipotent $\L\in\uL ^{*u}$, there is a natural  
identification of $\Gamma _F$-modules 
 $\c V[1]^*(\L )=\c V^*(\L )$. 
Indeed, the above embedding 
$R/x_0^p\longrightarrow \widetilde{A}_{cr ,1}$ 
can be extended to the embedding 
of $R_{st}/x_0^pR_{st}$ to 
$\widetilde{A}_{st ,1}=
\prod _{j\geqslant 0}
\widetilde{A}_{cr ,1}\gamma _j(\log (1+X))$, which induces the above identification.  
\end{proof}
\medskip

\subsection{Splittings $\Theta $ and $\widetilde{\Theta }$}
\label{S4.3}

Suppose $\L =(L,F(L),\varphi ,N)\in\uL ^*$. Then there is a standard 
short exact sequence 
\begin{equation}\label{E4.2}
0\longrightarrow \L ^u\overset{i}\longrightarrow\L 
\overset{j}\longrightarrow \L ^m\longrightarrow 0,
\end{equation} 
where 
$(\L ^u,i)$ is the maximal unipotent subobject and 
$(\L ^m,j)$  is the maximal multiplicative quotient of $\L$. 

If $\L ^m=(L^m,F(L^m),\varphi ,N)$ then $F(L^m)=L^m=L_0\otimes _{\F _p}\c W_1$, 
where $L_0=\{l\in L^m\ |\ \varphi (l)=l\}$. 
Suppose $S:L^m\longrightarrow F(L)\subset L$ is a 
$\W _1$-linear section. Then for any $l_0\in L_0$, 
$S(l_0)=\varphi (S(l_0))+g(l_0)$, where 
$g\in\Hom (L_0,L^u)$. If $S':L^m\longrightarrow F(L)$ 
is another $\c W_1$-linear section then for any $l_0\in L_0$, 
$S'(l_0)=\varphi (S'(l_0))+g'(l_0)$. Here 
$g'\in\Hom (L_0,L^u)$ is such that for some 
$h\in\Hom (L_0,L^u)$, it holds 
$$(g'-g)(l_0)=h(l_0)-\varphi (h(l_0)).$$ 

\begin{Prop} \label{P4.9}

a) There is a section $S$ such that $g(L_0)\subset uL^u$. 

b) If $g(L_0),g'(L_0)\subset uL^u$ then $h(L_0)\subset uF(L^u)$.
\end{Prop} 

\begin{proof} a) It will be sufficient 
to prove that for any $l\in L^u$, 
there is an $h\in F(L^u)$ such that 
$l\equiv h-\varphi (h)\,\mathrm{mod}\,uL^u$. 

Suppose $n_0\geqslant 1$ is such that 
$V^{n_0}(L^u)\subset uF(L^u)$. Then 
for all $n\geqslant n_0$, $V^n(L^u)\subset uF(L^u)$. 
Let $h=-(Vl+V^2l+\dots +V^{n_0+1}l)$. 
By the definition of the operator $V$ for all $1\leqslant i\leqslant n_0+1$, 
$V^il\in F(L^u)$ and $\varphi (V^il)\equiv V^{i-1}l\operatorname{mod}uL^u$. Therefore, 
$h\in F(L^u)$ and $\varphi (h)\equiv -(l+Vl+
\dots +V^{n_0}l)\equiv -l+h\,\mathrm{mod}\,uL^u$.

b) We must prove that if $h\in F(L^u)$ and 
$h-\varphi (h)\in uL^u$ then $h\in uF(L^u)$. 

Indeed, we have $V(h)-h\in V(uL^u)\subset uF(L^u)$ and for all 
$n\geqslant 1$, $V^n(h)\equiv h\mathrm{mod}\,uF(L^u)$ implies that 
$h\in uL^u$. Therefore, 
$\varphi (h)\in uL^u$ and $h\in uF(L^u)$.
\end{proof} 

\begin{Prop} \label{P4.10} 
With above notation the short exact sequence 
 $$0\longrightarrow\c V[1]^*(\L ^m)\longrightarrow 
\c V[1]^*(\L )\longrightarrow 
\c V[1]^*(\L ^u)\longrightarrow 0$$
obtained from \eqref{E4.2} by applying $\c V[1]^*$, 
has a canonical functorial splittings 
$\Theta :\c V[1]^*(\L ^u)\longrightarrow \c V[1]^*(\L )$ and 
$\widetilde{\Theta }:\c V[1]^*(\L )\longrightarrow \c V[1)]^*(\L ^m)$
in the category $\uM _F$.
\end{Prop}

\begin{proof} 
It will be sufficient to prove the existence 
of a functorial splitting 
$$\Theta :\Hom _{\widetilde{\uL }_0}(\L^u,\widetilde{\c A}_{cr ,1})
\longrightarrow\Hom _{\widetilde{\uL }_0}(\L ,\widetilde{\c A}_{cr ,1})$$
of the epimorphism 
$\Hom _{\widetilde{\uL }_0}(\L ,\widetilde{\c A}_{cr ,1})
\rightarrow \Hom _{\widetilde{\uL }_0}(\L ^u,\widetilde{\c A}_{cr ,1})$, 
obtained from exact sequence \eqref{E4.2}. 

Suppose $f_0=(a_{-1},a_0):L^u\longrightarrow (R/x_0^p)T_1\oplus (R/x_0^p)$ 
belongs to $\Hom _{\widetilde{\uL }_0}(\L ^u,\widetilde{\c A} _{cr ,1})$. Here 
$a_{-1},a_0\in\Hom _{\c W_1}(L^u,R/x_0^p)$ and for any $l\in L^u$, 
$a_{-1}(l)=-a_0(l)$, cf. Subsection \ref{S4.2}.  

Let $S:L^m\longrightarrow L$ be a $\c W_1$-linear section such that 
for any $l\in L_0$, $S(l_0)=\varphi (S(l_0))+g(l_0)$, where 
$g\in \Hom (L_0,uL^u)$. 

Extend $f_0$ to $\Theta (f_0)=(a_{-1},a_0):L
\longrightarrow (R/x_0^p)T_1\oplus (R/x_0^p)$ by setting 
$a_{0}(S(l_0))=-a_{-1}(S(l_0))=\c X$, where $\c X$ is a 
unique element of $R/x_0^p$ such that 
$\c X-\c X^p/x_0^{p(p-1)}=a_0(g(l_0))$. One can prove that 
$\Theta (f_0)\in\Hom _{\widetilde{\uL }_0}(\L ,\widetilde{\c A}_{cr ,1})$
by verifying relations \eqref{E4.1} with $m=S(l_0)$.
\end{proof}

\subsection{A modification of Breuil's functor}
\label{S4.4}

Remind that Breuil's functor $\c V ^t:\uL ^t\longrightarrow\uM _F$ 
attaches to any $\L\in\uL ^t$, the $\Gamma _F$-module $\c V(\L )=
\Hom _{\widetilde{\uL }}(\L ,\c A _{st,\infty })$.

\begin{Prop} \label{P4.11} 
The functor $\c V^t$ is fully faithful on the 
subcategory of unipotent objects $\uL ^{t,u}$. 
\end{Prop}

\begin{proof} Indeed, by Subsection \ref{S2.3}, 
$\c V[1]^*$ is fully faithful. Then the exactitude of 
$\c V^t$ together with Proposition \ref{P4.8} implies that $\c V^t|_{\uL ^{u,t}}$ is 
fully faithful. 
\end{proof}

Proposition \ref{P4.10} 
implies that $\c V^t$ 
is very far from to be fully 
faithful on the whole $\uL ^t$:  
 if $\L\in\uL ^t[1]$ and 
$0\longrightarrow\L ^u\longrightarrow\L
\longrightarrow\L^m\longrightarrow 0$ is the standard 
exact sequence then the corresponding sequence of 
$\Gamma _F$-modules admits a functorial splitting.

Introduce a modification $\widetilde{\c V}^{ft}:\uL ^{ft}
\longrightarrow\uM _F$ of Breuil's functor. 

Suppose $\L\in\uL ^{ft}$. From the definition of 
the category $\uL ^{ft}$ in Subsection 
\ref{S3} it follows the existence of 
$\L '\in\uL ^{ft}$ such that $p\L'=\L$. 
 More precisely, 
there are a strict monomorphism 
$i_{\L '}:\L\longrightarrow\L '$ 
and a strict epimorphism $j_{\L '}:\L '\longrightarrow\L $ such that 
$p\,\id _{\L '}=i_{\L '}\circ j_{\L'}$. 
(Note that $j_{\L '}\circ i_{\L '}=p\,\id _{\L }$.) 

Consider the following short exact sequences 

\begin{equation}\label{E4.3}
0\longrightarrow\L\overset{i_{\L '}}
\longrightarrow\L '\overset{C_p}\longrightarrow _p\L '\longrightarrow 0
\end{equation}
\begin{equation}\label{E4.4}
0\longrightarrow\L '_p\overset{K_p}
\longrightarrow\L '\overset{j_{\L '}}
\longrightarrow \L\longrightarrow 0
\end{equation}
and consider the corresponding sequence 
of $\Gamma _F$-modules and their morphisms 

$$\c V^t(_p\L ^{\prime u})\overset{\Theta }\longrightarrow 
\c V^t(_p\L ')\overset{\c V^t(C_p)}\longrightarrow 
\c V^t(\L ')\overset{\c V^t(K_p)}\longrightarrow 
\c V^t(\L '_p)\overset{\wt\Theta }\longrightarrow\c V^t(\L _p^{\prime m}).$$
As earlier, for any $\L\in\uL ^{ft}$, $\L ^u$ is the 
maximal unipotent subobject and $\L ^m$ is 
the maximal multiplicative quotient object for $\L $. 
\begin{Lem} \label{L4.12} 
 $\Ker (\wt\Theta \circ\c V^t(K_p))\supset\mathrm{Im}(\c V^t(C_p)\circ\Theta )$.
\end{Lem}

\begin{proof} The section $\Theta $ depends 
functorially on objects of the category $\uL ^t[1]\supset\uL ^{ft}[1]$. 
Therefore, we have the following commutative diagram

$$\xymatrix{\c V^t(_p\L ^{\prime u})
\ar[rr]^{\c V^t(C_p^u\circ K_p^u)}\ar[d]^{\Theta } &&
\c V^t(\L ^{\prime u}_p)\ar[d]^{\Theta } 
\\ \c V^t(_p\L ')\ar[rr]^{\c V^t(C_p\circ K_p)} && \c V^t(\L '_p) }
$$
and \ \ $\wt\Theta\circ\c V^t(K_p)\circ\c V^t(C_p)\circ\Theta =
(\wt\Theta\circ\Theta)\circ\c V^t(C_p^u\circ K_p^u)=0$. 
\end{proof}

\begin{definition} Set $\c V^t_{\L '}(\L )=
\Ker (\wt\Theta\circ \c V^t(K_p))/
\mathrm{Im}\,(\c V^t(C_p)\circ\Theta )$.
\end{definition}

\begin{Prop} \label{P4.13} With above notation it holds: 

a) $\c V^t_{\L '}(\L )=
\Coker\,\c V^t(C_p)=\c V^t(\L )$ if $\L\in\uL ^{u,ft}$;

b) $\c V^t_{\L '}(\L )=\Ker \,\c V^t(K_p)=\c V^t(\L )$ if 
$\L\in\uL ^{m,ft}$; 

c) for any $\L\in\uL ^{ft}$, we have the induced exact sequence of $\Gamma _F$-modules 
$0\longrightarrow\c V^t(\L ^m)\longrightarrow\c V^t_{\L '}
(\L )\longrightarrow \c V^t(\L ^u)\longrightarrow 0$. 
This sequence depends functorially on the pair $(\L ,\L ')$.
\end{Prop}

\begin{proof} The parts a) and b) are obtained directly from 
 definitions. In order to prove c), note that $p\L '=\L $ implies that 
$p\L ^{\prime u}=\L ^u$ and $p\L ^{\prime m}=\L ^m$. This gives a functorial sequence 
$$0\longrightarrow \c V^t_{\L ^{\prime m}}
(\L ^m)\longrightarrow\c V^t _{\L '}(\L )\longrightarrow
\c V^t_{\L ^{\prime u}}(\L ^u)\longrightarrow 0.$$
Then a diagram chasing proves that this sequence is exact. 
\end{proof}

\begin{Prop} \label{P4.14} Suppose for a given $\L\in\uL ^{ft}$, the objects 
 $\L ',\L ''\in\uL ^{ft}$ are such that $p\L '=p\L ''=\L$. Then there is 
a natural isomorphism $f(\L ',\L '')$ of 
$\Gamma _F$-modules such that the following diagram is commutative
$$\xymatrix{0\ar[r] & \c V^t(\L ^m) \ar[d]^{\id }\ar[r] & \c V^t_{\L ''}
(\L )\ar[d]^{f(\L ',\L '')}\ar[r] & \c V^t(\L ^u)\ar[d]^{\id }\ar[r] & 0\\
0\ar[r] & \c V^t (\L ^m)\ar[r] & \c V^t_{\L '}(\L )\ar[r] & \c V^t(\L ^u)\ar[r] & 0 }
$$
(The lines of this diagram are given by Prop \ref{P4.13})
\end{Prop}

\begin{proof} By replacing $\L ''$ by $\L '\underset{\L }\prod \L ''$ 
with respect to strict epimorphisms $j_{\L '}$ and $j _{\L ''}$, 
 we can assume that 
there is a map $f:\L ''\longrightarrow\L '$ which induces the identity map 
$p\L ''=\L\longrightarrow p\L '=\L $. Then the existence of $f(\L ',\L '')$ 
follows from functoriality and the diagram chasing implies 
that it induces the identity maps on $\c V^t(\L ^u)$ and $\c V^t(\L ^m)$.
\end{proof}

\begin{definition} For $\L ,\L '\in\uL ^{ft}$ such that $p\L '=\L$, set 
$\widetilde{\c V}^{ft}(\L )=\c V^t_{\L '}(\L )$. 
\end{definition}

The correspondence 
$\L\longrightarrow\widetilde{\c V}^{ft}(\L )$ 
induces the additive exact functor $\widetilde{\c V}^{ft}:\uL ^{ft}\longrightarrow \uM _F$. 

\subsection{$\varphi $-filtered module 
$\widetilde{\c A}_{cr,2}\in\widetilde{\uL }_0$}\label{S4.5}

Let $\xi =[x_0]+p\in W(R)\subset A_{cr}$, and for $n\in\N $, 
$\gamma _n(\xi )=\xi ^n/n!$

\begin{Lem}\label{L4.15} If $n\geqslant 2p$ then 
$\varphi (\gamma _n(\xi ))\in p^2A_{cr}$. 
\end{Lem} 

\begin{proof} We have $\varphi (\gamma _n(\xi ))=(p^{n-p+1}/n!)([x_0]^p/p+1)^n$.
Therefore, 
 it will be sufficient to verify that for 
$n\geqslant 2p$, $v_p(n!)+p+1\leqslant n$.
Using the estimate $v_p(n!)<n/(p-1)$ we obtain that the required inequality 
holds for $p\geqslant 5$ if $n\geqslant p+3$ and for $p=3$ if $n\geqslant 8$. 
It remains to check that our inequality holds for $p=3$ and $n=6$ and 7.
\end{proof}

Let $J_2$ be the closed ideal in $A_{cr}$ generated by 
$[x_0]^p\xi ^p/p$ and all 
$\xi ^n/n!$ with $n\geqslant 2p$. Then 
$J_2\subset F(A_{cr})$ and $\varphi (J_2)\subset p^2A_{cr}$. 
Introduce $\widetilde{A}_{cr,2}=A_{cr}/(J_2+p^2A_{cr})$ and consider 
the corresponding induced filtered $\varphi $-module 
$\widetilde{\c A}_{cr,2}=(\widetilde{A}_{cr ,2}, F(\widetilde{A}_{cr ,2}),
\varphi )\in\widetilde{\uL}_0$. Clearly, for any $\L\in\uL ^{t}_0$, 
the natural projection $\c A_{cr ,2}\rightarrow\widetilde{\c A}_{cr,2}$ 
induces the identification 
$\Hom _{\widetilde{\uL}_0}(\L ,\c A_{cr ,2})=
\Hom _{\widetilde{\uL}_0}(\L ,\widetilde{\c A}_{cr ,2})$. 

Consider the structure of 
$\widetilde{\c A}_{cr ,2}$ more closely. 

Let $T_1=\xi ^p/p$. With obvious notation the elements 
of $\widetilde{A}_{cr ,2}$ can be written uniquely modulo the subgroup 
$[x_0^pR]T_1+[x_0^{2p}R]+p[x_0^pR]+p^2W(R)$ in the form 
$[r_{-1}]T_1+[r_0]+p[r_1]$, where $r_{-1},r_0,r_1\in R$. 
Informally, we shall use that $r_{-1},r_1\in R/x_0^p$ and $r_0\in R/x_0^{2p}$.  
The $W(R)$-module structure on 
$\widetilde{A}_{cr ,2}$ is induced by usual operations on 
Teichmuller's representatives and the relation 
$pT_1\equiv [x_0]^p\,\mathrm{mod}\,p^2W(R)$. (Use that 
$T_1\equiv [x_0]^p/p+p[x_0]^{p-1}\mathrm{mod}\,p^2W(R)$.)

The $S$-module structure on $\widetilde{A}_{cr ,2}$ is induced 
by the $W(k)$-algebra morphism 
$S\longrightarrow W(R)$ 
such that $u\mapsto [x_0]$. Then $F(\widetilde{A}_{cr ,2})$ 
is generated over $W(R)$ by the images of $T_1$ and $\xi ^{p-1}$. Note that 
$\xi ^{p-1}\equiv [x_0]^{p-1}-p[x_0]^{p-2}\,\mathrm{mod}\,p^2W(R)$. 

The map $\varphi :F(\widetilde{A}_{cr ,2})\longrightarrow \widetilde{A}_{cr ,2}$ 
is uniquely determined by the knowledge of $\varphi (T_1)$ and 
$\varphi (\xi ^{p-1})$. Note that 
$$\varphi (T_1)=\left (\frac {1+[x_0]^p}{p}\right )^p
\equiv 1+[x_0]^p\,\mathrm{mod}(J+p^2A_{cr ,2}+p[\m _R])$$
$$\varphi (\xi ^{p-1})=\left (1+\frac{[x_0]^p}{p}\right )^{p-1}\equiv 1-T_1\,\mathrm{mod}\,(J+p^2A_{cr ,2}+p[\m _R])$$

Suppose $\L\in\uL ^{ft}[1]$ and $\L'\in\uL ^{ft}$ is such that $p\L'=\L$. 
Consider short exact sequences \eqref{E4.3} and \eqref{E4.4}. 
Then the points $f\in\c V^t(_p\L ')$ and 
$\c V^t(C_p)(f)\in\c V^t(\L ')$ are related via the commutative diagram 

$$\xymatrix{\L '\ar[d]^{C_p}\ar[rr]^{\c V^t(C_p)(f)}&&\widetilde{\c A}_{cr ,2}\ar[d]\\
_p\L '\ar[rr]^f&&\widetilde{\c A}_{cr ,1}}
$$
where the right vertical arrow is induced by the correspondence 
$$[r_{-1}]T_1+[r_0]+p[r_1]\,\mapsto [r_{-1}]T_1+[r_0\,\mathrm{mod}\,x_0^p].$$

Similarly, the points $g\in\c V^t(\L')$ and $\c V^t(K_p)(g)\in\c V^t(\L '_p)$ are related 
via the commutative diagram

$$\xymatrix{\L '\ar[rr]^{g}&&\widetilde{\c A}_{cr ,2}\\
\L _p'\ar[u]^{K_p}\ar[rr]^{\c V^t(K_p)(g)}
&&\widetilde{\c A}_{cr ,1}\ar[u]}
$$
where the right vertical arrow is induced by the correspondence 
$$[r_{-1}]T_1+[r_0]\mapsto [r_{-1}x_0^p]+p[r_0].$$

\subsection{Filtered $\varphi $-modules $\c A^0_{cr,1}$ and $\c A_{cr ,2}^0$.}
\label{S4.6} 

Let $A^0_{cr ,2}$ be the $W(R)$-submodule of $\widetilde{A}_{cr ,2}$ consisting of 
elements $[r_{-1}]T_1+[r_0]+p[r_1]$ such that $r_{-1}=-r_0\,\mathrm{mod}\,x_0^p$. 
Then $F(A^0_{cr ,2})=F(\widetilde{A}_{cr ,2})\cap A^0_{cr ,2}$ is 
generated over $W(R)$ by $[x_0^{p-1}]T_1+\xi ^{p-1}$ 
and the congruence 
$$\varphi ([x_0^{p-1}]T_1+\xi ^{p-1})
\equiv -T_1+1\,\mathrm{mod}\,(J_2+p^2A_{cr ,2}+p[\m _R])$$
implies that 
$\varphi (F(A^0_{cr ,2})\subset A^0_{cr ,2}$ and 
$A^0_{cr ,2}=(A^0_{cr ,2}, F(A^0_{cr ,2}),\varphi )\in\widetilde{\uL}_0$.

Note that  $p\c A^0_{cr ,2}=(pA^0_{cr ,2}, pF(A^0_{cr ,2}),\varphi )\in\widetilde{\uL }_0$. 
Then in notation from Subsection \ref{S4.4}, it holds:

$\bullet $\ $\mathrm{Im}\,\Theta =\Hom _{\widetilde{\uL}_0}
(_p\L ',p\c A^0_{cr ,2})$;

$\bullet $\ $\c V^t(C_p)(\mathrm{Im}\,\Theta )=\Hom _{\widetilde{\uL }_0}
(\L',\c A^0_{cr ,2})$;

$\bullet $\ $\Ker\widetilde{\Theta }=\Hom _{\widetilde{\uL }_0}(\L'_p,p\c A^0_{cr ,2})$;

$\bullet $\ $\Ker (\wt\Theta\circ\c V^t(K_p))
=\Hom _{\widetilde{\uL}_0}(\L',p\c A^0_{cr ,2})$.

Therefore, 
$$\widetilde{\c V}^{ft}(\L )=\c V^t_{\L'}(\L )=
\Hom _{\widetilde{\uL}_0}(\L ',\c A^0_{cr ,2}/\c A^0_{cr ,1})=
\Hom _{\widetilde{\uL}_0}(\L ,\c A^0_{cr ,2}/\c A^0_{cr ,1}).$$
\medskip 

\subsection{The functor $\widetilde{\CV }^{ft}$}\label{S4.7}
Let $\L\in\uL ^{ft}$ and let $i^{et}:\L ^{et}
\longrightarrow\L$ 
be the maximal etale subobject of $\L $. 

\begin{definition} The functor 
$\widetilde{\CV }^{ft}:\uL ^{ft}\longrightarrow\underline{\CM }_F$
 is the functor induced by the 
correspondence  $\L\mapsto \widetilde{\CV }^{ft}(\L )=
(\widetilde{\c V}^{ft}(\L ),
\widetilde{\c V}^{ft}(\L ^{et}),
\widetilde{\c V}^{ft}(i^{et}))$.
\end{definition}

The functor $\widetilde{\CV }^{ft}$ is not very far 
from Breuil's functor $\c V^t$ but it 
satisfies the following important property.

\begin{Prop} \label{P4.16} 
 The functor $\widetilde{\CV }^{ft}$ is fully faithful.
\end{Prop}

\begin{proof} By standard devissage it will be sufficient 
 to verify this property for the restriction 
$\widetilde{\CV }^{ft}|_{\uL ^{ft}[1]}$. Due to 
 Proposition \ref{P2.12} it will be sufficient 
to verify that the functor 
$\widetilde{\c V}^{ft}|_{\uL ^{ft}[1]}\circ\Pi ^{-1}$ coincides with the functor 
$\c V^*$ from Subsection \ref{S2.2}.
This can be proved similarly to the proof 
of the corresponding fact for unipotent objects in 
Subsection \ref{S4.2} as follows. 

 Let 
$$A^0_{st,2}=\underset{j\geqslant 0}\prod A^0_{cr ,2}
\gamma _j(\log (1+X))\subset \widetilde{A}_{st ,2}=
\underset{j\geqslant 0}\prod\widetilde{A}_{cr ,2}\gamma _j(\log (1+X))$$
with induced structures of the objects 
${\c A}^0_{st ,2}$ and $\widetilde{\c A}_{st ,2}$ of 
the category $\widetilde{\uL }$. Then from Subsection \ref{S4.6} it follows that 
$$\c V^t(\L )=\Hom _{\widetilde{\uL }}(\L ,
{\c A}^0_{st ,2}/p{\c A}^0_{st ,2}).$$

One can see easily that the correspondence 
$$[r_0\,\mathrm{mod}\,x_0^p]T_1+[r_0]+p[r_1]
\mapsto (r_0+x_0^pr_1)\,\mathrm{mod}\,x_0^p\m _R$$
induces the morphism 
${\c A}^0_{cr ,2}/p{\c A}^0_{cr ,2}
\longrightarrow \c R^0$ in the category 
$\widetilde{\uL }_0$. This morphism induces 
a unique identification of the abelian groups 
$\c V^t(\L )$ and 
$\Hom (\L ,{\c R}^0)=
\c V^*(\L )$. 
Now going to a suitable factor 
of the object ${\c A}^0_{st ,2}/p{\c A}^0_{st ,2}$ 
we obtain that this identification is compatible with the  
$\Gamma _F$-actions on both abelian groups.
\end{proof}

Now we can describe all Galois invariant lattices of semi-stable 
$\Q _p[\Gamma _F]$-modules with weights from $[0,p)$.

\begin{Cor}\label{C4.17} 
Suppose $V$ is a semi-stable representation of $\Gamma _F$ 
with weights from $[0,p)$, 
$\mathrm{dim}_{\Q _p}V=s$ and $T$ is a $\Gamma _F$-invariant 
lattice in $V$. Then there is a $p$-divisible group $\{\L ^{(n)}, i_n\}_{n\geqslant 0}$ 
of height $s$ in $\uL ^{ft}$   
such that 
$\underset{n}\varprojlim \,\widetilde{\CV }^{ft}(\L ^{(n)})=
(T,T^{et},i^{et})\in\uC  _F$.
\end{Cor}
\medskip

\section {Proof of Theorem \ref{T0.1}.}\label{S5}

As earlier, $p$ is a fixed prime number, $p\ne 2$. 
Starting Subsection \ref{S5.2} we assume $p=3$.

\subsection{} \label{S5.1} 
For all prime numbers $l$, choose embeddings of algebraic closures 
$\bar\Q \subset\bar\Q _l$ and use them to identify the inertia 
groups $I_l=\Gal (\bar\Q _l/\Q _{l,ur})$, where $\Q _{l,ur}$ 
is the maximal unramified extension of $\Q _l$,  with the appropriate  
subgroups in $\Gamma _{\Q }=\Gal (\bar\Q /\Q)$.

Introduce the category $\uM _{\Q}^t$. 
Its objects are the pairs $H_{\Q}=(H,\widetilde{H}_{st})$, where 
$H$ is a finite $\Z _p[\Gamma _{\Q}]$-module unramified outside $p$ 
and $\widetilde{H}_{st}=
(H_{st},H^0_{st},i)\in\uC _{F}^{st}$, where 
$H|_{I_p}=H_{st}$, $F=W(\bar \F_p)[1/p]$ and 
$\uC _F^{st}$ is the image of the functor $\widetilde{\CV}^{ft}$ 
from Subsection \ref{S4.7}. 
 The morphisms in 
$\uM _{\Q}^t$ are compatible morphisms of Galois modules. 
Clearly, $\uM _{\Q}^t$ is a special pre-abelian category, cf. Appendix \ref{A}.

Let $\uM _{\Q}^t[1]$ be the full suibcategory 
of killed by $p$ objects in $\uM _{\Q}^t$. Denote by 
$\c K(p)$ an algebraic extension of $\Q $ such that 
for any $H_{\Q }=(H,\widetilde{H}_{st})\in\uM _{\Q}^t[1]$, 
$\Gamma _{\c K(p)}=\Gal (\bar\Q /\c K(p))$ acts trivially on $H$. In other words, 
$\c K(p)$ can be taken as a common field-of-definition of 
points of all such $\Gamma _{\Q }$-modules $H$. 

Now assume that 
\medskip 

${\bf (C)} $\ $\c K(p)$ is totally ramified at $p$. 
\medskip

Under this assumption we have a natural identification 
$\Gal (\c K(p)/\Q)=\Gal (\c K(p)F/F)$, that is the Galois group 
of the global extension $\c K(p)/\Q $ comes as the Galois group of 
its completion over $F$. Therefore, we can 
identify   
$\uM _{\Q}^t[1]$ with the full subcategory  of $\uC _{F}^{st}$, 
consisting of $(H_{st}, H_{st}^0,i)$ such that $pH_{st}=0$ and 
all points of $H_{st}$ are defined over $\c K(p)F$. 
 In other words, the objects of 
$\uM ^t[1]$ can be described via our local results about killed by $p$ 
subquotients of semistable representations of $\Gamma _F$. 

Denote by $\uM _{\Q }^{ft}[1]$ a full subcategory in 
$\uM _{\Q }^t[1]$ which consists of 
killed by $p$ subquotients of $p$-divisible 
groups in the category $\uM _{\Q}^t$. 

Let $F'$ be the maximal tamely ramified extension of $F$ in $\c K(p)F$. 
Then $\Gal (F'/F)$ is abelian group of order prime  to $p$ 
(use that the residue field of $F$ is algebraically closed) and 
$\Gal (\c K(p)F/F')$ is a $p$-group. This gives an abelian extension 
$\c K'$ of $\Q $ in $\c K(p)$ of prime-to-$p$ degree and such that 
$\c K(p)/\c K'$ is a $p$-extension. This extension 
is unramified outside $p$ and, therefore, it coincides 
(use class field theory) with $\Q (\zeta _p)$. 
In particular, all simple objects in  
$\uM _{\Q}^t[1]$ are of the form 
$\c F(j)=(\F _p(j), 0,0)$ if $1\leqslant j<p$ and 
$\c F(0)=(\F _p(0), \F _p(0),\mathrm{id}\,)$ if $j=0$. 

Let $\uL _{\Q }^{ft}[1]$ and $\uL _{\Q }^t[1]$ be the full subcategories 
of $\uL ^t[1]$ mapped by the functor $\widetilde{\CV }^{ft}$ 
to the objects of $\uM _{\Q }^{ft}[1]$ and, resp., $\uM _{\Q }^t[1]$. 
Clearly, 
$\uL _{\Q }^{ft}[1]$ is a full subcategory in $\uL ^{ft}[1]$ and the only 
simple objects in these  categories are $\c L(r)$, where  
$r\in\{j/(p-1)\ |\ j=0,1,\dots ,p-1\}$.

Suppose 
$H^{\infty }=\{H^{(n)}_{\Q},i_n\}_{n\geqslant 0}$ is a $p$-divisible group 
in the category $\uM _{\Q}^t$. Here all  
$H_{\Q}^{(n)}=(H^{(n)},\widetilde{H}^{(n)}_{st})$ are objects of 
the category $\uM _{\Q }^t$.  
 Let $\c L\in\uL _{\Q }^{ft}[1]$ be such that 
$\widetilde{\CV }^{ft}(\c L)=\widetilde{H}^{(1)}_{st}$. 
 Note that 
the maximal etale subobject $\c L^{et}$ of $\c L$ is isomorphic to 
$\c L(0)^{n_{et}}$, 
where $n_{et}=n_{et}(\c L)\in\Z _{\geqslant 0}$, and 
$\c L/\c L^{et}$ has no simple subquotients isomorphic to 
$\c L(0)$. Similarly, the corresponding 
maximal multiplicative 
quotient $\c L^{m}$ is isomorphic to $\c L(1)^{n_m}$, where 
$n_m=n_m(\c L)\in\Z _{\geqslant 0}$, and the kernel of the canonical projection 
$\c L\longrightarrow\c L^m$ has no simple subquotients isomorphic to 
$\c L(1)$. Therefore, 
for any $\c M\in\uL _{\Q }^{ft}[1]$,  
$$\Ext _{\uL _{\Q }^{ft}[1]}(\c L(0),\c M)=
\Ext _{\uL _{\Q }^{ft}[1]}(\c M,\c L(1))=0.$$ 
This implies that for any $H\in \uM _{\Q }^{ft}[1]$,  
$$\Ext _{\uM _{\Q }^{ft}[1]}(H,\c F(0))=
\Ext _{\uM _{\Q }^{ft}[1]}(\c F(1), H)=0.$$

Therefore, by Theorem \ref{T2} of Appendix A there is an embedding of 
$p$-divisible groups 
$H^{\infty ,m}\subset H^{\infty }$, where  
$H^{(1)m}=\c F(1)^{n_m}$, and there is a projection of 
$p$-divisible groups $H^{\infty }\To H^{\infty ,et}$, where  
$H^{(1)et}=\c F(0)^{n_{et}}$. 

For similar reasons,  
$$\Ext _{\uM _{\Q }^{ft}[1]}(\c F(0),\c F(0))=
\Ext _{\uM _{\Q }^{ft}[1]}(\c F(1),\c F(1))=0$$
and by Theorem \ref{T1} of Appendix A, the corresponding $p$-divisible groups 
$H_{\Q }^{\infty ,m}$ and $H_{\Q }^{\infty ,et}$ 
are unique. Therefore they coincide with the products of trivial $p$-divisible groups 
$(\Q _p/\Z _p)(p-1)$ and, resp., $(\Q _p/\Z _p)(0)$.  

We state this result in the following form. 

\begin{Prop} \label{P5.1} 
Under assumption {\bf (C)}, for any $p$-divisible group 
$H^{\infty }$ in the category $\uM _{\Q }^t$ there is a filtration 
of $p$-divisible groups 
$$H^{\infty }\supset H_{1}^{\infty }\supset H_{0}^{\infty }$$ 
such that $H_{0}^{\infty }=(\Q _p/\Z _p)(p-1)^{n_m}$, 
$H^{\infty }/H_{1}^{\infty }=(\Q_p/\Z_p)(0)^{n_{et}}$ and 
all simple subquotients of  
$H_1^{\infty }/H_0^{\infty }$ belong to $\{\F _p(j)|1\leqslant j\leqslant p-2\}$. 
 \end{Prop}

\subsection{}\label{S5.2}
Assume that $p=3$.

\begin{Lem}\label{L5.2} 
$\c K(3)= \Q (\root 3\of 3,\zeta _9)$, where $\zeta _9$ is 
9-th primitive root of 1.
\end{Lem}

This Lemma will be proved in Subsection \ref{S5.3} below. 

In particular, $\c K(3)$ satisfies the assumption {\bf (C)}.  

\begin{Prop} \label{P5.3} If $H^{\infty }$ is a 3-divisible group in 
 $\uM _{\Q}^t$ then in its filtration from Proposition \ref{P5.1} 
the 3-divisible group $\hat H^{\infty }=H_1^{\infty }/H_0^{\infty }$ 
is a product of finitely many trivial 3-divisible groups 
$(\Q _3/\Z _3)(1)$.
\end{Prop}

\begin{proof} 
Let $\widehat \uL _{\Q}$ be the full subcategory of $\uL _{\Q}^{ft}[1]$ 
consisting of objects $\c L$ such that $\c L^m=\c L^{et}=0$. 
This category has only one simple object $\c L(1/2)$. 
Let $\widehat\uM _{\Q}$ be the full subcategory in $\uM _{\Q}^{ft}[1]$ 
consisting of the objects 
$\widetilde{\CV }^{ft}(\c L)$, where $\c L\in\widehat\uL _{\Q}$. 
Then $\widehat\uL _{\Q}$ and $\widehat\uM _{\Q}$ are antiequivalent categories 
and $\hat H^{(1)}\in\uM _{\Q}$. By Theorems \ref{T1} and \ref{T2} our Proposition 
is implied by the following result. 
\end{proof} 

\begin{Prop} \label{P5.4} 
 $\Ext _{\hat\uL _{\Q}}(\c L(1/2),\c L(1/2))=0$. 
\end{Prop}

\begin{proof} 
Consider the equivalence of the categories $\Pi :\uL ^t\To \uL ^*$ from 
Corollary \ref{C3.10}. This equivalence transforms the functor 
$\widetilde{\CV }^{ft}$ into the functor $\CV ^*$ from Section \ref{S2}, 
cf. the proof of Proposition \ref{P4.16}. 
Therefore, the objects $\L $ of the category 
$\Pi (\uL _{\Q }^t):=\uL ^*_{\Q }$ are characterised by the condition that 
all points of $\c V ^*(\L )$ are defined over the field $\c K(3)F$. 
The objects $\c L$ of the category $\Pi (\widehat {\uL }_{\Q }):=\widehat {\uL }_{\Q }^*$ 
are characterised by the additional properties:  
they are all obtained by subsequent extensions via $\c L(1/2)$ and $\c V ^*(\c L)$ 
appears as a subquotient of semi-stable representation of $\Gamma _F$ with Hodge-Tate weights 
from [0,2].

Introduce the object $\c L(1/2,1/2)=(L,F(L),\varphi ,N)$ 
of the category $\uL^*$ as follows:  

$\bullet $\ $L=\c W_1l\oplus \c W_1l_1$;

$\bullet $\ $F(L)$ is spanned by $ul_1$ and $ul+l_1$;

$\bullet $\ $\varphi (ul_1)=l_1$, $\varphi (ul+l_1)=l$;

$\bullet $\ $N(l_1)\equiv 0\,\mathrm{mod}\,u^3L$, 
$N(l)\equiv l_1\,\mathrm{mod}\,u^3L$.

Clearly, $\c L(1/2,1/2)$ has a natural structure of an element of the group 
$\Ext _{\uL ^*}(\c L(1/2),\c L(1/2))$. 

\begin{Lem} \label{L5.5} 
{\rm a)}\ $\c L(1/2,1/2)\in\uL ^*_{\Q }$;
 
{\rm b)}\ $\Ext _{\uL ^*_{\Q }}(\c L(1/2),\c L(1/2))\simeq\Z /3$ and 
is generated by the class of $\c L(1/2,1/2)$;

{\rm c)}\ $\Ext _{\uL ^*_{\Q }} (\c L(1/2),\c L(1/2,1/2))=
\Ext _{\uL ^*_{\Q}}(\c L(1/2,1/2),\c L(1/2))=0$.
\end{Lem}

This Lemma will be proved in Subsection \ref{S5.4} below. 

Lemma \ref{L5.5} implies that 
$\Ext _{\uL ^*_{\Q }}(\c L(1/2,1/2),\c L(1/2,1/2))=0$ 
and, therefore, any object 
$\c L$ of $\uL ^*_{\Q }$ is the product of several copies of 
$\c L(1/2)$ and $\c L(1/2,1/2)$. 
 
Suppose $\c L=\c L_1\times L(1/2,1/2)\in\widehat{\uL }^*_{\Q}$.  
Then there is a 3-divisible group $\widetilde{H}^{\infty }$ 
in $\uM _{\Q}^t$ such that $\widetilde{H}^{(1)}=H'\times H(1/2,1/2)$, 
where $H'$ and  $H(1/2,1/2)=\CV ^*(\c L(1/2,1/2))$ belong to $\hat\uM _{\Q}$. 
Clearly, we have $\Ext _{\uM _{\Q}^{ft}[1]}(H',H(1/2,1/2))=0$ and 
applying Theorem \ref{T2} we obtain a 3-divisible group $H^{\infty }$ 
in $\uM _{\Q}^{t}$ such that $H^{(1)}=H(1/2,1/2)$. 
This implies the existence of 2-dimensional semi-stable 
(and non-crystalline) representation 
of $\Gamma _F$ with the only simple subquotient $\F _3(1)$, 
that is for any Galois invariant lattice $T$ of such representation, 
the $\Gamma _F$-module $T/3T$  has semi-simple envelope $\F _3(1)\times\F _3(1)$. 
This situation appears as a very special case of Breuil's description  
of 2-dimensional semi-stable (and non-crystalline) 
representations. According to Theorem 6.1.1.2 of 
\cite{refBr1} the corresponding semi-simple envelope is 
either $\F _3(0)\times\F _3(1)$ or $\F_3(1)\times \F _3(2)$. 
The proposition is proved.
\end{proof}

Now our main Theorem appears as the following Corollary.

\begin{Cor} \label{C5.6} 
 If $Y$ is a projective variety with semi-stable 
reduction modulo 3 and good reduction modulo all primes $l\ne 3$ then 
$h^2(Y_{\mathbb C })=h^{1,1}(Y_{\mathbb C })$. 
\end{Cor}

\begin{proof} Indeed, let $V$ be the $\Q _3[\Gamma _F]$-module of 
 2-dimensional etale cohomology of $Y$. Then it is a semi-stable representation 
of $F$ and its $\Gamma _F$-invariant lattice determines a 
3-divisible group in the category 
$\uM _{\Q}^t$. By Proposition \ref{P5.3} this 3-divisible group can be built from 
the Tate twists $(\Q _3/\Z _3)(i)$, $i=0,1,2$. 
Equivalently, all $\Gamma _F$-equivariant 
subquotients of $V$ are $\Q _3(i)$ with $i=0,1,2$. 
Applying the Riemann Conjecture (proved by Deligne) 
to the reductions $Y\operatorname{mod}l$ 
with $l\ne 3$, we obtain that $\Q (0)$ and $\Q (2)$ 
do not appear. Therefore, 
$V$ is the product of finitely many $\Q _3(1)$ and 
$h^2(Y_{\mathbb C})=h^{1,1}(Y_{\mathbb C})$. 
\end{proof}

\subsection{Proof of Lemma \ref{L5.2}}\label{S5.3}

Use the ramification estimate 
from Subsection \ref{S2.9} to deduce that 
 the normalized discriminant of $\c K(3)$ over $\Q$ satisfies 
the inequality  
$|D(\c K(3)/\Q )|^{[\c K(3):\Q ]^{-1}}<3^{3-1/3}=18.72075$. 
Then Odlyzko estimates imply that $[\c K(3):\Q ]<230$ \cite{refDI}.

Let $K_0=\Q (\zeta _9)$ and $K_1=\Q (\root 3\of 3,\zeta _9)$. 
Then $K_0$ is the maximal abelian extension of $\Q $ in $\c K(3)$ 
and $K_1\subset\c K(3)$. 
We have also the inequality $[\c K(3):K_1]<60$ and, therefore, 
$\Gal (\c K(3)/\Q )$ is soluble. 

Prove that $K_1=\c K(3)$.

Suppose the field $K_2$ is 
the maximal abelian extension of $K_1$ in $\c K(3)$. 
One can apply the computer package SAGE to prove 
that the group of classes of $K_1$ is trivial. 
Therefore, $K_2$ is totally ramified at 3 and $\Gal (K_2/\Q )$ 
coincides with the Galois group of the corresponding 3-completions.
In particular, the maximal tamely ramified subextension 
of these completions comes from $\Q (\zeta _3)$ 
and, therefore, $K_2/K_1$ is 3-extension. 
Therefore, there is an $\eta\in O^*_{K_1}$ such that 
$K_1(\root 3\of \eta )\subset K_2$. Then a routine computation 
shows that the normalized discriminant for $K_1(\root 3\of\eta )$ 
over $\Q $ is less than $3^{3-1/3}$ if and only if 
$\eta\equiv 1\mathrm{mod}\,O_{K_1}^{*3}(1+3O_{K_1})^{\times}$. 
The Lemma will be proved if we show that
such $\eta\in O^{*3}_{K_1}$. (This is equivalent to 
the Leopoldt Conjecture for the field $K_1$.) 
This was proved via a SAGE computer program written by R.Henderson 
(Summer-2009 Project at Durham University 
supported by Nuffield Foundation). This program, cf. 
Appendix \ref{B}, constructed 
a basis $\varepsilon _i\mathrm{mod}\,O^{*3}_{K_1}$, $1\leqslant i\leqslant 9$, 
of $O_{K_1}^*/O_{K_1}^{*3}$ 
such that $18v_3(\varepsilon _i-1)$ takes values in the set 
$\{1,2,4,5,7,8,10,13,16\}$. In other words, 
$v_3(\eta -1)\geqslant 1>16/18$ implies that $\eta\in O_{K_1}^{*3}$.

Lemma \ref{L5.2} is proved. 
\medskip 

\subsection{Proof of Lemma \ref{L5.5}}\label{S5.4} 
a) Use the notation from the definition of the functor $\c V^t$ 
in Subsection \ref{S4}. 

If $f_0\in\c V^t(\c L(1/2,1/2))$ 
then the correspondence $f_0\mapsto (f_0(l_1),f_0(l))$ 
identifies $\c V^t(\c L(1/2,1/2))$ with the $\F_3$-module 
of couples $(X_{10},X_0)\in (R/x_0^6)^2$ such that 
$X_{10}^3/x_0^3=X_{10}$ and $(X_0^3+X_{10})/x_0^3=X_0$. 
Then the $\F _3[\Gamma _F]$-module $\c V^t(\c L(1/2,1/2))$ 
is identified with the module formed by the images 
of the couples $(X_{10},X_0+X_{10}Y)\in (R_{st}^0)^{2}$ 
in the module $\widetilde{R}^0_{st}=R^0_{st}/(x_0^3\m _R+x_0^2\m _RY+x_0\m _RY^2)$. 

In particular, the corresponding $\Gamma _F$-action on $\c V^t(\c L(1/2,1/2))$ 
comes from the natural $\Gamma _F$-action on the residues 
of $X_{10}$ and $X_0$ modulo $x_0^3\m _R$. Notice there is a natural 
$\Gamma _F$-equivariant identification 
$$\iota :\m _R/(x_0^3\m _R)\longrightarrow \bar \m /3\bar \m,$$
where $\bar\m $ is the maximal ideal of the valuation ring 
of $\bar\Q _3$. This isomorphism $\iota $ comes from the correspondence 
$r\mapsto r^{(1)}$, where for $r=\underset{n}\varprojlim (r_n\,\mathrm{mod}\,p)$, 
$r^{(1)}:=\underset{n\to\infty }\lim r_{n+1}^{p^n}$. 

Then Hensel's Lemma implies the existence of unique 
$Z_{10},Z_0\in \bar\m $ such that the following equalities hold 
$\iota (X_{10}\,\mathrm{mod}\,x_0^3\m _R)=Z_{10}\,\mathrm{mod}\,3\bar\m $,
$\iota (X_0\,\mathrm{mod}\,x_0^3\m _R)=Z_0\,\mathrm{mod}\,3\bar\m $, 
$Z_{10}^3+3Z_{10}=0$ and $Z_0^3+3Z_0=-Z_{10}$.  

Clearly, $F(Z_{10}, Z_0)=F(\zeta _9)$. Therefore, if $\tau\in\Gamma _F$ 
is such that $\tau (\zeta _9)=\zeta _9$ then 
$\tau (X_{10})=X_{10}$ and $\tau (X_0)=X_0$. 

Finally, it follows directly from definitions that 
if $\tau (\root 3\of 3)=\root 3\of 3$ then 
$\tau $ acts as identity on the image of $Y$ in $\widetilde{R}^0_{st}$. 
The part a) of the Lemma is proved.

b) Suppose $\c L=(L,F(L),\varphi ,N)\in\Ext _{\uL ^*_{\Q}}(\c L(1),\c L(1))$. 
Then $L=\c W_1l\oplus\c W_1l_1$, there is an $w\in\c W_1$ such that 
$F(L)$ is spanned by $ul_1$ and $ul+wl_1$ over $\c W_1$, and 
it holds $\varphi (ul_1)=l_1$, $\varphi (ul+wl_1)=l$, 
$N(l_1)\in u^3L$ and $N(l_1)\equiv w^3l_1\,\mathrm{mod}\,u^3L$. 
Notice that $\c L$ splits in $\uL ^*$ iff $w\in u\c W_1$. Therefore, 
we can assume that $w=\alpha \in k$. 

Then the field-of-definition of all points of $\c V^t(\c L)$ 
contains the field-of-definition of all solutions 
$(X_1,X)\,\mathrm{mod}\,x_0^3\m _R\in (R/x_0^3\m _R)^2$ of the 
following congruences: 
$X_1^3/x_0^3\equiv X_1\,\mathrm{mod}\,x_0^3\m _R$ and 
$(X^3+\alpha ^3X_1)/x_0^3\equiv X\,\mathrm{mod}\,x_0^3\m _R$. 

Let $x_1\in R$ be such that $x_1^2=x_0$. Then we can take  
$X_1=x_1^3$ and for $T=X/x_1^3$ one has the following 
Artin-Schreier-type congruence:
$$T^3-T\equiv -\alpha ^3/x^6\,\mathrm{mod}\,\m _R.$$

Using calculations from above part a) we can conclude that 
$\c L\in\uL ^*_{\Q}$ if and only if the field-of-definition 
of $T\,\mathrm{mod}\,\m _R$ over $k((x_1))$ 
belongs to the field-of-definition of  
$T_0\,\mathrm{mod}\,\m _R$ over $k((x_1))$, 
where $T_0^3-T_0\equiv -x_1^{-6}\,\mathrm{mod}\,\m _R$. 
By Artin-Schreier theory this happens if and only if 
$\alpha\in\F _3$ and,therefore, $\c L\simeq\c L(1/2,1/2)$. 

c) Suppose $\c L=(L,F(L),\varphi ,N)\in\Ext _{\uL ^*_{\Q}}
(\c L(1/2), \c L(1/2,1/2))$. 

Then we can assume that: 

---  $L=\c W_1l\oplus\c W_1l_1\oplus\c W_1m$;

--- $F(L)$ is spanned over $\c W_1$ by $ul_1$, $ul+l_1$ 
and $um+wl+w_1l_1$ with $w,w_1\in\c W_1$; 

--- $\varphi (ul_1)=l_1$, $\varphi (ul+l_1)=l$ and 
$\varphi (um+wl+w_1l_1)=m$. 

Then the condition $u^2m\in F(L)$ implies that $wl_1\in F(L)$, 
or $w\in u\c W_1$ and we can assume that $w=0$. Then 
the submodule $\c W_1m+\c W_1l_1$ determines a subobject 
$\c L'$ of $\c L$, $\c L'\in\uL ^*_{\Q}$ and using calculations 
from  b) we conclude that $w_1\in\F _3\,\mathrm{mod}\,u\c W_1$. 
Therefore, we can assume that $w_1=\alpha\in\F _3$ and for 
$m'=m-\alpha l$ we have $m'\in F(L)$ and $\varphi (um')=m'$, 
i.e. $\c L$ is a trivial extension.

Now suppose $\c L=(L,F(L),\varphi ,N)\in 
\Ext _{\uL^*_{\Q}}(\c L(1/2,1/2),\c L(1/2))$. 

Then we can assume that:

--- $L=\c W_1m\oplus\c W_1m_1\oplus\c W_1l$;

--- $F(L)$ is spanned over $\c W_1$ by $ul$, $um_1+wl$ 
and $um+m_1+w_1l$ with $w,w_1\in\c W_1$; 

--- $\varphi (ul)=l$, $\varphi (um_1+wl)=m_1$ and 
$\varphi (um+m_1+w_1l)=m$. 

Again the condition $u^2m\in F(L)$ implies that $w\in u\c W_1$ 
and,therefore, we can assume that $w=0$. Then 
the quotient module $L/\c W_1m_1$ is the quotient of $\c L$ in 
the category $\uL ^*$. This quotient must belong to the subcategory 
$\uL ^*_{\Q}$. This implies that $w_1\in\F _3\,\mathrm{mod}\,u\c W_1$, 
and, as earlier, $\c L$ becomes a trivial extension. 

The Lemma is completely proved.

\appendix 

\newpage

\section{$p$-divisible groups in pre-abelian categories}\label{A}

\subsection{Short exact sequences in pre-abelian categories}
\label{AS1}

\subsubsection{Pre-abelian categories} \label{AS1.1}

Introduce the concept of a special pre-abelian category following mainly 
\cite{A4}, cf. also \cite{A1,A2,A5}.
Remind that $\c S$ is a pre-abelian category if $\c S$ is an additive
category and for any its morphism $u\in\Hom _{\c S}(A,B)$, 
there exist $\Ker\,u=(A_1,i)$ and $\Coker\,u=(B_1,j)$, 
where $i\in\Hom _{\c S}(A_1,A)$ and $j\in\Hom _{\c S}(B,B_1)$. 
For any objects $A,B\in\c S$, let $A\prod B$ and $A\coprod B$
be their product and coproduct, respectively. 
There is 
a canonical isomorphism $A\prod B\simeq A\coprod B$ in $\ S$. 
More generally, for
given morphisms 

$\bullet $\  $\alpha \in\Hom _{\ S}(C,A)$, $\beta\in\Hom_{\c S}(C,B)$, 
there is a fibered coproduct 
$(A\coprod _CB,i_A,i_B)$, with $i_A\in\Hom _{\ S}(A,A\coprod _CB)$,
$i_B\in\Hom _{\c S}(B,A\coprod _CB)$ which completes  
the diagram 
$A\overset{\alpha }\longleftarrow C\overset{\beta }\longrightarrow B$ 
to a cocartesian square; 

$\bullet $\ $f\in\Hom _{\c S}(A,C)$ and $g\in\Hom _{\c S}(B,C)$, 
there is a fibered product  $(A\prod _CB, p _A,p_B)$, 
with $p _A\in\Hom _{\c S}(A\prod _CB,A)$, 
$p _B\in\Hom _{\c S}(A\prod _CB,B)$, which
completes the diagram $A\overset{f}\longrightarrow C\overset{g}\longleftarrow B$
to a cartesian square.

Suppose 
$i\in\Hom _{\c S}(A_1,A)$, $f\in\Hom _{\c S}(A_1,B)$ 
and $(B\coprod _{A_1}A, i_A,i_B)$ is their
fibered coproduct. If $(A_2,j)=\Coker\,i$ then there
is a morphism $j_B:B\coprod _{A_1}A\rightarrow C$ such that the
following diagram
$$\xymatrix{A_1\ar[d]^f\ar[rr]^i&&A\ar[d]^{i_A}\ar[rr]^j&&A_2\ar[d]^{id}\\
B\ar[rr]^{i_B}&&B\coprod_{A_1}A\ar[rr]^{j_B}&&A_2}
$$
is commutative (use the zero morphism from $B$ to $A_2$). 
A formal verification shows that 
$(A_2,j_B)=\Coker\,i_B$.

Suppose $j\in\Hom _{\c S}(A,A_2)$, $g\in\Hom _{\c S}(B,A_2)$
 and $(B\prod _{A_2}A, p _B,p _A)$ is
their fibered product. If $(A_1,i)=\Ker\,j$ then 
there is an $i_B:A_1\rightarrow B\prod_{A_2}A$ 
(use the zero map from $A_1$ to $B$) such that
the following diagram 
$$
\xymatrix{A_1\ar[rr]^{i}&& A\ar[rr]^{j} && A_2 \\ 
A_1\ar[rr]^{i_B}\ar[u]^{id} && B\prod _{A_2}A 
\ar[rr]^{p _B}\ar[u]^{p _A}&&B\ar[u]^g}
$$
is commutative and $(A_1,i_B)=\Ker\,p _B$. 

\subsubsection{Strict morphisms} \label{AS1.2}

A morphism $u\in\Hom_{\c S}(A,B)$ is strict if the canonical 
morphism $\mathrm{Coim}u=\Coker(\Ker\,u)\rightarrow\mathrm{Im}u
=\Ker(\Coker\,u)$ is isomorphism. One can verify that always 
$\Ker\,u=(A_1,i)$ is a strict monomorphism and 
$\Coker\,u=(B_1,j)$ is a strict epimorphism. 
By definition, a sequence of objects and morphisms 
\begin{equation}\label{AE1.1} 
0\longrightarrow A_1\overset{i}\longrightarrow A
\overset{j}\longrightarrow A_2\longrightarrow 0
\end{equation}
in $\c S$ is short exact if $(A_1,i)=\Ker\,j$ and
$(A_2,j)=\Coker\,i$. In particular, any strict monomorphism 
(resp. strict epimorphism) can be
included in a short exact sequence. 

\begin{definition} A pre-abelian category is special if it 
satisfies the following two axioms:

{\bf SP1.} if $\alpha :C\rightarrow A$ is strict monomorphism then 
$i_B:B\rightarrow A\underset{C}\coprod B$ is also strict monomorphism;

{\bf SP2.} if $f:A\rightarrow C$ is strict epimorphism then 
$p_B:A\underset{C}\prod B\rightarrow B$ is also strict epimorphism.
\end{definition}
A typical
example of pre-abelian special category is 
the category of modules with filtration.

Consider short exact sequence \eqref{AE1.1} in $\c S$. If
$f\in\Hom _{\c S}(A_1,B)$ then we have the following
commutative diagram
$$\xymatrix{0 \ar[r]& A_1\ar[d]^f\ar[rr]^{i}&&
A\ar[d]^{i_A}\ar[rr]^{j} && A_2 \ar[d]^{id}\ar[r]& 0\\ 
0\ar[r]& B\ar[rr]^{i_B} && A\coprod _{A_1}B\ar[rr]^{j_B} && A_2\ar[r]&0}
$$
Then $j_B=\Coker\,i_B$ is strict
epimorphism and by axiom {\bf SP1}, $i_B$ is  strict monomorphism. Then
$\Ker\,j_B=
\Ker(\Coker\,i_B)=\mathrm{Im}i_B=(B,i_B)$ and,
therefore, the lower row of the above diagram is exact. 

Dually, for any $g\in\Hom _{\c S}(B,A_2)$ there is a
commutative diagram
 $$\xymatrix{0\ar[r]&A_1\ar[rr]^i&&A\ar[rr]^j&&A_2\ar[r]&0\\
0\ar[r]&A_1\ar[rr]^{i_B}\ar[u]^{id}&&A\prod _{A_1}B\ar[rr]^{p
  _B}\ar[u]^{p _A}&&B\ar[u]^g\ar[r]&0}
$$
where $i_B=\Ker\, j_B$ is strict monomorphism, by Axiom {\bf SP2},  
$p _B$ is  strict epimorphism and the lower row of this diagram is
exact. 
\subsubsection{Bifunctor $\Ext _{\c S}$} \label{AS1.3}
Notice that in special pre-abelian categories, 
the composition of strict monomorphisms (resp., epimorphisms) 
is again strict and 
in the following commutative diagram with exact rows 
$$\xymatrix{0\ar[r]&A_1\ar[rr]\ar[d]^{id}&&A\ar[rr]\ar[d]^f&&A_2\ar[r]\ar[d]^{id}&0\\
0\ar[r]&A_1\ar[rr]&&A'\ar[rr]&&A_2\ar[r]&0}
$$
the morphism $f$ is isomorphism. 
Therefore, 
one can introduce the set of equivalence classes of
short exact sequences $\Ext _{\c S}(A_2,A_1)$. 
This set is functorial in both arguments due to axioms 
{\bf SP1} and {\bf SP2}. 

Suppose the objects of $\c S$ are provided with commutative 
group structure respected by morphisms of $\c S$. 
Then for any $A,B\in\c S$, $\Ext _{\c S}(A,B)$ has a 
natural group structure, where the class of split short exact
sequences plays a role of neutral element. Remind that the sum
$\varepsilon _1+\varepsilon _2$ of two extensions 
$\varepsilon _1:0\longrightarrow A_1\overset{i'}\longrightarrow A'
\overset{j'}\longrightarrow A_2\longrightarrow 0$ and 
$\varepsilon _2:0\longrightarrow A_1\overset{i^{\prime\prime }}\longrightarrow 
A^{\prime\prime}\overset{j^{\prime\prime}}\longrightarrow 
A_2\longrightarrow 0$ 
is the lower line of the following commutative diagram relating the rows 
$l=\varepsilon _1\oplus\varepsilon _2$, 
$\nabla ^*(l)$ and 
$(+)_*\nabla ^*(l)$,
$$\xymatrix{l:0\ar[r] &A_1\prod
  A_1\ar[r]
^{i'\prod i^{\prime\prime }}&A'\prod A^{\prime\prime
}\ar[r]^{j'\prod j^{\prime\prime }}&A_2\prod A_2\ar[r]&0\\
\nabla ^*(l):0\ar[r] &A_1\prod
  A_1\ar[d]^+\ar[u]^{id}\ar[r]^{i'\prod i^{\prime\prime }}&
A'\prod _{A_2}A^{\prime\prime}\ar[d]\ar[u]\ar[r]^{j'\prod j^{\prime\prime}}&
A_2\ar[d]^{id}\ar[u]^{\nabla}\ar[r]&0\\
(+)_*\nabla ^*(l):0\ar[r] &A_1\ar[r]
&A^{\prime\prime}\ar[r]&A_2\ar[r]&0}
$$
Here $\nabla $ is the diagonal morphism, $+$ is the morphism of the group
dstructure on $\c S$. For any $f\in\Hom _{\c S}(A_1,B)$ and 
$g\in\Hom _{\c S}(B,A_2)$ the corresponding morphisms 
$f_*:\Ext _{\c S}(A_2,A_1)\rightarrow
\Ext _{\c S}(A_2,B)$ and 
$g^*:\Ext _{\c S}(A_2,A_1)\rightarrow
\Ext _{\c S}(B,A_1)$ are homomorphisms of abelian groups. 
The proof is completely formal and goes along the lines of 
\cite{A3}. 

Suppose $\varepsilon\in\Ext _{\c S}(A_2,A_1)$, then 
the extension $\varepsilon +(-\mathrm{id})^*\varepsilon $ splits. We shall need
below the following explicit description of this splitting. 

Let 
$\varepsilon : 0\longrightarrow A_1\overset{i}\longrightarrow 
A\overset{j}\longrightarrow A_2\longrightarrow 0$.  
Then $\varepsilon +(-\mathrm{id})^*\varepsilon $ is the lower row in the
following diagram  
$$\xymatrix{0\ar[r] &A_1\prod A_1\ar[d]\ar[r]^{i\prod
    i}&A\prod_{A_2}A\ar[d]\ar[r]^{(j,j)}&A_2\ar[d]^{id}\ar[r]&0\\
0\ar[r]&A_1\ar[r]&A_0\ar[r]&A_2\ar[r]&0}
$$ 
where the left vertical arrow is the cokernel  of the
diagonal embedding $\nabla :A_1\rightarrow A_1\prod A_1$. One
can see that the epimorphic map $A_0\rightarrow A_1$, which
splits the lower exact sequence, is induced by the morphism 
$p _1-p _2:A\prod _{A_2}A\rightarrow A$.

Finally, one can apply Serre's arguments \cite{A6} 
to obtain for any short exact sequence 
$0\longrightarrow A_1\overset{i}\longrightarrow A
\overset{j}\longrightarrow A_2\longrightarrow 0$ 
and any $B\in\c S$, the following standard 6-terms exact sequences of
abelian groups  
$$\begin{array}{rrll}
0\longrightarrow \Hom _{\c S}(B,A_1)&
\overset{i_*}\longrightarrow\Hom_{\c
    S}(B,A)&\overset{j_*}\longrightarrow\Hom_{\c S}(B,A_2)&\\
\overset{\delta }\longrightarrow &\Ext_{\c S}(B,A_1)
\overset{i_*}\longrightarrow &\Ext _{\c S}(B,A)\overset{j_*}
\longrightarrow \Ext _{\c S}(B,A_2)
\end{array}
$$ 
$$
\begin{array}{rrll}
0\longrightarrow\Hom _{\c S}(A_2,B)& 
\overset{j^*}\longrightarrow\Hom_{\c 
S}(A,B)&\overset{i^*}\longrightarrow\Hom_{\c S}(A_1,B)&\\
\overset{\delta }\longrightarrow &\Ext_{\c S}(A_2,B)
\overset{i^*}\longrightarrow &\Ext _{\c S}(A,B)\overset{j^*}\longrightarrow 
\Ext _{\c S}(A_1,B)
\end{array}
$$ 

\subsection{$p$-divisible groups}\label{AS2}

In this section $\c S$ is special pre-abelian category 
consisting of group objects. Denote by 
$\c S_1$ the full subcategory of 
objects killed by $p$ in $\c S$, where $p$ is a
fixed prime number. Clearly, $\c S_1$ is again special pre-abelian category. 

\subsubsection{Basic definitions} \label{AS2.1}

Consider an inductive system $(C^{(n)},i^{(n)})_{n\geqslant 0}$ of
objects of $\c S$, where $C^{(0)}=0$ and $i^{(n)}:C^{(n)}\rightarrow
C^{(n+1)}$ are strict monomorphisms for all $n\geqslant 0$. Let for all
$n\geqslant m\geqslant 0$, $i_{mn}=i^{(n-1)}\circ \ldots \circ i^{(m+1)}\circ 
i^{(m)}\in\Hom _{\c L}(C^{(m)},C^{(n)})$. Then all
$i_{mn}$ are strict monomorphisms. 
Follow Tate's paper [Ta] to define 
a $p$-divisible group in $\c S$ as an inductive system
$(C^{(n)},i^{(n)})_{n\geqslant 0}$ in 
$\c S$ such that for all $0\leqslant m\leqslant n$,  

a) $\Coker\, i_{mn}=(C^{(n-m)},j_{n,n-m})$, i.e. there are 
short exact sequences: 
$$0\longrightarrow C^{(m)}\overset{i_{mn}}\longrightarrow C^{(n)}
\overset{j_{n,n-m}}\longrightarrow C^{(n-m)}\longrightarrow 0
$$

b) there are commutative diagrams 
$$\xymatrix{C^{(n)}\ar[rd]_{j_{n,n-m}}\ar[rr]^{p^m\mathrm{id}_{C^{(n)}}}&&C^{(n)}\\
&C^{(n-m)}\ar[ru]_{i_{n-m,n}}&}
$$ 
\medskip 

The above definition implies the existence of the 
following commutative diagrams with exact rows (where $m\leqslant n\leqslant n_1$):

$$
\xymatrix{0\ar[r]&C^{(m)}\ar[d]^{id}\ar[rr]^{i_{mn}}&&C^{(n)}
\ar[d]^{i_{nn_1}}\ar[rr]^{j_{n,n-m}}&&
C^{n-m)}\ar[d]^{i_{n-m,n_1-m}}\ar[r]&0\\
0\ar[r]&C^{(m)}\ar[rr]^{i_{mn_1}}&&C^{(n_1)}
\ar[rr]^{j_{n_1,n_1-m}}&&C^{(n_1-m)}\ar[r]&0}
$$

$$\xymatrix{0\ar[r]&C^{(n)}\ar[d]^{j_{nm}}
\ar[rr]^{i_{nn_1}}&&C^{(n_1)}
\ar[d]^{j_{n_1,m+n_1-n}}\ar[rr]^{j_{n_1,n_1-n}}
&&C^{(n_1-n)}\ar[d]^{id}\ar[r]&0\\
0\ar[r]&C^{(m)}\ar[rr]^{i_{m,m+n_1-n}}&&
C^{(m+n_1-n)}\ar[rr]^{j_{m+n_1-n,n_1-n}}&&C^{(n_1-n)}\ar[r]&0}
$$
Also, for all $n\geqslant m\geqslant 0$, it holds 

$\bullet $\  $(C^{(m)},i_{mn})=\mathrm{Ker}(p^mid _{C^{(n)}})$,
    $(C^{(m)},j_{nm})=\Coker\,(p^{n-m}id_{C^{(n)}})$;

$\bullet $\  $i_{mn}=i_{n-1,n}\circ \ldots \circ i_{m,m+1}$ 
and $j_{nm}=j_{m+1,m}\circ \ldots\circ 
j_{n,n-1}$.
\medskip

The set of $p$-divisible groups in $\c S$ has a natural
structure of category. This category is pre-abelian. In particular, 
$$0\longrightarrow (C^{(n)}_1,i_1^{(n)})_{n\geqslant 0}
\overset{(\gamma _n)_{n\geqslant 0}}\longrightarrow (C^{(n)},i^{(n)})_{n\geqslant 0}
\overset{(\delta _n)_{n\geqslant  0}}\longrightarrow (C_2^{(n)},i_2^{(n)})_{n\geqslant 0}
\longrightarrow 0
$$
is a short exact sequence of $p$-divisible groups iff for all $n\geqslant 1$, 
there are following commutative diagrams with short exact rows in $\c S$ 
$$\xymatrix{0\ar[r]&C_1^{(n)}\ar[d]^{i_1^{(n)}}\ar[r]^{\gamma _n}&
C^{(n)}\ar[d]^{i^{(n)}}\ar[r]^{\delta _n}&C_2^{(n)}\ar[d]^{i_2^{(n)}}\ar[r]&0\\
0\ar[r]&C_1^{(n+1)}\ar[r]^{\gamma _{n+1}}&C^{(n+1)}
\ar[r]^{\delta _{n+1}}&C_2^{n+1}\ar[r]&0}
$$

\subsubsection{A property of uniqueness of $p$-divisible groups}
\label{AS2.2}

\begin{Thm} \label{T1} Let $D$ be an object of ${\c S}_1$ such that
  $\Ext _{{\c S}_1}(D,D)=0$. If $(C^{(n)},i^{(n)})_{n\geqslant 0}$ and 
$(C_1^{(n)},i_1^{(n)})_{n\geqslant 0}$ are $p$-divisible groups in $\c S$
such that 
$C^{(1)}\simeq C_1^{(1)}\simeq D$ then these $p$-divisible groups are
isomorphic.
\end{Thm} 

\begin{proof} We must prove that for all $n\geqslant 1$, 
there are isomorphisms 
\linebreak 
  $f_n:C^{(n)}\rightarrow C_1^{(n)}$ such that
  $i_1^{(n)}\circ f_n=f_{n+1}\circ i^{(n)}$. Suppose 
$n_0\geqslant 1$ and all such isomorphisms have
  been constructed for $1\leqslant n\leqslant n_0$. 
Therefore, we can assume for simplicity that 
 $C^{(n)}=C_1^{(n)}$ for $1\leqslant n\leqslant n_0$. Consider the
  following commutative duagrams with exact rows: 
\begin{equation}\label{AE5.1}
\xymatrix{\varepsilon _{n_0+1}: & 0\ar[r]&C^{(1)}\ar[r]^{i_1}&
C^{(n_0+1)}\ar[r]^{j_1}&C^{(n_0)}\ar[r]&0\\
\varepsilon _{n_0}: & 0\ar[r]&C^{(1)}\ar[u]^{\id}\ar[r]^{i}&
C^{(n_0)}\ar[u]^{i^{(n_0)}}\ar[r]^{j}&
C^{(n_0-1)}\ar[u]^{i^{(n_0-1)}}\ar[r]&0}
\end{equation}
\begin{equation}
\label{AE5.2}
\xymatrix{\varepsilon '_{n_0+1}: & 0\ar[r]&C^{(1)}\ar[r]^{i'_1}&
C_1^{(n_0+1)}\ar[r]^{j'_1}&C^{(n_0)}\ar[r]&0\\
\varepsilon _{n_0}: & 0\ar[r]&C^{(1)}\ar[u]^{\id}\ar[r]^{i}&
C^{(n_0)}\ar[u]^{i_1^{(n_0)}}\ar[r]^{j}&C^{(n_0-1)}\ar[u]^{i^{(n_0-1)}}\ar[r]&0}
\end{equation}

Here in standard notation of Subsection \ref{AS2.1}, $i_1=i_{1,n_0+1}$,
$i_1'=i_{1,n_0+1}'$, $i=i_{1n_0}$, $j=j_{n_0,n_0-1}$, 
$j_1=j_{n_0+1,n_0}$ and 
$j'_1=j_{n_0+1,n_0}'$ (the dash means that the corresponding morphism is
related to the second $p$-divisible group). We must construct  
isomorphism $f_{n_0+1}:C^{(n_0+1)}\rightarrow C_1^{(n_0+1)}$ such that 
$f_{n_0+1}\circ i^{(n_0)}=i_1^{(n_0)}$. Consider the following commutative
diagram obtained from above two diagrams. 
 \begin{equation}
\label{AE5.3}
\xymatrix{0\rightarrow C^{(1)}\prod C^{(1)}\ar[rr]^{i_1\prod i_1'}&&
C^{(n_0+1)}\underset{C^{(n_0)}}\prod C_1^{(n_0+1)}\ar[rr]^{(j_1,j_1')}
&&C^{(n_0)}\rightarrow 0\\
0\rightarrow C^{(1)}\prod C^{(1)}\ar[u]^{\id}\ar[rr]^{i\prod i}&&
C^{(n_0)}
\underset{C^{(n_0-1)}}\prod C^{(n_0)}\ar[u]^{i^{(n_0)}\prod i^{(n_0)}_1}
\ar[rr]^{(j,j)}&&C^{(n_0-1)}
\ar[u]^{i^{(n_0-1)}}\rightarrow 0}
\end{equation}
Notice that the morphisms of multiplication by $p$ in $C^{(n_0+1)}$
and $C_1^{(n_0+1)}$ can be factored as follows 
$$\xymatrix{C^{(n_0+1)}\ar[rd]_{j_1}\ar[rr]^p&&C^{(n_0+1)}&
  C_1^{(n_0+1)}\ar[rd]_{j_1'}\ar[rr]^p&&C_1^{(n_0+1)}\\
& C^{(n_0)}\ar[ru]_{i^{(n_0)}}&&& C^{(n_0)}\ar[ru]_{i_1^{(n_0)}}&}
$$
Therefore, we obtain the following commutative diagram 
\begin{equation}
\label{AE5.4}
\xymatrix{C_1^{n_0+1)}\prod
  _{C^{(n_0)}}C^{(n_0+1)}\ar[d]^{(j_1',j_1)}\ar[rr]^p&&C_1^{(n_0+1)}\prod
    _{C^{(n_0)}}C^{(n_0+1)}\\
C^{(n_0)}\ar[rr]^{\nabla }&& C^{(n_0)}\prod _{C^{(n_0-1)}}C^{(n_0)}
\ar[u]^{i_1^{(n_0)}\prod i^{(n_0)}}}
\end{equation}
(here $\nabla $ is the diagonal morphism). Let $\alpha :C^{(1)}\prod
C^{(1)}\rightarrow C^{(1)}$ be the cokernel of the diagonal morphism 
$\nabla :C^{(1)}\rightarrow C^{(1)}\prod C^{(1)}$. Clearly, $\nabla $
and $\alpha $ are, resp., strict monomorphism and strict epimorphism. 
Set $(D_{n_0+1},\alpha _1)=\Coker\,(
(i_1\prod i'_1)\circ\nabla )$ and $(D_{n_0},\alpha _0)=\Coker\,(
  (i\prod i)\circ\nabla )$. Applying $\alpha _*$ to diagram \eqref{AE5.3} 
  obtain the two lower rows of the following diagram 
\begin{equation}
\label{AE5.5}
\xymatrix{0\ar[r]&C^{(1)}\ar[r]&D_0\ar[r]&C^{(1)}\ar[r]&0\\
0\ar[r]&C^{(1)}\ar[u]^{id}\ar[r]&D_{n_0+1}
\ar[r]\ar[u]^s&C^{(n_0)}\ar[u]^{j_{n_01}}\ar[r]&0\\
0\ar[r]&C^{(1)}\ar[u]^{id}\ar[r]&D_{n_0}
\ar[r]\ar[u]^u&C^{(n_0-1)}\ar[r]\ar[u]^{i_{n_0-1,n_0}}&0}
\end{equation}
Note that the middle line of this diagram equals $\varepsilon
_{n_0+1}-\varepsilon '_{n_0+1}\in\Ext\,(C^{(n_0)},C^{(1)})$, and 
at the third row we have a trivial extension. This implies the
existence of 
the first row of our diagram. As it was pointed out earlier, 
a splitting of the third line can be done via the morphism $f$ from
the commutative diagram 
\begin{equation}
\label{AE5.6}
\xymatrix{& C^{(n_0)}\prod
_{C^{(n_0-1)}}C^{(n_0)}\ar[rd]^{\alpha _0}\ar[ld]_{p_1-p_2} & \\
C^{(1)}&&D_{n_0}\ar[ll]_f}
\end{equation}
(Notice that the morphism $s:D_{n_0+1}\rightarrow D_0$ is the cokernel 
of the composition $\Ker f\rightarrow D_{n_0}\overset{u}\rightarrow D_{n_0+1}$. )

Above diagram \eqref{AE5.4} means that the morphism of multiplication by
$p$ on $C_1^{(n_0+1)}\prod _{C^{(n_0)}}C^{(n_0+1)}$ factors through
the diagonal embedding of $C^{(n_0)}$ into $C^{(n_0)}\prod
_{C^{(n_0-1)}}C^{(n_0)}$. 
From diagram \eqref{AE5.6} it follows that $p\,\mathrm{id} _{D_{n_0+1}}$ factors through
the embedding 
$\mathrm{\,Ker}f\rightarrow D_{n_0}\overset{u}\rightarrow  D_{n_0+1}$. Therefore, 
$pD_0=0$ i.e. the first line in diagram \eqref{AE5.5} is an element of the
trivial 
group $\Ext _{\c S_1}(C^{(1)},C^{(1)})=0$. 
So, the second row in \eqref{AE5.5} is a trivial extension,
i.e. the extensions $\varepsilon _{n_0+1}$ and $\varepsilon '_{n_0+1}$
from diagrams \eqref{AE5.1} and \eqref{AE5.2} are equivalent. This implies the existence of 
isomorphism $f_{n_0+1}$. 
\end{proof}

\subsubsection{Splitting of extensions of $p$-divisible groups}\label{AS2.3}

\begin{Thm} \label{T2} Suppose $(C^{(n)},i^{(n)})_{n\ge 0}$ is a $p$-divisible
  group in the category $\c S$ and there are $D_1,D_2\in\c S_1$
  such that $C^{(1)}\in\Ext _{\c S_1}(D_2,D_1)$ and 
$\Ext _{\c S_1}(D_1,D_2)=0$. Then there is an exact sequence
of $p$-divisible groups 
$$\xymatrix{0\ar[r]&(C_1^{(n)},i_1^{(n)})_{n\geqslant
    0}\ar[r]&(C^{(n)},i^{(n)})_{n\geqslant 
    0}\ar[r]&(C_2^{(n)},i_2^{(n)})_{n\ge 0}\ar[r]&0}
$$
in $\c S$ such that $C_1^{(1)}=D_1$ and $C_2^{(1)}=D_2$. 
\end{Thm}

\begin{proof} We have the exact sequence 
$0\longrightarrow D_1\overset{i}\longrightarrow C^{(1)}
\overset{j}\longrightarrow D_2\longrightarrow 0$. 
Set $C_1^{(1)}=D_1$ and $\gamma _1=i$.  We must show for all $n\geqslant 0$, 
the existence of objects $C^{(1)}_n$, strict monomorphisms 
$\gamma _n:C_1^{(n)}\rightarrow C^{(n)}$ and
$i_1^{(n)}:C_1^{(n)}\rightarrow C_1^{(n+1)}$ such that 
$(C_1^{(n)},i_1^{(n)})_{n\geqslant 0}$ is a $p$-divisible group and the
system $(\gamma _n)_{n\geqslant 0}$ defines an embedding of this
$p$-divisible group into the original 
$p$-divisible group $(C^{(n)},i^{(n)})_{n\geqslant 0}$. Agree to use for all $0\leqslant m\leqslant n$, the 
notation $i_{mn}$ and $j_{nm}$ from Subection \ref{AS2.1} for the original
$p$-divisible group and set $C^{(n)}=C_{n0}$. 

Illustrate the idea of proof by considering the case $n=2$. 

Consider the following commutative diagram with exact rows 
$\varepsilon _2$ and $\varepsilon _2^{(1)}=i^*\varepsilon _2$:
$$\xymatrix{\varepsilon _2:0\ar[r]&C_{10}\ar[rr]^{i_{12}}&&C_{20}
\ar[rr]^{j_{21}}&&C_{10}\ar[r]&0\\
\varepsilon _2^{(1)}:0\ar[r]&C_{10}\ar[u]^{id}\ar[rr]^{i_{12}^{(1)}}
&&C_{21}\ar[u]^{\gamma
  _2^{(1)}}
\ar[rr]^{j_{21}^{(1)}}&&C_{11}\ar[u]^i\ar[r]&0}
$$
By axiom {\bf SP1} from Subsection \ref{AS1.2}, $\gamma _2^{(1)}$ is a strict monomorphism
and the equality $p\,\mathrm{id} _{C_{20}}=i_{12}\circ j_{21}$ implies that 
  $p\,\mathrm{id} _{C_{21}}=(i_{12}^{(1)}
\circ i)\circ j_{21}^{(1)}$. Then 
the morphism 
  $j_*:\Ext _{\c S}(C_{11},C_{10})\rightarrow
  \Ext _{\c S}(C_{11},D_2)$ induces the following commutative
    diagram 
 $$\xymatrix{0\ar[r]&C_{10}\ar[d]^j\ar[rr]^{i_{12}^{(1)}}&&
C_{21}\ar[d]^f\ar[rr]^{j_{21}^{(1)}}&&C_{11}\ar[d]^{id}\ar[r]&0\\
0\ar[r]&D_2\ar[rr]&&D_{21}\ar[rr]&&C_{11}\ar[r]&0}
$$
and $(C_{11},i_{12}^{(1)}\circ i)=\Ker f$. From the above
decomposition of $p\,\mathrm{id} _{C_{21}}$ it follows that it factors through
the embedding of $\mathrm{Ker}f$, therefore, $p\id_{D_{21}}=0$, i.e. 
$D_{21}\in\Ext _{\c S_1}(C_{11},D_2)=0$. Then the exact
  sequence $\Hom _{\c S}-\Ext _{\c S}$ implies the
  existence of a commutative diagram 
 $$\xymatrix{0\ar[r]&C_{10}\ar[rr]^{i_{12}^{(1)}}&&
C_{21}\ar[rr]^{j_{21}^{(1)}}&&C_{11}\ar[r]&0\\
0\ar[r]&C_{11}\ar[u]^i\ar[rr]^{i_{12}^{(2)}}&&C_{22}\ar[rr]^{j_{21}^{(2)}}
\ar[u]^{\gamma _2^{(2)}}&&C_{11}\ar[u]^{id}\ar[r]&0}
$$
Verify that one can set $C_1^{(2)}=C_{22}$ and
$i_1^{(1)}=i^{(2)}_{12}$. Indeed, 
$$\gamma _2^{(2)}\circ p\,\mathrm{id} _{C_{22}}= p\,\mathrm{id} _{C_{21}}
\circ\gamma _2^{(2)}=
(i_{12}^{(1)}\circ i)
\circ 
(j_{21}^{(1)}\circ\gamma _2^{(2)})
=
\gamma _2^{(2)}
\circ i_{12}^{(2)}\circ 
j_{21}^{(2)}$$ 
and because $\gamma _2^{(2)}$ is monomorphism,
$p\,\mathrm{id}_{C_{22}}= i_{12}^{(2)}\circ j_{21}$. This means that we constructed a
segment of length 2 of the $p$-divisible group 
$(C_1^{(n)},i_1^{(n)})_{n\geqslant  0}$. 

Consider the general case. 
\begin{Lem} For $k\geqslant 1$ and $1\leqslant t\leqslant k$, 
in the category $\c S$ there are the following commutative diagrams
with exact lines (for second diagram $E^t_k, t\ne 1$ and for 
forth diagram $\Omega ^t_k, t\ne k$): 
$$\xymatrix{E_k^1)&0\ar[r]&C_{k-1,0}\ar[rr]^{i_{k-1,k}}&&
C_{k0}\ar[rr]^{j_{k1}}&&C_{10}\ar[r]&0\\
&0\ar[r]&C_{k-1,0}\ar[u]^{id}\ar[rr]^{i_{k-1,k}^{(1)}}&&C_{k1}\ar[rr]^{j_{k1}^{(1)}}
\ar[u]^{\gamma _k^{(1)}}&&C_{11}\ar[u]^{i}\ar[r]&0}
$$
$$\xymatrix{E_k^t)&0\ar[r]&C_{k-1,t-2}\ar[rr]^{i_{k-1,k}^{(t-1)}}&&
C_{k,t-1}\ar[rr]^{j_{k1}^{(t-1)}}&&C_{11}\ar[r]&0\\
&0\ar[r]&C_{k-1,t-1}\ar[u]^{\gamma _{k-1}^{(t-1)}}\ar[rr]^{i_{k-1,k}^{(t)}}&&C_{kt}\ar[rr]^{j_{k1}^{(t)}}
\ar[u]^{\gamma _k^{(t)}}&&C_{11}\ar[u]^{id}\ar[r]&0}
$$
$$\xymatrix{\nabla _k^t)&0\ar[r]&C_{k-1,t-1}\ar[d]_{j_{k-1,k-2}^{(t-1)}}\ar[rr]^{i_{k-1,k}^{(t)}}&&
C_{kt}\ar[d]_{j_{k,k-1}^{(t)}}\ar[rr]^{j_{k1}^{(t)}}&&C_{11}\ar[d]_{id}\ar[r]&0\\
&0\ar[r]&C_{k-2,t-1}\ar[rr]^{i_{k-2,k-1}^{(t)}}&&C_{k-1,t}\ar[rr]^{j_{k-1,1}^{(t)}}
&&C_{11}\ar[r]&0}
$$
$$
\xymatrix{\Omega
  _k^t)&0\ar[r]&C_{kt}\ar[d]_{j_{k,k-1}^{(t)}}
\ar[rr]^{\gamma _k^{(t)}}&&
C_{k,t-1}\ar[d]_{j_{k,k-1}^{(t-1)}}
\ar[rr]^{f_{k}^{(t)}}&&D_2\ar[d]_{id}\ar[r]&0\\
&0\ar[r]&C_{k-1,t}\ar[rr]^{\gamma _{k-1}^{(t)}}&&C_{k-1,t-1}\ar[rr]^{f_{k-1}^{(t)}}
&&D_2\ar[r]&0}
$$
(Here for all $k\geqslant 0$, $C_{k0}=C^{(k)}$, $i_{k,k+1}=i^{(k)}$, 
$j_{k+1,1}$ and $j_{k+1,k}^{(0)}=j_{k+1,k}$ are the morphisms from
Subsection \ref{AS2.1}, all $i_{k,k+1}^{(t)}$ and $\gamma _k^{(t)}$ are strict
monomorphisms and all $f^{(t)}_{k+1}$ and 
$j_{k+1,k}^{(t)}$ are strict epimorphisms.)
\end{Lem}
\begin{proof} Use diagram $E_1^1$ to set
  $C_{11}=D_1$, $\gamma _1^{(1)}=i$, $j_{11}=\mathrm{id} _{C_{10}}$,
  $j_{11}^{(1)}=\mathrm{id} _{C_{11}}$. 
Then for any $k\geqslant 2$, the upper row of $E_k^1$ is
the short exact sequence $\varepsilon _k\in\Ext _{\c
  S}(C_{10},C_{k-1,0})$ from the original $p$-divisible group 
$(C_{k0},i^{(k)})_{k\ge 0}$. Therefore, its lower row equals 
$i^*\varepsilon _k\in\Ext _{\c S}(D_1,C_{k-1,0})$. This
defines the objects $C_{k1}$, strict monomorphisms $i_{k-1,k}^{(1)}$,
strict epimorphisms $j_{k1}^{(1)}$ and morphisms $\gamma
_k^{(1)}$, which are strict monomorphisms (use axiom A1 and 
that $i$ is a strict monomorphism). 

For any $k\ge 2$, the relation 
$(j_{k,k-1})_*\varepsilon _k=\varepsilon _{k-1}$ 
implies the relation 
$(j_{k,k-1})_*(i^*\varepsilon
_k)=i^*\varepsilon _{k-1}$. This gives  the morphism
$j_{k,k-1}^{(1)}:C_{k1}\rightarrow C_{k-1,1}$ such that $\Delta _k^1$
commutes. Because $j_{k-1,k-2}$ is a strict
epimorphism so is the morphism $j_{k,k-1}^{(1)}$. 

The upper row of diagram $\Omega _k^1$ is obtained from the
middle column of diagram $E_k^1$ because $\mathrm{Coker}\gamma
_k^{(1)}\simeq \mathrm{Coker}i=(D_2,j)$. Similarly, 
the lower row of $\Omega _k^1$ is obtained from diagram 
$E_{k-1}^1$. The left square of $\Omega ^1_k$ is commutative by the
definition of $j_{k,k-1}^{(1)}$. The 
right square is commutative because $\Omega ^1_k$ relates 
diagrams $E_k^1$ and $E_{k-1}^1$. For $k=2$, the constructed morphism
$j_{k,k-1}^{(1)}$ clearly coincides with the morphism $j_{k,k-1}$ from diagram $E_k^1$. 

Suppose now we are given integers $k_0\geqslant 2$ 
and $t_0<k_0$ such that
diagrams $E_k^t$, $\Delta _k^t$ and $\Omega ^t_k$ have been
already constructed for all $k<k_0$ and all 
relevant $t$ and for 
$k=k_0$ and all $1\leqslant t\leqslant t_0$.

{\it Constructing $E_{k_0}^{t_0+1}$}. 
Consider the following diagram obtained by applying
$(f_{k_0-1}^{(t_0)})_*$ to the lower
row of $E_{k_0}^{t_0}$: 
$$\xymatrix{\varepsilon _{k_0}^{(t_0)}:&0\ar[r]&C_{k_0-1,t_0-1}\ar[d]_{f^{(t_0)}_{k_0-1}}
\ar[rr]^{i_{k_0-1,k_0}^{(t_0)}}&&C_{k_0,t_0}\ar[d]
\ar[rr]^{j_{k_01}^{(t_0)}}&&C_{11}\ar[d]_{id}\ar[r]&0\\
&0\ar[r]&D_2\ar[rr]&&D^*\ar[rr]&&C_{11}\ar[r]&0}
$$
Clearly, 
$\Ker \,(C_{k_0t_0}\rightarrow D^*)=(C_{k_0-1,t_0},
i^{(t_0)}_{k_0-1,k_0}\circ\gamma _{k_0-1}^{(t_0)})$. 
Consider the strict
monomorphism 
$\gamma _{k_0t_0}:=\gamma_{k_0}^{(1)}
\circ \ldots \circ \gamma _{k_0}^{(t_0)}:C_{k_0t_0}\rightarrow C_{k_00}$ 
and an analogous morphism 
$\gamma _{k_0-1,t_0-1}:C_{k_0-1,t_0-1}\rightarrow C_{k_0-1,0}$. 
Because $t_0\ne k_0$, one can obtain from diagrams $\Omega
^{t_0}_{k_0}$ and $E^t_{k_0}$ the following commutative diagram
$$\xymatrix{C_{k_0t_0}\ar[d]_{j_{k_0,k_0-1}^{(t_0)}}
\ar[rrrr]^{\gamma _{k_0t_0}}&&&&C_{k_00}\ar[d]_{j_{k_0,k_0-1}}\\
C_{k_0-1,t_0}\ar[rr]^{\gamma _{k_0-1}^{(t_0)}}&&C_{k_0-1,t_0-1}\ar[d]_{i^{(t_0)}_{k_0-1,k_0}}
\ar[rr]^{\gamma _{k_0-1,t_0-1}}&&C_{k_0-1,0}\ar[d]_{i_{k_0-1,k_0}}\\
&&C_{k_0t_0}\ar[rr]^{\gamma _{k_0t_0}}&&C_{k_00}}
$$
Then $p\mathrm{id} _{C_{k_00}}=i_{k_0-1,k_0}\circ j_{k_0,k_0-1}$ implies 
that 
$$p\mathrm{id} _{C_{k_0t_0}}=
 (i^{(t_0)}_{k_0-1,k_0}\circ\gamma _{k_0-1}^{(t_0)})
\circ j_{k_0,k_0-1}$$ 
i.e.  $p\mathrm{id} _{C_{k_0t_0}}$ factors through 
$\mathrm{Ker}(C_{k_0t_0}\rightarrow D^*)$ and, therefore, 
$p\mathrm{id} _{D^*}=0$. Then $\Ext _{\c S_1}(C_{11},D_2)=0$ implies
that 
$(f^{(t_0)}_{k_0-1})_*\varepsilon ^{(t_0)}_{k_0}=0$, 
and we obtain from the exact sequence 
$\Hom _{\c S}-\Ext _{\c S}$ that the following commutative diagram
with rows $\varepsilon _{k_0}^{(t_0)}$ and $\varepsilon _{k_0}^{(t_0+1)}$ 
can be taken as $E_{k_0}^{t_0+1}$:
$$\xymatrix{0\ar[r]&
C_{k_0-1,t_0-1}\ar[rr]^{i^{(t_0)}_{k_0-1,k_0}}&&C_{k_0t_0}
\ar[rr]^{j_{k_01}^{(t_0)}}&&C_{11}\ar[r]&0\\
0\ar[r]&C_{k_0-1,t_0}
\ar[u]_{\gamma _{k_0-1}^{(t_0)}}\ar[rr]^{i_{k_0-1,k_0}^{(t_0+1)}}&&
C_{k_0,t_0+1}\ar[u]_{\gamma _{k_0}^{(t_0+1)}}
\ar[rr]^{j_{k_01}^{(t_0+1)}}&&C_{11}\ar[u]^{id}\ar[r]&0}
$$

{\it Constructing $\Delta ^{t_0+1}_{k_0}$}. Assume that $t_0+1<k_0$. 
The above extension $\varepsilon _{k_0}^{t_0+1}$ is not
uniquely defined. Show that its choice can be done in such a way
that all diagrams $\Delta _{k_0}^{t_0+1}$ were commutative. Consider the 
short exact sequences from diagram $\Omega _{k_0-1}^{t_0}$.  
They give rise to the following exact sequences of abelian groups 
(where for $i=1,2$, $H_i:=\Hom (C_{11},D_i)$ and 
$E_i=\Ext (C_{11},D_i)$),

\begin{equation}\label{AF1}
\xymatrix 
{H_1\ar[dd]^{id}\ar[r]&\Ext(C_{11},C_{k_0-1,t_0})
\ar[dd]_{j^{(t_0)}_{k_0-1,k_0-2*}}\ar[rr]^{\gamma _{k_0-1*}^{(t_0)}}
&&\Ext(C_{11},C_{k_0-1,t_0-1})\ar[dd]^{j^{(t_0-1)}_{k_0-1,k_0-2*}}\ar[r]&
E_1\ar[dd]^{id}\\
&&&&&\\
H_2\ar[r]&\Ext(C_{11},C_{k_0-2,t_0})
\ar[rr]^{\gamma _{k_0-2*}^{(t_0)}}
&&\Ext(C_{11},C_{k_0-2,t_0-1})\ar[r]&E_2}
\end{equation}
As we saw earlier, the commutativity of $E_{k_0}^{t_0+1}$ is
equivalent  to the relation 
\begin{equation}\label{AF2}
(\gamma _{k_0-1}^{(t_0)})_*\varepsilon _{k_0}^{(t_0+1)}=\varepsilon
_{k_0}^{(t_0)}
\end{equation}
From $\Delta _{k_0}^{t_0}$ it follows that 
$\varepsilon _{k_0-1}^{(t_0)}=(j_{k_0-1,k_0-2}^{(t_0-1)})_*
\varepsilon _{k_0}^{(t_0)}$, and from $E_{k_0-1}^{t_0+1}$ it follows that 
$(\gamma _{k_0-2}^{(t_0)})_*\varepsilon _{k_0-1}^{(t_0+1)}=\varepsilon
_{k_0-1}^{(t_0)}$. Then \eqref{AF1} implies that the extension
$\varepsilon _{k_0}^{(t_0+1)}$ from relation \eqref{AF2} can be
chosen in such a way that 
$(j_{k_0-1,k_0-2}^{(t_0)})_*\varepsilon _{k_0}^{(t_0+1)}=\varepsilon
_{k_0-1}^{(t_0+1)}$, 
and this gives $\Delta _{k_0}^{t_0+1}$. 

{\it Constructing $\Omega _{k_0}^{t_0+1}$.}
The above arguments imply that the left squares of diagrams
$E_{k_0}^{t_0+1}$ and $E_{k_0-1}^{t_0}$ are related via the following
commutative diagram 
$$\xymatrix{
C_{k_0-1,t_0-1} \ar[rrrrrr]^{i_{k_0-1,k_0}^{(t_0)}}
\ar[rdd]^{j_{k_0-1,k_0-2}^{(t_0-1)}}  & & & & & & C_{k_0t_0} 
\ar[ldd]^{j_{k_0,k_0-1}^{(t_0)}} \\ 
& & & & & &   \\ 
 & C_{k_0-1,t_0-2} \ar[rrrr]^{i^{(t_0-1)}_{k_0-2,k_0-1}} & & & &   
C_{k_0-1,t_0-1}   & \\
& & & & & &  \\ 
& & & & & &  \\ 
& & & & & & \\ 
& & & & & & \\ 
 & C_{k_0-2,t_0-1}\ar[uuuuu]^{\gamma _{k_0-2}^{(t_0-1)}} \ar[rrrr]^{i_{k_0-2,k_0-1}^{(t_0)}} & & & &
C_{k_0-1,t_0} \ar[uuuuu]^{\gamma _{k_0-1}^{(t_0)}}&   \\
& & & & & & \\
C_{k_0-1,t_0} \ar[rrrrrr]^{i^{(t_0+1)}_{k_0-2,k_0-1}}\ar[ruu]^{j_{k_0-1,k_0-2}^{(t_0)}} 
\ar[uuuuuuuuu]^{\gamma _{k_0-1}^{(t_0-1)}} & & & 
& & & C_{k_0,t_0+1} \ar[uuuuuuuuu]^{\gamma _{k_0}^{(t_0+1)}}
\ar[luu]^{j_{k_0,k_0-1}^{(t_0+1)}} \\
}$$
From diagrams $\Omega _{k_0-1}^{t_0}$, $E^{t_0+1}_{k_0}$ and
  $E_{k_0-1}^{t_0}$ it follows that the induced map 
$\Coker\gamma _{k_0}^{(t_0+1)}\rightarrow\Coker\gamma
_{k_0-1}^{(t_0)}\simeq D_2$ is isomorphism. This is equivalent to the
existence of $\Omega ^{t_0+1}_{k_0}$. The lemma is proved.
\end{proof} 
For any $k\geqslant 1$, set $C_{kk}=C_1^{(k)}, i_{k-1,k}^{(k)}=i_1^{(k)}$. 
Then use diagrams $E_k^k$ to define the inductive system 
$(C_1^{(k)}, i_1^{(k)})_{k\geqslant 0}$. Denote by $\gamma _k$ the strict
monomorphism 
$\gamma _k^{(1)}
\circ\ldots \circ
\gamma _k^{(k)}
:C_1^{(k)}\rightarrow C^{(k)}$. From diagrams $E_k^t$, 
$1\le t\le k$, obtain the following commutative diagrams:
\begin{equation} \label{AD1}
\xymatrix{0\ar[r]&C^{(k-1)}\ar[rr]^{i^{(k)}}&&C^{(k)}
\ar[rr]^{j_{k1}}&&C^{(1)}\ar[r]&0\\
0\ar[r]&C_1^{(k-1)}\ar[u]^{\gamma _{k-1}}\ar[rr]^{i_1^{(k)}}&&C_1^{(k)}
\ar[u]^{\gamma _k}\ar[rr]^{j_{k1}^{(k)}}&&C_1^{(1)}\ar[u]^{\gamma _1}\ar[r]&0}
\end{equation}
It remains only to prove that the inductive system
$(C_1^{(n)},i_1^{(n)})_{n\ge 0}$ is a $p$-divisible group in $\c
S$. 
From diagrams $E_k^k$ and $\Delta _k^{k-1}$ obtain the following
commutative diagrams with exact rows 
\begin{equation}\label{AD2}
\xymatrix{0\ar[r]&C_{k-1,k-1}
\ar[dd]_{j_{k-1,k-2}^{(k-2)}\circ\gamma _{k-1}^{(k-1)}}\ar[rr]^{i_{k-1,k}^{(k)}}&&
C_{kk}\ar[dd]^{j_{k,k-1}^{(k-1)}\circ\gamma _k^{(k)}}\ar[rr]^{j_{k1}^{(k)}}&&C_{11}
\ar[dd]^{id}\ar[r]&0\\
&&&&&&\\
0\ar[r]&C_{k-2,k-2}\ar[rr]^{i_{k-2,k-1}^{(k-1)}}&&C_{k-1,k-1}\ar[rr]^{j_{k-1,1}^{(k-1)}}
&&C_{11}\ar[r]&0}
\end{equation}

If $k=3$ then the left vertical morphism of this diagram is equal to 
$j_{21}^{(1)}\circ\gamma _2^{(2)}=j_{21}^{(2)}$ and is a strict
monomorphism. By induction all
morphisms $j'_{k,k-1}:=j_{k,k-1}^{(k-1)}\circ\gamma _k^{(k)}$ are strict
epimorphisms and are included in the following commutative diagrams 
\begin{equation}\label{AD3}
\xymatrix{C^{(k)}\ar[rr]^{j_{k,k-1}}&&C^{(k-1)}\\
C_1^{(k)}\ar[u]^{\gamma _k}\ar[rr]^{j'_{k,k-1}}&&C_1^{(k-1)}\ar[u]^{\gamma _{k-1}}}
\end{equation}
For $0\leqslant m\leqslant n$, set 
$j'_{nm}=j'_{m+1,m}
\circ\ldots \circ 
j'_{n,n-1}$ and
$i'_{mn}=
i'_{n-1,n}
\circ\ldots\circ 
 i'_{m,m+1}$. Composing diagrams \eqref{AD2} 
obtain the following commutative diagram with exact rows 
$$\xymatrix{0\ar[r]&C_1^{(n-1)}\ar[dd]^{j'_{n-1,m-1}}\ar[rr]^{i'_{n-1,n}}&&
C_1^{(n)}\ar[dd]^{j'_{nm}}\ar[rr]^{j^{(n)}_{n1}}&&C_1^{(1)}\ar[dd]^{id}\ar[r]&0\\
&&&&&\\
0\ar[r]&C_1^{(m-1)}\ar[rr]^{i'_{m-1,m}}&&C_1^{(m)}
\ar[rr]^{j_{m1}^{(m)}}&&C_1^{(1)}\ar[r]&0}
$$
Thus, $i'_{n-1,n}$ induces the isomorphism
$\Ker\,j'_{n-1,m-1}\simeq \Ker\, j'_{nm}$. Therefore, 
$\Ker\,j'_{nm}=(C_1^{(n-m)},i'_{n-m,n})$ if we prove that 
\begin{equation}\label{AD4}
\Ker\, j'_{k1}=(C_1^{k-1}, i'_{k-1,k}).
\end{equation}

As we noticed earlier, $j'_{k1}=
j_{21}
\circ\ldots
\circ 
j_{k,k-1}$. Therefore, diagrams \eqref{AD3} imply that 
$j_{k1}\circ\gamma _k=\gamma _1\circ j'_{k1}$. Now diagram \eqref{AD1} implies that 
$\gamma _1\circ j_{k1}^{(k)}=\gamma _1\circ j'_{k1}$ and, therefore,
$j_{k1}^{(k)}=j'_{k1}$ because $\gamma _1$ is monomorphism. Hence
equality \eqref{AD4} folows from diagram \eqref{AD1} and 
$(C_1^{(n)}, i_1^{(n)})_{n\geqslant 0}$ satisfies 
the part a) of the definition of $p$-divisible groups. 

From diagrams \eqref{AD1} and \eqref{AD3} 
one can easily obtain for all indices $0\leqslant 
m\leqslant n$, the commutativity  of the following diagrams 
$$\xymatrix{C^{(n)}\ar[rr]^{j_{n,n-m}}&&
C^{(n-m)}\ar[rr]^{i_{n-m,n}}&&C^{(n)}\\
C_1^{(n)}\ar[u]^{\gamma _n}\ar[rr]^{j'_{n,n-m}}&&
C_1^{(n-m)}\ar[u]^{\gamma
  _{n-m}}\ar[rr]^{i'_{n-m,n}}&&C_1^{(n)}\ar[u]^{\gamma _n}}
$$
Because $\gamma _n$ is monomorphism, the equality
$i_{n-m,n}\circ j_{n,n-m}=p^m\mathrm{id} _{C^{(n)}}$ implies the equality 
$ i'_{n-m,n}\circ j'_{n,n-m}=p^m\mathrm{id} _{C_1^{(n)}}$. This gives the part b)
of the definition of $p$-divisible groups for $(C_1^{(n)},
i_1^{(n)})_{n\geqslant 0}$. The proposition is proved. 
\end{proof}

\section{SAGE program} \label{B} 

This program computes the class number of 
the field $\Q (\root 3\of 3,\zeta _9)$ and finds the basis 
$f_1,f_2,\dots ,f_9$ of the 3-subgroup of units in 
this field such that 
for the normalized 3-adic valuation $v_3$ and all $1\leqslant i\leqslant 9$, 
the natural numbers  
$a_i=18v_3(f_i\pm 1)$  
are prime to 3 and 
$1\leqslant a_1<a_2<\dots <a_9$. The result appears as the 
vector $af =(a_1,a_2,\dots ,a_9)=(1,2,4,5,7,8,10,13,16)$. 

\ \ 

\begin{verbatim}

sage: L.<b>=NumberField(x^3-3); 
sage: R.<t>=L[]
sage: M.<c>=L.extension(t^6+t^3+1);  
sage: X.<d>=M.absolute_field(); 
sage: h=X.class_number(); 
sage: e=X.units()
sage: e.append(X.zeta(9))
sage: def p(x): 
...       for i in range(1,3):
...          if valuation(norm(X(x+2*i-3)),3)!=0:
...               break 
...       return valuation(norm(X(x+2*i-3)),3)
...          
...
sage: a=[p(x) for x in e]
sage: f=[e.pop(a.index(min(a)))]
sage: while len(e)!=0:
...       a=[p(x) for x in e]
...       i0=a.index(min(a))
...       
...       for j in range(len(f)):
...           for k in range(5):
...               s=0
...               if p(f[j]^(3^k))>p(e[i0]):
...                   break
...              
...               if p(e[i0])==p(f[j]^(3^k)):
...                   s=1
...                   break
...               
...           if s==1:
...               for i in range(1,3):
...                   if p(e[i0])<p(e[i0]/(f[j]^(i*3^k))):
...                       e[i0]=e[i0]/(f[j]^(i*3^k))
...                       break
...                       
...               break
...           if j+1==len(f) and s==0:
...                     f.append(e.pop(i0))
...
sage: af=[p(x) for x in f]; 
sage: print h
sage: print af
1
[1, 2, 4, 5, 7, 8, 10, 13, 16]


\end{verbatim}

\begin{remark}
 First 4 lines introduce the field $X=\Q (\root 3\of 3,\zeta _9)$; 
its elements appear  
as polynomials in variable $d$ of degree $\leqslant 17$. Then we find the 
class number of $X$ and form the array $e=(e[1],\dots ,e[9])$ 
of minimal generators of the group $U/U^3$, where 
$U$ is the group of units in $X$. 
Next block gives a standard procedure to determine for any  
$x\in U$ the maximal natural number $p(x)$ 
 such that $x\pm 1$ is divisible precisely by $\pi ^{p(x)}$, 
where $\pi\in X$, $(\pi ^{18})=(3)$. The remaining part of the program is 
based on a standard 
technique from Linear algebra to rearrange the given system of generators 
$e$ into a new system $f$ with required properties.  As a matter of fact, we use 
that  
the class number of $X$ is prime to 3 (it equals 1) by allowing $k<5$ on 
line 22. (Any unit $x\equiv 1\operatorname{mod}\pi ^{28}$ is a cube 
in the 3-completion of $X$ by Hensel's Lemma and, therefore, is a cube in $X$.) 
The last two lines contain the values of the 
class number of $X$ and the exponents 
$(a(f[1]),\dots ,a(f[9]))$. 

\end{remark}


\begin{thebibliography}{xxx}


\bibitem{refAb1} {\sc V.\,A.\,Abrashkin},  \textit
{Modification of the Fontaine-Laffaille functor} (Russian), {Izv. Akad. Nauk SSSR 
Ser Mat}, {\bf53} (1989), 451--497; translation in {Math. USSR-Izv}, {\bf34} 
(1990), 57--97

\bibitem{refAb2} {\sc V.\,A.\,Abrashkin}, \textit{Modular representations of the 
Galois group of a local field and a generalization of a 
conjecture of Shafarevich} (Russian), {Izv. Akad. Nauk SSSR 
Ser Mat}, {\bf53} (1989), 1135--1182; translation in {Math. USSR-Izv}, {\bf35} 
(1990), 469--518

\bibitem{refAb3} {\sc V.\,Abrashkin}, 
\textit{Characteristic $p$ analogue of modules with finite crystalline height},
{Pure Appl. Math. Q.},  {\bf5} (2009), 469--494

\bibitem{refB-K} {\sc A.\,Brumer, K.\,Kramer} 
\textit{Non-existence of certain semistable 
abelian varieties}, {Manuscripta Math.}, {\bf 106} (2001), 291--304  

\bibitem{refBr1} {\sc Ch.\,Breuil}, 
\textit{Construction de repr\'esentations $p$-adiques semi-stable}, 
{Ann. Sci. \'Ecole Norm. Sup., 4 Ser.}, {\bf31} (1998), 281--327

\bibitem{refBr2} {\sc Ch.\,Breuil}, 
\textit{Repr\'esentations semi-stables et modules fortement divisibles}, 
{Invent. Math.}, {\bf136} (1999), 89--122

\bibitem{refBr3} {\sc Ch.\,Breuil},  
\textit{Integral p-adic Hodge theory}, {Advanced Studies in Pure Math.} 
{\bf 36} (2002), 51--80. 

\bibitem{refCa1} {\sc X.\, Caruso},  
\textit{Representations semi-stables de torsion dans le cas $er<p-1$}, 
{J. Reine Angew. Math.} 
{\bf 594} (2006), 35-92. 

\bibitem{refCa2} {\sc X.\,Caruso},  
\textit{$\F _p$-representations semi-stables}, to appear in {Ann. Inst. Fourier}.  

\bibitem{refDM} {\sc M.Demazure, P. Gabriel}, \textit {Groupes Alg\'ebriques},  
Elsevier, 1970 
 

\bibitem{refDI} {\sc Diaz Y Diaz},  
\textit{Tables minorant la racine $n$-i\` eme du 
discriminant d'un corps de degr\' e $n$}, {Publications Math\' ematiques d'Orsay} 
{\bf 80.06}, Universit\'e de Paris-Sud, 
D\' epartment de Math\'e matique, Bat. 425, 91405, Orsay, France

\bibitem{refFa} {\sc  G.\,Faltings}, \textit{Crystalline 
cohomology and p-adic Galois representations}. 
{In: J.I. Igusa, Editor, Algebraic analysis, 
geometry and number theory} Johns Hopkins University Press, Baltimore (1989), 25--80.

\bibitem{refFL} {\sc J.-M.\,Fontaine, G. Laffaille},  \textit{Construction de 
repr\'esentations $p$-adiques}, {Ann. Sci. \'Ecole Norm. Sup., 4 Ser.} 
{\bf15} (1982), 547--608


\bibitem{refFo1} {\sc J.-M.\,Fontaine},
\textit{Il n'y a pas de vari\'et\'e ab\'elienne sur $\Z$}, {Invent. Math.}, 
{\bf81}  (1985), 515--538. 

\bibitem{refFM} {\sc J.-M. Fontaine, W. Messing}, 
\textit{p-adic periods and p-adic étale cohomology}. 
{In: Current Trends in Arithmetical Algebraic Geometry} 
K. Ribet, Editor, Contemporary Math. {\bf67}, 
Amer. Math. Soc., Providence (1987), 179--207.

\bibitem{refFo2} {\sc J.-M. Fontaine}, 
\textit{Sch\'emas propres et lisses sur $\Z$}.  {Proceedings of the 
Indo-French Conference on Geometry (Bombay, 1989)},  43--56, Hindustan 
Book Agency, Delhi, 1993. 

\bibitem{refHat} {\sc Sh. Hattori}, 
\textit{On a ramification bound of torsion semi-stable 
representations over a local field}, 
{J. Number theory}, {\bf129} (2009), 2474-2503


\bibitem{refKis} {\sc M. Kisin}, \textit{Crystalline 
representations and $F$-crystals}.  {Algebraic geometry 
and number theory}, 459--496, Progr. Math., 
{\bf253} Birkh\"auser Boston, Boston, MA, 2006. 

\bibitem{Li1} {\sc T. Liu}, \textit{On lattices in semi-stable 
representations: a proof of a conjecture of Breuil}.  
{Comp. Math.},  {\bf144} (2008), 61--88

\bibitem{Li2} {\sc T.Liu}, \textit{A note on 
lattices in semi-stable representations} {Preprint} (to appear in Math. Annalen.)

\bibitem{refSch} {\sc R.Schoof}, \textit{Abelian varieties over $\Q$ 
with bad reduction in one prime only}.   {Comp. Math.},  {\bf141} (2005), 847--868.

\bibitem{refSer} {\sc J.-P.Serre}, \textit {Local Fields} 
Berlin, New York: Springer-Verlag, 1980 

\bibitem{refTsu} {\sc T.Tsuji}, 
\textit{p-adic étale cohomology and crystalline 
cohomology in the semi-stable reduction case}.  
{Inv. Math.}, {\bf137} (1999), 233--411,

\bibitem{refWtb} {\sc J.-P. Wintenberger}, 
\textit{Le corps des normes de certaines extensions
infinies des corps locaux; application} 
{Ann. Sci. Ec. Norm. Super.,
IV. Ser}, {\bf16} (1983), 59--89



\bibitem{A1} {\sc N.V.Glotko, V.I.Kuzminov},
\textit{On the cohomology sequence in a semiabelian category}. 
{Siberian Math J.}, {\bf43} (2002), 28-35

\bibitem{A2} {\sc Ya.A.Kopylov, V.I.Kuzminov},
\textit {On the Ker-Coker-Sequence in a 
semiabelian category}. {Siberian Math.J.},  {\bf 41} (2000), 509-517

\bibitem{A3} {\sc S.Maclane},
\textit{Homology}.  Die Grundlehren der mathematischen wissenschaften, Band 114, 
Springer-Verlag, Berlin-Goettingen-Heidelberg, 1963



\bibitem{A4} {\sc F.Richman, E.A.Walker},
\textit{Ext in pre-Abelian categories}. 
{Pacific J.Math.}, {\bf71} (1977), 521-535

\bibitem{A5} {\sc J.-P Schneiders}, \textit{Quasi-Abelian Categories and Sheaves}. 
{M\'emoires Soci\'et\'e. Math. 
Fr. (Nouvelle S\'erie)}, {\bf76} (1999)  

\bibitem{A6} {\sc J.-P.Serre}, \textit{Groupes Alg\'ebriques et corps de classes}. 
Cours de College 
de France, Hermann, 1959


\end{thebibliography}
\end{document}